\documentclass[12pt,reqno]{amsart}
\usepackage{amsmath}
\usepackage{amssymb}
\usepackage{amsthm}
\usepackage{latexsym,array,tensor,verbatim}%,hyperref}

\usepackage{enumitem}

\usepackage[utf8]{inputenc}
\usepackage{mathrsfs}%% required to write \mathscr, e.g. Fourier transf \mathscr{F}
\usepackage[mathscr]{eucal}
\usepackage{eucal} %(with mathscr option; Euler script alphabet)
\usepackage[pagebackref]{hyperref} 

\theoremstyle{plain}
\newtheorem{theorem}{Theorem}[section]
\newtheorem{lemma}[theorem]{Lemma}
\newtheorem{proposition}[theorem]{Proposition}
\newtheorem{prop}[theorem]{Proposition}
\newtheorem{corollary}[theorem]{Corollary}

\newtheorem*{assumptionAV}{Assumption $(A,V)$}

\theoremstyle{remark}%definition}
\newtheorem{remark}[theorem]{Remark}

\newtheorem{SP}{Spectral Property}

\allowdisplaybreaks 

\usepackage{graphicx}
\usepackage{epsfig,pstricks}
\usepackage{psfrag,color,float,caption,subcaption}
\newcommand{\norm}[1]{ \left\|  #1\right\| }
\newcommand{\cor}[1]{{\color{red}#1}}
\newcommand{\cob}[1]{{\color{blue}#1}}

 \newcommand{\be}{\begin{equation}}
 \newcommand{\ee}{\end{equation}}
 \newcommand{\bd}{\begin{description}}
 \newcommand{\ed}{\end{description}}

\newcommand{\bea}{\begin{align*}}
\newcommand{\ena}{\end{align*}}
\newcommand{\bec}{\begin{cases}}
\newcommand{\enc}{\begin{cases}}

\definecolor{orange}{rgb}{0.995, 0.75, 0.35}
\definecolor{purple}{rgb}{0.7, 0.2, 0.5}
\definecolor{royalblue}{rgb}{0.2, 0.7, 0.8}
\definecolor{darkgreen}{rgb}{0.2,0.725,0.25}

\hypersetup{colorlinks,
linkcolor=royalblue,
filecolor=darkgreen,
urlcolor=royalblue,
citecolor=royalblue
}

\def\al{\alpha}
\def\de{\delta}
\def\eps{\epsilon}
\def\ga{\gamma}
\def\lam{\lambda}
\def\om{\omega}
\def\veps{\varepsilon}
\def\vphi{\varphi}
\def\De{\Delta}
\def\Ga{\Gamma}

\def\Om{\Omega}

\DeclareMathOperator{\dive}{div}%\operatornorm{}

\def\sh{\sinh}
\def\ch{\cosh}
\def\th{\tanh}

\def\curl{\mathrm{curl}}
\def\iy{\infty}
\def\inv{^{-1}}
\def\pa{\partial}

\def\tr{\mathrm{tr}}

\def\cL{\mathcal{L}}
\def\cT{\mathcal{T}}
\def\sH{\mathscr{H}}
\def\cO{\mathcal{O}}  
\def\cR{\mathcal{R}} 
\def\sgn{\textrm{sgn}}
\def\To{\Rightarrow}

\newcommand{\mf}{\mathfrak}

\newcommand{\la}{\langle}
\newcommand{\ra}{\rangle}
\newcommand{\nd}{\noindent}

\newcommand{\td}{\tilde}
\newcommand{\wtd}{\widetilde}

\newcommand{\crr}{{\color{red}$\dagger$}}

\newcommand{\rk}{{\bf Remark.}\ \ }

\newcommand{\mk}{\mathfrak}

\newcommand{\Z}{\mathbb{Z}}
\newcommand{\R}{\mathbb{R}}
\newcommand{\C}{\mathbb{C}}
\newcommand{\N}{\mathbb{N}}

\begin{document} 
\title[RNLS with a repulsive potential] %with repulsive potential
{Blowup rate for Rotational NLS with a repulsive potential}
%{lens-type transform for repulsiV}

\author{Yi Hu}
\address{(Yi Hu)
Department of Mathematical Sciences,
Georgia Southern University,
Statesboro,
GA 30460,
US}
\email{yihu@georgiasouthern.edu}

\author{Yongki Lee}
\address{(Yongki Lee)
Department of Mathematical Sciences,
Georgia Southern University,
Statesboro,
GA 30460}
\email{yongkilee@georgiasouthern.edu}

\author{Shijun Zheng}
\address{(Shijun Zheng)
Department of Mathematical Sciences,
Georgia Southern University,
Statesboro,
GA 30460}
\email{szheng@georgiasouthern.edu}

\maketitle

\begin{abstract} In this paper we give an analytical proof
of the ``$\log$-$\log$'' blowup rate for mass-critical nonlinear Schr\"odinger equation (NLS) with a rotation ($\Om\ne 0$) and a repulsive harmonic potential
$V_\ga(x)=\sgn(\ga)\ga^2|x|^2$, $\ga<0$ when the initial data has a mass slightly above that of $Q$,
the ground state solution to the free NLS. 
The proof is based on a virial identity and an $\mathcal{R}_{\gamma}$-transform,
a pseudo-conformal transform in this setting. 
Further, 
we  obtain a limiting behavior description concerning the mass concentration near blowup time. 
%a nonexistence  result 
A remarkable finding is that increasing the value $|\gamma|$ for the repulsive potential 
$V_\ga$ can give rise to global in time solution for the focusing RNLS,  
which is in contrast to the case where $\ga$ is positive. 
This kind of phenomenon was earlier observed  in the non-rotational case $\Om=0$ in Carles' work. %\cite{Car03re}. 
%\edz{thus recovering a similar result in the magnetic potential free case \cite{Car03re}}
In addition, we provide numerical simulations
to partially illustrate the blowup profile along with the blowup rate using dynamic rescaling and adaptive mesh refinement method.
% Our approach is motivated by some recent work in  Merle, Raphael, Fibich \cite{FibMeRa06}
% as well as Svetlana, Yang and Zhao \cite{YangRouZh18, YangRouZh19b} for standard NLS
% Similar results for the  mass-critical   RNLS have been considered recently in   Dinh20-22 BHHZ 
\end{abstract}

\tableofcontents

\section{Introduction}

The nonlinear Schr\"odinger equation (NLS) is a standard model representative of the largely extensive class of dispersive equations. 
They are called \textit{dispersive} owing to the property that given any initial data of different frequencies there corresponds to 
solutions or waves propagating in all different velocities whose waveguide  relies on the media.
%The nonlinear term denotes the potential energy from self-interactions between %(bosonic or femion) 
%particles %or quasi-particles(atoms, electrons, magnon, protons, ions, or excitons) 
%on atomic level/ sub-atomic 
The following NLS of rotation type describe the mean field dynamics of 
Bose-Einstein condensation (BEC) with integer spin 
\begin{align}
i \hbar \pa_t u= -\frac{\hbar^2}{2m} \De u+ V(x)u+ \mu |u|^{p-1} u+ i \Om (x_1\pa_2-x_2\pa_1)u \, , \label{e:nls-VOm}
\end{align}
where $\hbar$ is the reduced Planck constant, $m$ is the mass of the bosonic particle, 
 $u$ is the wave function, $V$ is an external electric potential, and 
$L_\Om := i \Om (x_1\pa_2-x_2\pa_1)$, $\Om\in\R$ 
denotes the angular momentum operator in $\R^n$, $n=2,3$. 
In \cite{BHZ19a}, equation (\ref{e:nls-VOm}) is generalized to the setting of $\R^n$ for all dimensions
in the presence of a trapping potential $V(x):=\gamma^2|x|^2$
and $L_\Om$ is substituted with $L_A:= i A \cdot \nabla $ for 
certain divergence free magnetic field $A$. 
It is proven in \cite{BHZ19a} that in the mass-critical case $p=1+\frac{4}{n}$,
if the initial data is slightly above the ground state $Q$,
then a $\log$-$\log$ blowup rate as \eqref{phi(t):loglog} holds on an open submanifold in $H^1$
as $t\to T_{\max}$\,.  % spectral property for H_{\Om,V}

In atomic condensation,  the coefficient $\mu$ in (\ref{e:nls-VOm}) 
measures the strength of interaction, which depends on a quantity called the $s$-scattering length
that is often tunable 
using magnetic Feshbach resonances.  
It has positive sign  (defocusing) for $\tensor*[^{87}]{Rb}{}$, 
$\tensor*[^{23}]{Na}{}$,  $\tensor*[^1]{H}{}$ atoms, but 
%some elements used in recent experiments 
negative sign (focusing) for $\tensor*[^7]{Li}{}$,
%\sideset{7}{Li}, %\href{http://atomlaser.anu.edu.au/research/}{Rubidium 85} atoms
$\tensor*[^{85}]{Rb}{}$, ${}^{39} K$, $\tensor*[^{133}]{Cs}{}$ %possess a negative $s$-scattering length in  ground state 
\cite{%WTs98,Car02c, %Mo2001K^41,
Sa14potassium,LeoZheng22n}. 
%which can be tuned using magnetic Feshbach resonances
%stable alkali species, 87Rb, 23Na, {}^7Li, 85Rb  {}^{41}K, have been condensed along with hydrogen and metastable {}^4 He

In general, one has in (\ref{e:nls-VOm}) the quadratic function 
$V(x)=V_{\td{\ga}}(x):=\sum_{j=1}^n \sgn(\ga_j)\ga_j^2 x_j^2$,
which represents an external  %trapping 
potential imposed by a system of laser beams, %  trapped particles are bosons  femions ( anyons, ions)
 where $\td{\ga}:=(\gamma_j)\in \R^n$ are  the frequencies of the oscillator. %in three directions.
When $\ga_j>0$, the corresponding term works as  %$\td{V}=V$ (see \eqref{e:tdV-AV}) 
 an anisotropic confinement along the $x_j$ axial direction. %contractive, or attractive  in the $x_j$- direction. 
  %traveling along a waveguide with varying width or excitations when a BEC is released from a trap.
When $\ga_j<0$, %for some j 
it plays a repulsive role that can expand the condensate \cite{Car03re}.   
 Note that such potentials are real version  of  the complex potentials $V(x)=\sum z_j^2 x_j^2$, $z=(z_j)\in \C^n$  that generally involve a damping effect. 
 % darwich ... inhomo damping 
Equation (\ref{e:nls-VOm}) is a special form of the magnetic NLS (\ref{mNLS_AV}) that arises in a physics context such as plasma, BEC spinor, superfluids, superconductor, 
quantum vorticity \cite{Aftalion2009,AHS78a,BaoWaMar05,HelMo2001,LeNaRou18}. 

In this paper,
as a continuation of \cite{BHZ19a} with an underlying theme to study the $\log$-$\log$ law for general potentials, 
we consider the following rotational nonlinear Schr\"odinger equation (RNLS) on $\mathbb{R}^{1+n}$ in the presence of a repulsive harmonic potential
 $V(x)=V_{\ga}(x):=-\gamma^2|x|^2$ with $\ga<0$, $\mu=-1$ and $p=1+\frac4n$
	\begin{align}\label{eq:nls_va}
	\begin{cases}
	iu_t=-\Delta u+V_\ga u-|u|^{p-1}u+i A \cdot\nabla u \\
	u(0,x)=u_0 \in\mathscr{H}^1,
	\end{cases}
	\end{align}
where $u=u(t,x):\mathbb{R}\times\mathbb{R}^n \rightarrow \mathbb{C}$,
 $A=Mx$ with $M$ being an $n\times n$ real-valued skew-symmetric matrix,
i.e. $M^T=-M$, see examples of $M$ in Section \ref{s:prelimin}.
The angular momentum operator $L_A u= iA\cdot \nabla u$ generates the rotation as $e^{tA\cdot \nabla} u(x)=u\left( e^{tM} x \right)$, 
where $e^{tM} \in SO(n)$. 
The space $\mathscr{H}^1=:\Sigma$ denotes the weighted Sobolev space $\mathscr{H}^{1,2}$, 
for which notation
	\begin{align*}
	\mathscr{H}^{1,r}(\mathbb{R}^n)
	:= \{ f\in L^r: \nabla f,\, xf \in L^r \} \qquad
	\text{for \ }
	1<r<\infty
	\end{align*}
with the norm
	\begin{align*}
	\norm{f}_{\mathscr{H}^{1,r}}
	:=\norm{f}_{L^r}
	+\norm{\nabla f}_{L^r}
	+\norm{xf}_{L^r}.
	\end{align*}

Recall that the standard NLS reads
	\begin{align}\label{eq:nls}
	\begin{cases}
	i\varphi_t=-\Delta \varphi-|\varphi|^{p-1} \varphi, \\
	\varphi(0,x)=\vphi_0\in H^1. % u_0\in \Sigma
	\end{cases}
	\end{align}
Here $H^s$ denotes the usual Sobolev space $H^{s,2}:=\{f\in L^2: |\nabla|^s f\in L^2\}$.
The well-posedness of equation \eqref{eq:nls} for $1<p\leq 1+\frac{4}{n-2}$ in the euclidean space is well studied.
Let  $p=1+\frac{4}{n}$ and 
let $Q=Q_0\in H^1(\mathbb{R}^n)$ be the unique positive decreasing radial function that satisfies
	\begin{align}\label{eq:ground_state_solution}
	-Q=-\Delta Q-Q^{1+\frac{4}{n}}.
	\end{align}   
% p = 1+ 4/n  Bourgain and Wang \cite{BW98} constructed solutions of positive energy having blowup rate 
%(T -t)\inv in dimensions $n = 1, 2$, which was later shown unstable  
When the initial data is above the ground state, namely,
$\| \vphi_0 \|_2$ is slightly larger than $\| Q \|_2$,
numerical studies in \cite{LanPaSuSu88} indicate the existence of singular solutions in 2D having a blowup speed
	\begin{align}\label{phi(t):loglog}
	\| \nabla \varphi(t) \|_{2}
	\approx\sqrt{ \frac{\log |\log({\cT} - t)|}{{\cT}-t} }
	\end{align}
at finite blowup time ${\cT} = {\cT}_{\max}$\,.
A rigorous analytical proof of \eqref{phi(t):loglog} in one dimension was initially given in Perelman \cite{Pere01}.
% of the existence and stability for the $\log$-$\log$ speed for generic blowup solutions in one dimension. 
Under certain spectral conditions in Spectral Property  \ref{h:spectral_property} 
(Section \ref{s:spec-loglog}), 
Merle and Rapha\"el \cite{MerRa05a,MerRa05b}
proved the sharp blowup rate of the solutions to equation \eqref{eq:nls} for $n \geq 1$,
i.e.,
there exists a universal constant $\alpha^*$ such that if $\vphi_0\in B_{\alpha^*}$ with negative energy,
then $\vphi(t)=\varphi(t,\cdot)$ is a blowup solution to equation \eqref{eq:nls} with maximal interval of existence $[0, \mathcal{T})$
satisfying the $\log$-$\log$ blowup rate \eqref{phi(t):loglog},
where
	\begin{align}\label{eB_alpha}
	B_{\alpha^*}
	:= \left\{ \phi \in H^1: \int |Q|^2 dx < \int |\phi|^2 dx < \int |Q|^2 dx + \alpha^* \right\},
	\end{align}
see Theorem \ref{thm:log-log-phi}.
Such a blowup rate is also known to be stable in 
the submanifold $S_{\al^*}:=B_{\al^*}\cap \{u\in H^1: \mathcal{E}(u)<0\}$. %of H^1 
The $\log$-$\log$ blowup rate has also been shown for more general nonlinear evolutionary PDEs,
see e.g. \cite{BaFiGa2010s,
Darw12a,
FanMen20c, Lan2021b} %PlanchonRa07e,YangRouZhao20s} 
for NLS type equations. 
In addition, see \cite{DeGangSiWa08,Pere2022,Wein86ss} and \cite{Lan2021b,Zhao2013} 
 %\cite{BozGhMasmYang24}
for the surveys and studies of singularity formation and profiles for other dispersive equations. 
% coherent structure

% \cite{BHZ19a} proved such a log-log law for RNLS when $V=V_\ga$ is a trapping 

The study of the case $V=V_{\ga}(x)$, $\ga<0$ is motivated by 
the treatment of the attractive case $\ga>0$ in \cite{BHZ19a},
where is introduced a pseudo-conformal type transform for the trapping harmonic potential and rotation.
We will show that such ``$\log$-$\log$ law'' continues to hold for equation \eqref{eq:nls_va} in the repulsive case $\ga<0$.
Our proof is based on a virial identity for equation \eqref{eq:nls_va},
the $\mathcal{R}_{\ga}$-transform \eqref{eq:p_transform} that maps solutions of  \eqref{eq:nls} 
to solutions of  \eqref{eq:nls_va},  and an application of the above result of Merle and Rapha\"el's.

Our main result is stated as follows.
Let $\alpha^*$ be the above-mentioned universal constant.

\begin{theorem}\label{thm:log-log-u}
Let $p = 1 + \frac{4}{n}$ and $V_\ga(x)=-\ga^2|x|^2$, $\ga<0$ in equation \eqref{eq:nls_va},
where $1\leq n \leq 10$.
Suppose $u_0\in \mathcal{S}_{\al*}$, namely, 
	\begin{align*}
	u_0 \in \mathcal{B}_{\alpha^*}
	:= \left\{ \phi \in \mathscr{H}^1: \int |Q|^2 dx < \int |\phi|^2 dx < \int |Q|^2 dx + \alpha^* \right\}
	\end{align*}
and
	\begin{align}\label{eq:negative-initial-energy}
	\mathcal{E}(u_0)
	:=\int |\nabla u_0|^2 dx
	-\frac{n}{n+2} \int |u_0|^{2+\frac{4}{n}} dx
	< 0.
	\end{align}
Then there exists ${T^*} =T_{max}< \infty$ such that
$u\in C([0,{T^*}); \mathscr{H}^1)$ is a blowup solution to equation \eqref{eq:nls_va}
  with the $\log$-$\log$ blowup rate
	\begin{align}
	\lim_{t \rightarrow {T^*}}
	\frac{\Vert \nabla u(t) \Vert_2}{\Vert \nabla Q \Vert_{2}}
	\sqrt{ \frac{{T^*} - t}{\log \left| \log ( {T^*} - t)  \right|} }
	= \frac{1}{\sqrt{2 \pi}}\,, \label{eV:loglog}
	\end{align}
where $Q=Q_0$ is the unique solution to equation \eqref{eq:ground_state_solution}.
In addition,
the above conclusion holds true for radial data when $n=11,12$.
\end{theorem}  

We comment that the set $\mathcal{S}_{\al^*}:=\mathcal{B}_{\al^*}\cap \{u\in \Sigma: \mathcal{E}(u)<0\}$ 
in $\Sigma=\sH^1$  
is a stable manifold of such blowup solutions satisfying the $\log$-$\log$ blowup rate. 
In fact, Lemma \ref{l:u-vphi:Sigma-norm} suggests that  
$\norm{u}_\Sigma\approx \norm{\vphi}_\Sigma$ by means of the $\cR_\ga$-transform.   
 Therefore, the stability follows from  Theorem \ref{thm:log-log-u} and Theorem \ref{thm:log-log-phi}. 

The dimension restriction in Theorem \ref{thm:log-log-u} is due to 
the validity of the Spectral Property \ref{h:spectral_property} 
in Section \ref{s:spec-loglog},
which has been verified for $n=1$ to $n=10$ and radially for $n=11, 12$.
See more detailed discussion in Section \ref{s:spec-loglog}.
It would be of interest to further look into  the case of a general (anisotropic) quadratic potential %i.e., the anisotropic 
 $V_{\td{\ga}}(x)$, 
which has been a noteworthy problem in the literature of both mathematics and physics \cite{Aftalion2009,ALZh24n,
BaoWaMar05,Car2011t, %Dinh2022e,
Fetter2009r,LeoZheng22n,LieSei06}. %SowBi04} 

Concerning the limiting behavior for a singular solution when $u_0\in \mathcal{S}_{\alpha^*}$,
we show in Proposition \ref{t:nonexistence} that there exist no blowup solution as in 
Theorem \ref{thm:log-log-u} that admits $L^2$-convergent sequence for any $(t_n)\to T=T^*$. 
Moreover, in 
Theorem \ref{thm:limiting_profile},
we provide an $L^2$- concentration property 
for RNLS (\ref{eq:nls_va}), where the result holds for all $\gamma\in \R$.  
It is worth mentioning that 
the concentration region given by (\ref{eq:w(t)})
shows that for initial data $u_0$ in $\mathcal{S_{\alpha^*}}$, the blowup profiles have
more focused concentration than the 
usual parabolic region $\{|x-x(t)|^2\lesssim |T-t|\}$ as $t\to T$. 
A further investigation in this direction is to obtain a description of the blowup profile.
% Merle  Rapha\"el  have systematically  studied \cite{merle2005blow  merle2004universality, merle2006sharp} 
   %M and R  proved  for  H^1(\R^n)$, $1\le n\le 4$  
Recall from \cite{MerRa05a,MerRa05b} that for equation \eqref{eq:nls}, if 
$\varphi_0 \in \tilde{S}_\al:=
\{\phi\in H^1\cap B_{\alpha^*}: \mathcal{E}(\phi)\le 0\}$, 
% which have non-positive energy and with mass slightly above ground state,
then the corresponding solution  $\vphi(t)$ blows up in finite time ${\cT}={\cT}_{\max}$,
with the following precise asymptotics as $t$ approaches  ${\cT}$:
  % [Log-log blowup dynamics]%\label{loglog:profile}
\begin{equation}\label{phi:limit-behavior}
\varphi(t,x)= \frac{1}{\lambda(t)}(Q+\epsilon)\left(\frac{x-x(t)}{\lambda(t)}\right)e^{-i\gamma(t)},  
\end{equation}
 where  $\lambda(t)\inv=\norm{\nabla_x \vphi}_2\approx \left(\frac{\ln |\ln ({\cT}-t)|}{{\cT}-t}\right)^{1/2} $ and 
\[
\epsilon(t)\xrightarrow{t\rightarrow {\cT}} 0 \textup{ in }\dot{H}^{1}\cap L_{loc}^{2}\,.
\] 
Such blowup property is known to be stable in $H^{1}$ \cite{Ra05s}.
By means of a pseudo-conformal transform (\ref{eq:p_transform}),
we can obtain this type of limiting behavior for blowup profile of equation \eqref{eq:nls_va}
\begin{equation}\label{behavior:loglog} 
u(t,x)= \frac{1}{{\lambda}(t)}(Q+\epsilon)\left(\frac{x-\td{x}(t)}{{\lambda}(t)}\right)e^{-i\td{\gamma}(t)}\,  
\end{equation}
as $t$ goes to $T^*=T_{max}$.
% where  $\lambda(t)\inv\approx \left(\frac{\ln |\ln (\mathcal{T}-t)|}{\mathcal{T}-t}\right)^{1/2} $ and 
%\[
%\epsilon(t)\xrightarrow{t\rightarrow \mathcal{T}} 0 \textup{ in }\dot{H}^{1}\cap L_{loc}^{2}\,.
%\] 

% was later proved to be stable under $H^{s}$ perturbation   for all $s>0$ by Colliander and Rapha\"el \cite{collian09rough}, although one need  reformulate the notion of log-log blup (in a natural way) 

In Section \ref{s:numerics} we provide numerical results that partially support the analytical results above-mentioned.  
 The profile characterization for blowup solutions have  been studied recently with numerical assist 
 for a variety of dispersive equations, see e.g. \cite{BaFiGa2010s,BaFibMan10,BozGhMasmYang24,FibMeRa06,SimZw2011} and the references therein. 

\section{Preliminaries}\label{s:prelimin} 
%[from the article \cite{ReZaStri01} and [Zheng23] 

The RNLS (\ref{e:nls-VOm}) is the modeling equation that describes the wave interaction of a  
  dilute rotating Bose gas with external forces at ultracold temperature.  
  %In general 
  For a quadratic potential $V(x)=V_{\td{\ga}}(x)=\sum_j \sgn(\ga_j)\ga^2_jx_j^2$,
  the term ``repulsive'' (component) refers to either $\ga_j$ is negative for some $j$, or 
   $|\Om|>\sqrt{2}\min \{\ga_1,\ga_2\}$ ({\em fast rotating}) 
when assuming $m=\hbar=1$, %\uga:=\min ga_1,ga_2 V=\ga_1^2 x_1^2+\ga_2^2x_2^2+\ga_3^2x_3^2  time-indep ext oscillator pot
cf. \cite{ReZaStri01,ALZh24n,LeoZheng22n}.  %providing  external confine ga_x> ga_y 
%Common ex of atoms used in BEC experiments include:
%Rubidium: Rubidium-87 was the first atom successfully used to create a BEC in 1995.
%Sodium: Shortly after, a BEC of sodium atoms was created.
%Lithium: BECs have also been created using lithium atoms.
%Other elements: BECs have also been achieved with other elements like hydrogen, strontium, calcium, cesium, chromium, and potassium, among others 

 %here we simulate the ground state and its dynamics by observing the quantum \textcolor{blue}{vortex lattice}   in the (fast) rotation case, see recent work in physics where anisotropic harmonic potential produces Landau level bound state  

%  x2 models a magnetic field whose role is to confine the particles (this is one of the ingredients for BEC, once the atoms have been cooled by a laser, see e.g. [1]), 
%and the nonlinear  takes the (main) interactions between the particles into account
%\iffalse{For the elliptic problem the state Q minimizes the Gross-Pitaevskii

  For large systems, the semiclassical flow motion is given by the hydrodynamical equations of superfluids  
 when the anisotropic potential $V_{\td{\ga}}$ rotates with angular velocity $\widetilde{\Om}=(\Om_1,\Om_2,\Om_3)$.
  %these equation takes the form in the corotating frame:
%\begin{align} & \frac{\pa \rho}{\pa t}+\nabla (\rho ({\bf u}-\Om \times {\bf r}))=0\label{e:rho-u-Om}\\
%& \frac{\pa u}{\pa t}+\nabla \cdot ( \frac{\nu^2}{2}+ \frac{V}{M} +\frac{\mu}{M}- {\bf u}\cdot (\Om \times {\bf r}))=0\label{e:ut-VOm}  \end{align}
In  equation (\ref{e:nls-VOm}), $\mu$ %=\mu_{loc}(\rho)=g\rho$ 
is the chemical potential of the uniform gas 
that is proportional to $4\pi \hbar^2 a N$,
% $g=4\pi \hbar^2 a/M$ is the coupling constant fixed 
where $N$ is the number of particles and
 $a $  the $s$-wave scattering length. %{\bf u}  the velocity field %in the laboratory frame,
In this model, the rotational NLS  %a special form of the Gross-Pitaevskii equation
can assume quantum vortex solutions in the sense of metastable configurations.  
%(since the critical frequency needed to generate a stable vortex becomes smaller and smaller as the number of atoms increases, the solutions in \cite{ReZaStri01} corresponds to metastable configurations in gen

\iffalse \footnote{
In geometric notion, $A$ can be regarded as 1-form, and $B=dA$ the induced 2-form: 
$A=\sum_{j=1}^n A_jdx_j$,
$B=dA=\sum_{j<k}B_{jk} dx_j\wedge dx_k$, $B_{jk}=\pa_j A_k-\pa_k A_j$.
} 
\fi

 %\footnote{Easy (and interesting) to verify that $\frac{d}{dt}H(t)=0$ for all $t$ by noting the anti-symmetry $B^T=-B$, thus
% H(t)=H(0)\quad \forall t 
%How does an electro-magnetic field affect  the NLS
% Newton's law and Lorentz force  from q to classical particles  Soliton dynamics   p<1+4/n

% Since under the assumption above $\cL$ is bounded from below, throughout this paper we may assume $V(x)\approx \la x\ra^2$ without loss of generality. Then all the results 
%remain true as long as we use $(\cL + \lam)^{-s/2}$ for sufficiently large number $\lam$instead of $\cL^{-s/2}$ to define the space $\sH^s$. let $G(u)=2\mu u^{(p+1)/2}/(p+1)  $ that is $F=\mu|u|^{p-1}u$.

\subsection{Rotational NLS on $\R^n$} 
The RNLS (\ref{eq:nls_va}) is on $\R^n$ a special form of the magnetic NLS (mNLS) 
\begin{align}\label{mNLS_AV}
	iu_t= H_{A,V} u+ N(u),  \qquad
	u(0,x)=u_0 \in\mathscr{H}^1, 	%\end{cases}
	\end{align}
% generated  by the Hamiltonian  H_{A,V}   inhomog  \mu |u|^{p-1}u 
where $H_{A,V}=-(\nabla-\frac{i}{2} A)^2+ V_e$,
$A$ is sublinear and $V_e$ is subquadratic as given in  \cite{Ya91}; %Assumption $(A,V)$; %\ref{hyp:AVB}; 
$N(u)$ represents a suitable  power or Hartree nonlinearity,    
see e.g. \cite{De91,Zheng2012a,HLZh21u} and \cite{BHZ19a}.

\iffalse
\footnote{ 
\begin{assumptionAV}\label{hyp:AVB}
% H_{\td{A},\td{V}}:=-(\nabla - i\td{A})^2 +\td{V}$.  % _effect 
 $A=(A_j)_1^n$ is sublinear and $V$ subquadratic, that is,   
%\begin{assumption}\label{a:A-lin-V-quadr}
  $A_j\in C^\iy(\R^n)$ and  $V\in C^\iy(\R^n)$ are real-valued satisfying    
\begin{equation}\label{eAV:sublin-qua}
\begin{aligned} 
&\pa^\al A_j(x)=O(1),  \quad\forall \,|\al| \ge 1, %\label{e:A-sublin} \\
&\pa^\al V(x)=O(1),   \quad\forall \,|\al| \ge 2  %\label{e:V-subquad}
\end{aligned}
\end{equation}
as $|x|\to \iy$.  In addition,  there exists some $\veps>0$ such that for all $|\al | \ge 1$
\begin{align} 
|\pa^\al B(x)|\le c_\al \la x\ra^{-1-\veps} , \label{eB:eps}
\end{align} 
where  $B=(b_{jk})_{n\times n}$ is the magnetic field induced by $A$,
$b_{jk}= \pa_j A_k- \pa_k A_j$, $1\le j,k\le n$,
$|B|=|B(x)|=\sum_{jk} |b_{jk}(x)|$,  cf. \cite{Ya91,De91,Zheng2012a}.   
\end{assumptionAV}
}
\fi

For a RNLS (\ref{eq:nls_va}), 
 the canonical examples are given as follows.  
If $n=2m$,
\begin{align}\label{eM:even}
M=M_{2m}=\begin{pmatrix}
\Om_1\sigma& \quad&\\
 &\Om_2\sigma&\quad\\
&\quad&\ddots\\
&\quad&\quad& \Om_m\sigma
\end{pmatrix}
\end{align}
where 
\begin{align}\label{sigma90}
\sigma=\begin{pmatrix}
0&-1 \\
1& 0
\end{pmatrix}\,.
\end{align}
If $n=2m+1$,
\begin{align}\label{eM:odd}
M=M_{2m+1}=\begin{pmatrix}
\Om_1\sigma& \quad&\\
 &\Om_2\sigma&\quad\\
&\quad&\ddots\\
&\quad&\quad& \Om_m\sigma\\
&\quad&\quad& \quad&0
\end{pmatrix}\,.
\end{align}
For $\Om=(\Om_1,\dots,\Om_m)$ in $\R^m$,
we shall denote $M=M_\Om$ in the sequel. 

If $A=Mx$, $M$ skew-symmetric, then $A\cdot x=0$, % Poincare gauge 
which implies that $A$ is perpendicular to $x$ in $\R^n$.   
Then if $f$ is radial symmetric, then 
one always has $L_Af=0$, since $L_A=A\cdot \nabla $ is a direction derivative
along a tangential vector to the  sphere of radius $r=|x|$.

%\end{document}

\subsection{Solution to the linear flow of RNLS (\ref{eq:nls_va})}
Assuming $u$ is small as $x$ large. 
Then 
\eqref{eq:nls_va} reduces to a linear equation with $\mu=0$
\begin{align}
&iu_t=-\De u -\ga^2 |x|^2 u +\mu|u|^{4/n} u+ L_A u\,, \label{U(t):VA-mu} 
% u(0,x)=u_0(x)\in \mathscr{H}^1
\end{align}
or, 
\begin{align}
&iu_t=Hu=-\De u -\ga^2 |x|^2 u + L_A u \label{U(t):VA} 
% u(0,x)=u_0(x)\in \mathscr{H}^1
\end{align}
where given $\Om=(\Om_1,\dots,\Om_m)$ and $\gamma<0$, %(\ga_1,\dots,\ga_m)$,
$H=H_{\Om,\ga}:=-\De+V_\ga+ i A\cdot \nabla$ is essentially selfadjoint. 
In fact, one can write $H$ in the magnetic form 
$H_{\Om,\ga}=H_{A,V_{\text{effect}}}:=-(\nabla-\frac{i}{2} A)^2+ V_{\text{effect}}$,
where 
in view of (\ref{eM:even}) and (\ref{eM:odd}), we have if $\ga<0$ and $A=Mx$, then
\begin{align}
V_{\text{effect}}(x)=&-\ga^2|x|^2-\frac{|A|^2}{4}  \notag\\
=& \la -\ga^2-\frac{M^2}{4}x,x \ra\notag\\
=&
\begin{cases} 
\sum_{j=1}^m -(\ga^2+\frac{\Om_j^2}{4})(x_{2j-1}^2+x_{2j}^2)& \text{if }\ n=2m\\
\sum_{j=1}^m -(\ga^2+\frac{\Om_j^2}{4})(x_{2j-1}^2+x_{2j}^2)-\ga^2x^2_{2m+1}& \text{if} \ n=2m+1.
\end{cases} \label{Ve:gaOm}
\end{align}

According to {Theorem} \ref{thm:p_transform}, 
%Denote $\cR:=\cR_{\ga}$ for simplicity. %\edz{$\ga<0$; sign of the following needs to be changed to  -\ga 
%Let $p=1+\frac4n$ and $\ga<0$.
Let $\varphi\in C([0,\mathcal{T}),H^1)$ be a global solution to equation \eqref{eq:nls} with $\cT=+\iy$.
Then for all $0\ne\ga$ in $\R$
	\begin{align}\label{eR_transf}
	u(t,x)=& %\mathcal{R} \varphi (t,x)
	\frac{1}{\cosh^\frac{n}{2}(2\gamma t)}
	e^{i\frac{\gamma}{2}|x|^2\tanh(2\gamma t)}
	\varphi\left(\frac{\tanh(2\gamma t)}{2\gamma},\frac{e^{tM}x}{\cosh(2\gamma t)}\right) \notag\\
	=&
	 \left( \frac{\gamma}{2 \pi i \sinh(2 \gamma t)} \right)^\frac{n}{2}
	e^{i \frac{\gamma}{2} |x|^2 \coth(2 \gamma t)}\,
		\end{align}
where we note that $\vphi(t,x)= \frac1{ (4 \pi i  t)^{\frac{n}{2}}}
	e^{i \frac{ |x|^2}{4t}}
	$ solves %with $\mu=0$
	\begin{align}
&i\vphi_t=-\De \vphi  \label{vphi(t):0}  %+\mu |\vphi |^{4/n} \vphi 
% u(0,x)=u_0(x)\in \mathscr{H}^1
\end{align}
	and 
\[ \th(2\ga t)+ \frac1{\ch(2\ga t) \sh(2\ga t)}=\coth(2\ga t).
\]	
% t in \R   
%Let $\cT=\cT_{max}$ and $T^*=T_{max}$  and denote $[0,\cT)$ and $[0,T^*)$ the forward maximal intervals for $\vphi$and $u$ respectively.  
%If $\cT\ge  \frac{1}{2|\gamma|}$, then $u=\cR\vphi$ is a global solution.  

% If $\cT<\frac{1}{2|\gamma|}$, then $u:=\mathcal{R} \varphi$ is a solution to equation \eqref{eq:nls_va}in $C\left( \left[ 0,\frac{\tanh^{-1}(2\gamma \mathcal{T})}{2\gamma} \right),\mathscr{H}^1 \right)

Hence
 the fundamental solution for (\ref{U(t):VA}) %with $\mu=0$ 
 is given by:
\iffalse 
\footnote{ If  $\ga>0$, 
\begin{align}\label{eq:fund-U(t)}
    U_\ga(t,x,y)= \left( \frac{\gamma}{2 \pi i \sin(2 \gamma t)} \right)^\frac{n}{2}
	e^{i \frac{\gamma}{2} |x|^2 \cot(2 \gamma t)}
	e^{i \frac{\gamma}{2} \frac{|y|^2}{\tan(2 \gamma t)}}
	e^{-i \gamma \frac{(e^{tM} x) \cdot y}{\sin(2 \gamma t)}}\,
	%\varphi_0(y) dy
	%:= \int K(t, x, y) \varphi_0(y) dy
	\end{align}  
}
\fi

 If $\ga<0$,  %\eqref{eq:fund-U(t)}:
\begin{align}\label{eq:fund-U(t)}
    U_\ga(t,x,y)= \left( \frac{\gamma}{2 \pi i \sinh(2 \gamma t)} \right)^\frac{n}{2}
	e^{i \frac{\gamma}{2} |x|^2 \coth(2 \gamma t)}
	e^{i \frac{\gamma}{2} \frac{|y|^2}{\tanh(2 \gamma t)}}
	e^{-i \gamma \frac{(e^{tM} x) \cdot y}{\sinh(2 \gamma t)}}\,.
	\end{align}  
This suggests that 
$U(t)f=U_\ga(t)f:=\int U_\ga(t,x,y)f(y)dy$ satisfies the $L^1\to L^\iy$ dispersive estimate:

\iffalse
\footnote{If $\ga>0$,
\begin{equation*}
\norm{U(t)f}_\iy\le C |t|^{-n/2}\norm{f}_1\qquad \forall |t|<\frac{\pi}{4|\ga|}.
\end{equation*}
}
\fi

If $\ga<0$,
\begin{align}\label{eU:disp}
\norm{U(t)f}_\iy\le
\begin{cases}
 C |t|^{-n/2}\norm{f}_1\quad & |t|\le \frac{\pi}{4|\ga|}\\
C e^{-2|\ga| t}\norm{f}_1\quad & |t|>\frac{\pi}{4|\ga|}\,. 
\end{cases}
\end{align} 

Since $U(t)=e^{-itH_{\Om,\ga}}$ is unitary,  
\begin{equation}\label{L2:U(t)}
\norm{U(t)f}_2=\norm{f}_2. 
\end{equation}
According to \cite[Theorem 1.2]{KT98}, 
these two estimates (\ref{eU:disp})-(\ref{L2:U(t)}) are sufficient for us to
obtain the long time Strichartz estimates. 

\begin{lemma} \label{l:stri-global} Let $U(t)=e^{-itH_{\Om,\ga}}$ with $\ga<0$
and $\Om\in\R^m$, $m=\lfloor \frac{n}{2}\rfloor$.  Then there exists $C_0>0$ such that  for all $t$ in $\R$
\begin{align*}
\norm{U(t)f}_{L^q(\R,L^r)}\le& C_0\norm{f}_2\\
\norm{\int_0^t U(t-s)F(s,\cdot) ds}_{L^q(\R,L^r)}\le& C_0\norm{F}_{L^{\td{q}'}(\R,L^{\td{r}'}) }\,.
\end{align*}
where $(q,r), (\td{q},\td{r})$ are any sharp admissible pairs satisfying 
$(q,r,n)\ne (2,\iy,2)$ and  
\begin{equation}
\frac{1}{q}=\frac{n}{2}\left(\frac12-\frac1{r}\right)\,. 
\end{equation} 
\end{lemma}

\subsection{Local theory} %for RNLS} 
\label{ss:local:AV} % (\ref{eq:nls_va} )
Let $X=C_b(\R,\Sigma)\cap L^q(\R,\Sigma^{1,r})$,
here $\Sigma=\sH^1$, $\Sigma^{1,r}=\sH^{1,r}$, 
and $(q,r)$ is admissible such that $q=r=\frac{2n+4}{n}$.  
 Let $E:=\{ \norm{u}_X\le 2C_0 \norm{u_0}_\Sigma\} $ be a closed ball in $X$, 
where 
\begin{align*}
\norm{u}_X:=&\norm{u}_{L^\iy(\R,  \Sigma)}+  \norm{u}_{L^q (\R, \Sigma^{1,r})}\,.
\end{align*}

%\begin{align}
%\psi= e^{-itH_V}\psi_0-i \int_0^t e^{-i(t-s)H_V}  (|\psi|^{p-1} \psi) ds\, ,  \label{Phi:duhamel} \end{align}
%which  is known to be equivalent to (\ref{e: 
Define the map $\Phi$: $X\to X$ by % Duhamel principle 
\[
\Phi(u):=e^{-itH}u_0-i \int_0^t e^{-i(t-s)H_{\Om,\ga}}  (|u|^{p-1} u) ds .
\]
Then the map $\Phi$  is a contraction on  $E\subset X$ if $\norm{u_0}_{\Sigma}<\veps$
provided $\veps$ is sufficiently small.  
%\begin{align*} \norm{f}_{L^\iy_{[0,T_0]} H^2}\le a, \qquad \norm{f}_{L^q_{[0,T_0]} W^{2,r}}\le b  \end{align*}
This shows the  global existence of weak solution $u(t)$ to (\ref{eq:nls_va}) 
in $C_b(\R,\Sigma)\cap L^q(\R,\Sigma^{1,r})$ with $q=r=(2n+4)/n$.
here $C_b$ is the space of continuous functions that are
uniformly bounded in $\R$.
The proof for the contraction property of $\Phi$ 
relies on  the Strichartz estimates in Lemma \ref{l:stri-global}  and smallness of data $u_0$,
which proceeds by following 
 a routine argument  in e.g. \cite{De91,Zheng2012a}. %and \cite[Prop.2.2]{ALZ24non}  

For general data in $\Sigma$, 
we state the local theory for RNLS (\ref{eq:nls_va}) as follows. 
%blowup alternative lemma holds. 

\begin{proposition}\label{p:blup-mNLS} Let $p=1+\frac4{n}$. Let $V(x)=\sgn(\ga)\ga^2|x|^2$, $\ga\in\R$ and $A=M x$, 
$M=M_\Om$ as given in (\ref{eM:even}) and (\ref{eM:odd}). 
Suppose $u_0\in \Sigma$. We have with $q=r=(2n+4)/n$
\begin{enumerate}
\item[(a)] There exists an interval $I_T=[0,T)$ for some positive $T$ and a
unique solution % \Sigma-bounded 
$u(t)$ of (\ref{eq:nls_va}) in $C(I_T,\Sigma)\cap L^q(I_T,\Sigma^{1,r})$.

\item[(b)] Denote $[0,T^*)$, $T^*=T_{max}$ as 
  the maximal lifespan of the solution $u(t)$ of (\ref{eq:nls_va}).   
%The following blowup alternative holds: 
If $T^*=\iy$, then $u$ is a global solution in  $C_{loc}([0,\iy),\Sigma)\cap L^q_{loc}([0,\iy),\Sigma^{1,r})$.
%\begin{enumerate}
   If $T^*<\iy$ is finite, %\mbox{(respectively, $T_*<\iy$)}, 
then $u$ blows up finite time in the sense that 
\begin{align*}
& \lim_{t\to T^*} \norm{\nabla u}_2=\lim_{t\to T^*} \norm{u}_\iy=\iy\,.\\
%&(resp.   \quad  \lim_{t\to -T_*} \norm{\nabla u}_2=\lim_{t\to -T_*} \norm{u}_\iy=-\iy).
\end{align*}
%\end{enumerate}  
\item[(c)] Moreover, there hold the conservation laws for (\ref{eq:nls_va}), namely, 
the mass and energy are conserved on $[0,T^*)$ %$T_*=T_{max}$ %  Hamiltonian % H_{A,V}   whose trajectories 
\begin{align}
& M(u)=\int |u|^2=M(u_0) \label{eM:mass}\\ 
&E(u ):=\int \bar{u} H_{\Om,\ga} u +\frac{2\mu}{p+1}\int | u |^{p+1} = E(u_0)\, \label{Ene:Omga}
%& \ell_A(u):= i\int \bar{u} A\cdot \nabla u \label{ell_A:angular}
\end{align}  
where $\mu=-1$ and  
$H_{\Om,\ga}= -\De+\sgn(\ga)\ga^2 |x|^2 + i (Mx)\cdot \nabla$.  
%H_{\frac{A}{2},V_{\text{effect}}}=-(\nabla-\frac{i}{2} A)^2+ V_{\text{effect}}$. 
%\end{enumerate} 
\item[(d)] The angular momentum $\ell_\Om(u):=- \la \Om\cdot L u,u \ra$ is real-valued and 
\begin{align} 
 \ell_\Om(u)=&\ell_A(u):= i\int \bar{u} A\cdot \nabla u \notag\\ %\label{ell_A:angular}
=& i \la A\nabla u, u\ra 
= \ell_\Om(u_0). \label{conser-Lom} 
\end{align}
\end{enumerate} 
%\end{enumerate}
\end{proposition} 

The local wellposedness for general magnetic NLS (\ref{mNLS_AV}) 
were proven where $A$ is sublinear and $V$ subquadratic, cf. 
 e.g., \cite{De91,Zheng2012a}. 
Here we will only show the conservation of the angular momentum 
(\ref{conser-Lom}) in the rotational case (\ref{eq:nls_va}).
%\end{proposition}\footnote{note that there is sign difference between \eqref{conserLom} and \cite{AMS12}.
% }

\begin{proof} 
Let $ \ell_\Om (t): = -\int \bar{u} \Om\cdot  L u= i \int \bar{u} A\nabla u=:\la L_A u,u\ra$. 
% will  prove the case for $n=3$ only.  
Note that $A\nabla= i\Om \cdot L$.  

\iffalse
\footnote{ The scheme is not working for each $L_k$, since 
$\la  L_3 u,u \ra$ is Not a constant.  

Let $ \ell_\Om (t): = -\int \bar{u} \Om\cdot  L u$. 
We will  prove the case for $n=3$ only.  
Since $L=(L_k)$, $\Om=(\Om_k)$, $k=1,2,3$. 
It suffices to show  $\la  L_3 u,u \ra$ is a constant (Not!)
in $t$ owing to symmetry. 
}
\fi

Taking the derivative in $t$, we have, with $ L_A=i A\nabla$ being selfadjoint
\begin{align*}
 \frac{d}{dt} (  \la  L_A u,u \ra ) %\frac{d}{dt}\left(\int\bar{u} L_z u\right)
=&\int (\bar{u}_t L_A u + \bar{u} L_A u_t)\\
= &\la L_Au,u_t\ra+\la u_t,L_Au\ra=2\Re \la u_t,L_Au\ra \\ 
 %=&- \int \overline{ i u_t} A\nabla ( u) - \int  i u_t A \nabla ( \bar{u})\\
%=& - \int(-\De\bar{u} +V_\ga(x)\bar{u}-\lam |u|^{p-1}\bar{u} + i A\nabla u )\ (A\nabla u)\\
 %& - \int (-\De u+V_\ga u-\lam |u|^{p-1}u +i A\nabla u) A\nabla\bar{u}\\
 =&2\Re\int\De u (A\nabla \bar{u})-2\Re\int V_\ga {u} A\nabla \bar{u}\\
 +&2\Re\int \lam |u|^{p-1}{u} (A\nabla \bar{u}) 
 +2\Re \int A\cdot \nabla u\ (-i A\nabla \bar{u})\\
 =&2\Re\int\De u (A\nabla \bar{u}) -2\Re\int V_\ga(x)u (A\nabla \bar{u})\\ 
 +&2\Re\int \lam {u}|u|^{p-1} (A\nabla  \bar{u})+0 \\ 
=:& 2\Re (I_1 - I_2 + I_3) \,.
 \end{align*}

\iffalse
\footnote{\begin{align*}
 &\frac{d}{dt} (  \la  L_3 u,u \ra ) %\frac{d}{dt}\left(\int\bar{u} L_z u\right)
=\int (\bar{u}_t L_3 u + \bar{u} L_3 u_t)\\
 =&\cob{2}\Re\int\De u (x_2\pa_{x_1}\bar{u}-x_1\pa_{x_2}\bar{u})-2\Re\int V_\ga \bar{u}(x_2\pa_{x_1}-x_1\pa_{x_2})u\\
 +&2\Re\int \lam |u|^{p-1}\bar{u} (x_2\pa_{x_1}-x_1\pa_{x_2})u 
 +2\Re \int  \Om\cdot (L_1,L_2,L_3) u\cdot (x_2\partial_{x_1} - x_1\partial_{x_2})\bar{u}\\
 =&\cob{2}\Re\int\De u (x_2\pa_{x_1}\bar{u}-x_1\pa_{x_2}\bar{u})-2\Re\int V_\ga(x)u(x_2\pa_{x_1}-x_1\pa_{x_2})\bar{u}\\ 
 +&2\Re\int \lam \bar{u}|u|^{p-1} (x_2\pa_{x_1}-x_1\partial_{x_2}) u+0\\ 
=:& 2\Re (I_1 - I_2 + I_3) \,.
 \end{align*}
}
\fi

Some identities. 
\begin{lemma} Let $M$ be a real matrix in $\mk{gl}(n,\R)$ and $A=Mx$. Then 
\begin{enumerate}
\item
\begin{align*}
\nabla (A\cdot\nabla u)= M^T\nabla u+ (\mathbb{H}u) (Mx)
\end{align*}
where $\mathbb{H}u=(\pa_{jk}u)_{n\times n}$ is the Hessian of $u$. 
\item 
\be \label{etM:f}
e^{tMx\cdot \nabla} f(x)= f(e^{tM}x)\,. 
\ee
\item 
\begin{align*}
\nabla (u(Mx) )=& M^T(\nabla u) (Mx).
\end{align*}
\item If $M$ is skew-symmetric, then
\be \label{laplacian:A-nabla}
\dive \nabla (A\cdot\nabla u)=A\nabla (\De u) \,.
\ee
\end{enumerate}
\end{lemma}

\begin{proof}
Given any matrix $M$ in $\mk{gl}(n,\R)$,
to show  the identity (\ref{etM:f}),
we  differentiate both sides of (\ref{etM:f}) in $t$ 
\be 
(Mx\cdot\nabla)  e^{tMx\cdot \nabla} f(x)= Me^{tM}x\cdot \nabla f (e^{tM}x)\,.
\ee
Then let $t=0$ to give 
\be
Mx\cdot\nabla  f(x)= M x\cdot \nabla f (x)\,.
\ee

The commuting relation (\ref{laplacian:A-nabla}) says that $[\De, A\cdot \nabla]=0$,
which is given in \cite[Lemma 4.2]{BHZ19a}, or 
a simple calculation $[\De, L_k]=0$ for $L=(L_k)$, $k=1,2,3$. 
Here we present a different proof using the property that 
the laplacian commutes with all orthogonal matrices  in $O(n)$. % Stein-Weiss 

Let $A=Mx$, $M^T=-M$, then $\cO=e^{tM}\in O(n)$ and  
\be
\De (u(e^{tM}x) )= (\De u) (e^{tM}x). 
\ee
Differentiating in $t$ yields 
\be 
 \De (Mx)\cdot (\nabla u)(e^{tM}x)  =(Mx)\cdot \nabla (\De u)(e^{tM}x).
\ee
Thus (\ref{laplacian:A-nabla}) follows by setting $t=0$.
\end{proof}

\iffalse
\begin{lemma} Let $\wtd{L}_3=x_2\pa_1-x_1\pa_2$ s.t. $L_3= i\wtd{L}_3$.
Similarly, $L_1= i\wtd{L}_1$ and $L_2= i\wtd{L}_2$.  
Then
\begin{equation}
[\td{L}_2,\td{L}_3]=\wtd{L}_1,\ 
[\td{L}_3,\td{L}_1]=\wtd{L}_2,\
[\td{L}_1,\td{L}_2]=\wtd{L}_3.
\end{equation}
\end{lemma}
However, the relations do not show 
$I_4:=\Re \int  \Om\cdot (L_1,L_2,L_3) u\cdot (x_2\partial_{x_1} - x_1\partial_{x_2})\bar{u}$ is zero. 
I have to include the whole angular momentum operator $i A\nabla$ to show that $I_4$ is zero
as above. 
\fi

 We compute by integration by parts, noting that $[\De,L_A]=0$ in $\R^n$, %\eqref{e:De-Lz-commut}
\begin{align*}
 & I_1= \int \De u \ A\nabla \bar{u}\\
=& -\int A\nabla u\ \De\bar{u}
%  =&\int  \left(\partial^2_{x_1} \psi (x_2\partial_{x_1}\bar\psi ) - \partial^2_{x_2}\psi (x_1\partial_{x_2}\bar\psi) - \pa^2_{x_1}\psi ( x_1\partial_{x_2}\bar\psi ) +\pa^2_{x_2} \psi (x_2\partial_{x_1}\bar\psi)\right)\\
%+&\int \pa_3^2u (x_2\pa_{x_1}\bar{u} - x_1\pa_{x_2}\bar{u}).
 \end{align*}
 %Applying integration by parts, we have\begin{align*}
%& I_1= -\int x_2\partial_{x_1}\psi \partial^2_{x_1}\bar\psi + \int x_1\partial_{x_2}\psi \partial^2_{x_2}\bar\psi + \int \partial_{x_2}\bar\psi \partial_{x_1}\psi + \int x_1\partial_{x_1}\psi \partial^2_{x_2 x_1}\bar{u}\\
  %-&\int \partial_{x_1}\bar{u}\pa_{x_2}\psi - \int x_2 \partial_{x_2}\psi\partial^2_{x_1 x_2}\bar{u}\\
  %-&\int \pa_3^2 \bar{u} (x_2\pa_{x_1}{u} - x_1\pa_{x_2}{u}).\\
  %&I_1= - \int \Delta\bar{u}( x_2\partial_{x_1} - x_1\partial_{x_2})\psi + 2(i) Im \int\partial_{x_1}\psi\partial_{x_2}\bar\psi + \int (x_1\partial_{x_1}u \pa^2_{x_2 x_1}\bar\psi - x_2 \partial_{x_2}\psi\partial^2_{x_1 x_2}\bar\psi)

\begin{align*}
 I_1= - \overline{I_1}\To \Re I_1=0.
\end{align*}
 %After applying integration by parts on $ I_{1,1} $ we get,\\
 
%$ Re I_1= 0$
Let $V_\ga(x)=-\ga^2|x|$ , $\ga< 0$.  
For $ I_2 $, we have by I.B.P., 
\begin{align*}
& I_2= \int V_\ga (x)u\ A\nabla \bar{u}\\
=& -\int A\nabla V_\ga\ |u|^2 - \int V_\ga \bar{u} A\nabla u\\
=&0-\int V_\ga \bar{u} A\nabla u\\ %provided \;\ga_1=\ga_2\\ 
\To& \Re I_2= 0.
\end{align*} 
%here $\widetilde{L}_z=x_2\pa_{x_1}-x_1\pa_{x_2}=\la x_2,-x_1,0\ra\cdot \nabla=-A\cdot\nabla$.

Now for $ I_3 $, writing $A\nabla u= i\Om\cdot L$, we may assume $L=L_3$ w.l.o.g. when $n=3$.
%\edz{the proof for the rotation works now! ---oct.25} 

\begin{align*}
 &\overline{I_3}=\int \lam \bar{u} |u|^{p-1} \widetilde{L}_3  u\\ 
 =& - \int  x_2\pa_{x_1}(\lam(x) \bar{u}|u|^{p-1}) u +\int x_1\pa_{x_2}(\lam(x) \bar{u}|u|^{p-1}) u \\
 =&-\int x_2(\pa_{x_1}\lam) |u|^{p+1} + \int x_1(\pa_{x_2} \lam) |u |^{p+1} -\int \lam u  (x_2\pa_{x_1} - x_1\pa_{x_2}) (\bar{u}|u |^{p-1})\\
 =&-\int (x_2\pa_1-x_1\pa_2)( \lam)\cdot (|u|^{p+1})\\
  -&\int\lam u |u|^{p-1}  (x_2\pa_{x_1}- x_1\pa_{x_2})\bar{u} -\int\lam |u|^2  (x_2\pa_{x_1}- x_1\pa_{x_2})(| u|^{p-1})\\
=&-\int\lam u |u|^{p-1}  \wtd{L}\bar{u} -(p-1) \Re \int \lam (| u|^{p-1}u) \la x_2, - x_1,0\ra \cdot \nabla \bar{u}
\end{align*}
where we used that $\lam(x)=\lam(|x|)\To L_3\lam(x)=0$.
Taking the real part gives 
\begin{align*}
 &\Re (I_3)= -p \Re(I_3)\To
 \Re(I_3) =0. \end{align*}%---Yi first gave proof for 2d; nyla proof did not work, she made it right on the board later 
Combining the results of $ I_1, I_2, I_3 $ above, we obtain
\begin{equation*}
 \frac{d}{dt} ( \left< L_A u,u \right> ) = 0 .
\end{equation*}
\end{proof}

\begin{remark} \label{re:local:AV}
The local existence theory in Proposition \ref{p:blup-mNLS}
continues to hold for
 the magnetic NLS  (\ref{mNLS_AV}).   
The proof can be carried out in the same manner as  in e.g. \cite{ALZh24n}.
\end{remark}   %Section \ref{s:pf-lwpAV}.  

\subsection{Threshold for $L^2$-critical focusing RNLS} %\eqref{eq:nls_va}}    
The threshold theorem for mass critical focusing NLS (\ref{eq:nls}) was obtained in \cite{Wein83}.
It states that if $\norm{\vphi_0}_2< \norm{Q}_2$, then $\vphi(t)$ exists globally in $H^1$.
For each $c>1$, 
 there exists an initial data $\vphi_0=\vphi_{0,c}\in H^1$ with $\vphi_{0,c}=cQ$ %certain profile
  such that the lifespan of the solution $\vphi(t)$ 
is finite, that is, $\cT_{max}<\iy$ and $\vphi$ blows up on $[0,\cT_{max})$. 
Note that  $E(Q)=0$ and  $E(cQ)<0$ for all $c>1$. 

Singularity formation and its structure was studied in \cite{Wein86ss} for a class of dispersive equations.
For NLS (\ref{eq:nls}), Merle \cite{Mer93} gave a characterization for minimal mass blowup solution at the ground state level $\norm{\vphi_0}=\norm{Q}_2$. 
These  minimal mass blow-up solutions are shown to be unstable, where $E(\vphi)>0$, 
cf. e.g. \cite[p.597]{MeRa03}.  

In \cite[Theorem 1.1]{BHHZ23t}, there is proven the  $L^2$-critical threshold result for RNLS (\ref{eq:nls_va}) in the case of
  $V_\gamma={\ga^2} |x|^2$ for positive $\ga$.  
  We now state the analogue  in the case of $\ga<0$.
  Let  $ Q \in H^1(\R^n)$ is be the unique positive, radial  ground state of   
\begin{align} &  -\De Q- |Q|^{4/n} Q= -Q.\quad \label{e:eigen-Q_critic}
\end{align}

\begin{proposition}\label{t:u0>Q-blowup}
 Let $p=1+\frac4n$ and $V_\ga(x)=-\ga^2|x|^2$, $\ga<0$.
 Suppose $u_0\in \mathscr{H}^1$. Then: %and $\lam=1$  
\begin{enumerate}
\item[(a)] If $\Vert u_0\Vert_2<\Vert Q\Vert_2$,  there exists a global in time solution of \eqref{eq:nls_va}. %{e:nls-H-rot}.
\item[(b)] Given any $c\ge 1$, there exists $\ga=\ga(c)$ and $\norm{u_0}_2=c\norm{Q}_2$
such that for all $0<|\ga|< \ga(c)$, 
there exists finite time blowup solution of (\ref{eq:nls_va}) 
so that $\norm{\nabla u(t) }_2\to +\iy$ as $t\to T^*=T_{max}$.  
%\Vert Q\Vert_2$, there exist finite time blowup  solutions of \eqref{eq:nls_va}
%when taking $u(0,x)=c \lam^{n/2}Q(\lam x)$ with $|c|\ge 1$,  \lam>0 
\end{enumerate}
\end{proposition}

\begin{proof} (a) If $\Vert u_0\Vert_2<\Vert Q\Vert_2$, 
then in view of \cite{Wein83}, 
the solution $\vphi(t)$ of (\ref{eq:nls}) exists globally so that $\norm{\vphi(t,\cdot)}_{H^1}$ is bounded for all $t$.
According to Theorem \ref{thm:p_transform}, 
 $u(t)$ of (\ref{eq:nls_va}) exists globally in time so that $\norm{u(t,\cdot)}_\Sigma$ is bounded for all $t$.

(b) For any $c\ge 1$, there is $\vphi_0$ with $\norm{\vphi_0}= c\norm{Q}_2$  so that
$\vphi$ blowups finite time on $[0,\cT)$ such that $\norm{\nabla \vphi(t)}_2\to +\iy$
as $t\to \cT$.  

According to Theorem \ref{thm:p_transform}, 
if $|\ga|<\frac{1}{2\cT}$, we have 
$u=\cR \vphi$ blows up finite time with $\norm{\nabla u(t) }_2\to +\iy$ 
as $t\to \frac{\tanh^{-1}(2\gamma \mathcal{T})}{2\gamma}$.  

If  $|\ga|\ge \frac{1}{2\cT}$, we have 
$u=\cR \vphi$ exists globally. %with $\norm{\nabla u(t) }_2\to +\iy$ 
%as $t\to \frac{\tanh^{-1}(2\gamma \mathcal{T})}{2\gamma}
\footnote{The $\Sigma$-bound of $u(t)$ will depend on the profile of the initial data $u_0$}
%$\norm{u(t)}_\Sigma$ }
\end{proof}

\begin{proof}[The $\Sigma$-boundedness for $u(t)$ may not be obtained 
as in the trapping case $\ga>0$]  
We will apply the sharp Gagliardo-Nirenberg inequality 
\begin{equation} \label{sharp:GN:Q}
\norm{u}_{2+4/n}^{2+4/n}\le c_{GN}  \norm{ u}_2^{4/n} \norm{\nabla u}_2^2
\end{equation}
where $c_{GN}= \frac{n+2}{n}\norm{Q}_2^{-4/n}$, $Q=Q_0$ is the positive radial solution of 
(\ref{eq:ground_state_solution}),  %e:eigen-Q_critic}),
cf. e.g. \cite{Wein83} or \cite[Lemma 5.2]{BHHZ23t}. 

%\footnote{ (39') along with simple scaling property for $L^2$ norm of $Q_{\lam,1}$ by taking $\lam=1/2$.Note that $\norm{Q_{\lam,1}}_2=\norm{Q_0}_2$ }

Together with the diamagnetic inequality we have 
\begin{equation} \label{GN:mag}
\norm{u}_{2+4/n}^{2+4/n}\le \frac{p+1}{2\norm{Q}_2^{\frac4n} }  \norm{ u}_2^{\frac4n} \norm{(\nabla-i \td{A}) u}_2^2
\end{equation}
where $\td{A}\in L^2_{loc}(\R^n,\R^n)$ is a real, vectorial-valued function that is locally square integrable. 

From (\ref{Ene:Omga}), (\ref{eM:mass}) and (\ref{GN:mag})
\begin{align*}
E(u_0)=E(u )=&\int \bar{u} H_{\Om,\ga} u -\frac{2}{p+1}\int | u |^{p+1}\\ 
\ge& \int \bar{u} H_{\Om,\ga} u -\left( \frac{  \norm{ u}_2}{\norm{Q}_2} \right)^{\frac4n} \norm{\nabla u}_2^2\\ 
\ge& \left( 1-\left( \frac{  \norm{ u}_2}{\norm{Q}_2} \right)^{\frac4n}\right) \norm{(\nabla-\frac{i}{2} A) u}_2^2+ \int V_{\text{effect}} |u|^2
\end{align*}
where the effect potential $V_{\text{effect}}$ is given as $\eqref{Ve:gaOm}$.  
 One sees that if $\norm{ u_0}_2<\norm{Q}_2 $,   since $\in V_{\text{effect}}|u|^2$ has negative sign, 
we could not yet claim the $\Sigma$-bound for $u(t)$ for all $t$. 
\end{proof} 

\begin{remark} The argument above also indicates that 
for general $V$ with a repulsive component along the $x_j$ axial direction,  
under the condition  $\norm{ u_0}_2<\norm{Q}_2 $,   %since $\in V_{\text{effect}}|u|^2$ has negative sign, 
we could not prove the $\Sigma$-bound for $u(t)$ for all $t$. 
\end{remark}

\section{Blowup criterion} %for RNLS} 
\label{s:blup-Criterion} % (\ref{eq:nls_va})

\subsection{Blowup result if $\norm{ u}_2\ge \norm{Q}_2 $}   
% Given $ V_{\text{effect} } $  having negative sign, 
% can we show the blowup % the $\Sigma$-bound for $u(t)$ for all $t$.}
% using the virial identity?

Following \cite{BHHZ23t} we show the virial id. for (\ref{eq:nls_va}) first. 
Let $u$ in $C_0^2(\R^n)\cap \Sigma$ solve \eqref{mu:nls-VA} with $\mu=-\lam(x)$ real
\begin{align}\label{mu:nls-VA}
	\begin{cases}
	iu_t=-\Delta u+V_\ga u - \lam |u|^{p-1}u+i A \cdot\nabla u \\
	u(0,x)=u_0 \in\Sigma=\mathscr{H}^1.
	\end{cases}
	\end{align}
 
\begin{lemma}\label{l:J(t)-VOm} Let $V_\ga=\sgn(\ga)|x|^2$, $\ga<0$ and $A=Mx$, $M$ is skew-symmetric in $\mk{gl}(n,\R)$.
Let $u\in \Sigma\cap C_0^2(\R^n)$ be solution of (\ref{mu:nls-VA}) or \eqref{eq:nls_va}.  
Define $J(t):=\int |x|^2 |u|^2$.  Then 
\begin{align}
J'(t) =& {4} \Im \int x \overline{u} \cdot \nabla u=-4\Im \int x u\cdot \nabla_A \bar{u}\quad(\because x\cdot A=0) \label{e:dJ:AV} \\
% J''(t)=& %2 \int |\nabla u|^2-2 \gamma^2 \int |x|^2 |u|^2
              %	- \left( 6 - \frac{12}{p+1} \right) \lambda \int |u|^{p+1}.
J''(t) 
=& 8 \left( E_{\Om,\ga}(u_0) - \ell_{\Om}(u_0)\right) 
+\frac{4\lam(4-n (p-1)) }{p+1}\int |u|^{p+1}\notag\\
 &+16\ga^2\int |x|^2 |u|^2 \,.\label{ddJ:E-ell}
\end{align}
\end{lemma} %\edz{checked Aug.11, 25'}

The lemma can be proved using  integration by part as was carried out 
 in the proof of \cite[Lemma 3.1]{BHHZ23t}. 
See also \cite{Gar12} for the general case of the magnetic NLS. %with a different  

Recall that     %	\begin{equation*} \begin{cases}
	%i u_t = - \frac{1}{2} \Delta u + \frac{\gamma^2}{2} |x|^2 u - \lambda |u|^{p-1} u - \Omega \cdot L u, \\
	%u(x, 0) = u_0(x) \in H^s(\mathbb{R}^3), \end{cases}
	%\end{equation*}
 $\gamma^2$ is the potential frequency, $\lambda$ is a positive constant, $\Omega = (\Omega_1, \Omega_2,  \Omega_3)$ is the rotation frequency, and $L = -i x \wedge \nabla$.
%For  convenience, we write the equation as	\begin{equation*}
%	u_t = i \frac{1}{2} \Delta u - i \frac{\gamma^2}{2} |x|^2 u + i \lambda |u|^{p-1} u + i \Omega \cdot L u.\end{equation*}
 Denote
	\begin{align*}
	L_{x_1} = i \left(x_3 \partial_{x_2} - x_2 \partial_{x_3} \right), \quad
	L_{x_2} = i \left(x_1 \partial_{x_3} - x_3 \partial_{x_1} \right), \quad
	L_{x_3} = i \left(x_2 \partial_{x_1} - x_1 \partial_{x_2} \right).
	\end{align*}
Then %an easy computation shows that 
$L = (L_{x_1}, L_{x_2}, L_{x_3})$.
We use the notation $\langle \cdot, \cdot \rangle$ to denote the inner product defined as
$\displaystyle	\langle f, g \rangle= \int f \overline{g}$. 

To show the blowup assertion,  
we use the virial identity for RNLS as given in Lemma \ref{l:J(t)-VOm} along with an uncertainty principle inequality. 
However, the solution of the ODE associated with (\ref{ddJ:E-ell}) has a different behavior in $t$.    
If $p=1+4/n$, we have
\begin{equation}
	% J(t) := \int |x|^2 |u|^2.\label{e:variance-J(t)}	
J''(t)- 16\ga^2J(t)
= 8 \left( E_{\Om,\ga}(u_0) - \ell_{\Om}(u_0)\right).\label{ode:J(t)}
\end{equation} %\end{document} 

\subsection{Limiting behavior near ground state profile} \label{ss:behavior-aboveQ} % Q profile  
For the standard NLS \eqref{eq:nls}, if $ p=1+4/n$
and
 $\vphi_0\in S_{\al^*}$, i.e., $\vphi_0 \in H^1$ is slightly above the ground state $Q$ as in (\ref{eB_alpha})
with negative energy, 
there holds the $\log$-$\log$ law  (\ref{phi(t):loglog}) 
along with its limiting behavior (\ref{behavior:loglog}), 
%This is Merle and Raphael made improvements on Weinstein's result by obtaining the exact blowup rates and profiles for the solutions when 
%$\norm{\vphi_0}_2$ is above the ground state 
see \cite{MeRa03,MerRa05a, MerRa05b,Ra05s}.   
For the profile of the blowup solutions, they are self-similar near the singularity,
i.e., one can write $\vphi=\vphi_R+\eps$, where 
\begin{equation*}
    \vphi_R(t,x)=
        \frac{1}{L^{n/2}(t)}
        Q\left( \frac r{L(t)} \right) e^{i\int_0^t\frac{ds}{L^2(s)}}
\end{equation*} 
with~$r=|{x}|$ and $\eps\in L^2$,  
cf. \cite{MeRa03,YangRouZh18}. % Subsection 3.2
%The self-similar profile~$Q(\rho)$ is the ground state of the equation\[
%    R^{\prime\prime}(\rho) +\frac{d-1}{\rho} R^\prime -R+| R|^{4/d} R=0. \]
%where $R$ attains  global maximum  $\vphi_R$ at 0.  
Thus it is easy to observe that the blowup rate ~$\norm{\nabla\vphi}_2
=\norm{\nabla\vphi_R}_2=L(t)\inv$ is given by %the  $\log$-$\log$ law
\begin{equation}
    L(t) \sim \left(
        \frac{2\pi( \cT-t)}{\log\log1/(\cT-t)} \right)^{1/2}\to 0_+ 
    \qquad \text{as}\ t\to \cT=\mathcal{T}_{max}.
\label{eq:logloglaw}
\end{equation}
%$\int_0^t\frac{1}{L^2(s)}ds= \int_0^t\frac{ |\log|\log (T-s)|\ |  }{T-s} ds \le C(t)<\iy$  converges. 
\iffalse
\footnote{one can calculate 
\begin{align*} 
& \int_{\R^n} |\vphi_R(t,x)|^2 dx=  \int_{\R^d} |Q(x)|^2 dx\\ %=\norm{R}^2_{L^2(\R^d)}\\
& \norm{\vphi_R(t,\cdot)}_\iy=\frac1{ L(t)^{n/2} }Q(0) = \frac1{ L(t)^{n/2} }\norm{Q}_{L^\iy(\R^n)}\\
& \norm{\nabla_x\vphi_R(t,\cdot)}_{L^2(\R^n)}=  \frac1{ L(t)}\norm{\nabla Q}_{L^2(\R^n)}\,. 
\end{align*}
}\fi
Concerning the rotational  NLS (\ref{eq:nls_va}), % {Blowup profile description for \eqref{eq:nls_va}}
  %  data near  ground state  \norm{u_0}_2\ge \norm{Q_0}_2 
it is desirable to further examine the blowup behavior for the solution $u(t)$ near  $Q$ 
in view of  % threshold level $\norm{Q}_2
 Proposition \ref{t:u0>Q-blowup} and %the $\log$-$\log$ law in 
Theorem \ref{thm:log-log-u}. %for general %sublinear $A$ and  subquadratic V 
 Heuristically, the blowup profile is of self-similar concentration type, 
and the confinement geometry of a quadratic potential is not sufficient to counter-effect 
the concentrating behavior of the ground state profile which has exponential decay. % away from the center of concentration 
We comment that  
the presence of  $A$ generally may contribute to the {\em phase} as well as 
the concentration center of the blowup solution,  
cf. \cite{CazE88,Squa2009} %for $p<1+4/d
and \cite{BHZ19a}. %for $p=1+4/d 

%\footnote{ The path $x=x(t)$ is the {\em concentration line} (which can be considerably influenced by the presence of $B$,
%see the phase portraits in figures~\ref{1fig}-\ref{2fig}). Initial data~\eqref{initialD} should also be thought as corresponding to a {\em point particle}
%with position $x_0$ and velocity $\xi_0$.  }

At the ground state level  $\norm{u_0}_2=\norm{Q}_2$, 
in  the attractive case $\ga>0$, %via $\cR$ transform  
 the blowup solutions were described 
 in \cite[Proposition 4.5]{BHZ19a} 
 as  having finite time blowup speed $\norm{\nabla u}_2=O( (T^*-t)\inv)$ as $t\to T^*<\iy$.
In the repulsive case $\ga<0$, 
Corollary \ref{c:blup-infty} % $L^2$ subcritical focusing mNLS  subcritical regime $p\in (1,1+\frac4d)
 constructs a global solution that is exponentially unbounded in the sense $\norm{\nabla u}_2= O( e^{2|\ga| t} )$ as $t\to \iy$.  

When the data is slightly above the ground state level, %$\norm{u_0}_2>\norm{Q}_2$, 
\eqref{phi:limit-behavior} provides certain limiting behavior asymptotics.  
 % isotropic case for $V$.  
In Section \ref{s:numerics}, our computations for the isotropic potential $V_\ga$ in 2D
   supports  that such blowup profile is a stable one, %
where we apply the dynamic rescaling method in the simulations. %the  non-radial case.  
%where  Xiao-ping Wang and Landman initially did for NLS in 1991 

We remark that 
 the anisotropic case $V_\ga=\sum\sgn(\ga_j)\ga_j^2x_j^2$ remains a challenging one. %which has been a project under active study recently.   
In this case, there is the breaking symmetry of rotation, hence the failure of angular momentum conservation.
Such  technical issue  % induce 
 might lead to difficulty in proving the blowup criterion for (\ref{eq:nls_va}).  %based on virial identity argument 
Numerically it may give rise to some instability when applying the moving meshing method.

\section{Spectral property and the $\log$-$\log$ law}\label{s:spec-loglog}

Now we will always assume $p=1+\frac{4}{n}$ unless otherwise specified.
To show the blowup rate for initial data above the ground state $Q$ as stated in Theorem \ref{thm:log-log-u},
we need the following \emph{Spectral Property}.
Let $y$ denote the spatial variable in $\mathbb{R}^n$.

\begin{SP}\label{h:spectral_property}
Consider the two linear operators
	\begin{align*}
	L_1
	:= - \Delta + \frac{2}{n} \left( \frac{4}{n} + 1 \right) Q^{\frac{4}{n} - 1} y \cdot \nabla Q, \qquad
	L_2
	:= - \Delta + \frac{2}{n} Q^{\frac{4}{n} - 1} y \cdot \nabla Q,
	\end{align*}
and the real-valued quadratic form
	\begin{align*}
	H(\varepsilon, \varepsilon)
	:= (L_1 \varepsilon_1, \varepsilon_1) + (L_2 \varepsilon_2, \varepsilon_2)
	\end{align*}
for $\varepsilon = \varepsilon_1 + i \varepsilon_2 \in H^1$.
Let
	\begin{align*}
	Q_1
	:= \frac{n}{2} Q + y \cdot \nabla Q, \qquad
	Q_2
	:= \frac{n}{2} Q_1 + y \cdot \nabla Q_1.
	\end{align*}
Then there exists a universal constant $\delta_0 > 0$ such that for every $\varepsilon \in H^1$,
if
	\begin{align*}
	(\varepsilon_1, Q)
	=& (\varepsilon_1, Q_1)
	= (\varepsilon_1, y_j Q)_{1 \leq j \leq n}=0,\\
	 (\varepsilon_2, Q_1)
	=& (\varepsilon_2, Q_2)
	= (\varepsilon_2, \partial_{y_j} Q)_{1 \leq j \leq n}
	= 0,
	\end{align*}
then
	\begin{align*}
	H(\varepsilon, \varepsilon)
	\geq \delta_0 \left( \int |\nabla \varepsilon|^2 dy + \int |\varepsilon|^2 e^{-|y|} dy \right).
	\end{align*}
\end{SP}

The proof of  Spectral Property \ref{h:spectral_property} in any dimension is not complete.
Merle and Rapha\"el \cite{MerRa05a} first proved that this spectral property holds for dimension $n=1$,
using the explicit solution $Q(x) = \left( \frac{3}{\cosh^2(2x)} \right)^\frac{1}{4}$ of \eqref{eq:ground_state_solution}.
Later,  Fibich, Merle and Rapha\"el \cite{FibMeRa06} gave a numerically-assisted proof in $n = 2, 3, 4$.
Recently,  an improved numerically-assissted proof by Yang, Roudenko and Zhao \cite{YangRouZh18}
shows that  Spectral Property \ref{h:spectral_property} holds for $n\leq 10$ and also for $n=11, 12$ in the radial case.
\footnote{The  Spectral Property \ref{h:spectral_property} is equivalent to the coercivity for $L_1$ and $L_2$ on quadratic forms,
 the study of which involving the ground
state solution $Q$ naturally appears in a perturbation setting when dealing with stability problem. 
These two operators are related to the Lyapounov functionals $L_\pm$,  where 
$L_+=-\De+1 -(1+\frac{4}{n})  Q^{\frac{4}{n}} $ and
$L_-=-\De+1 -Q^{\frac{4}{n}} $, see \cite{FibMeRa06}. 
%quadratic forms are then related to the asymptotic form of the Hamiltonian near Q and their coercivity properties can be derived from the variational formulation of Q  Weinstein 
}

%Based on the  Spectral Property [A] \ref{h:spectral_property},  the following blowup rate for equation \eqref{eq:nls} holds according to

\begin{theorem}\label{thm:log-log-phi}
Let $p = 1 + \frac{4}{n}$ in equation \eqref{eq:nls}. 
Let $1\leq n \leq 10$.
Then there exists a universal constant $\alpha^*> 0$ such that the following is true.
Suppose	$\vphi_0 \in B_{\alpha^*}$ satisfies
	\begin{align*}
	\mathcal{E}(u_0)
	=\int |\nabla u_0|^2 dx
	-\frac{n}{n+2} \int |u_0|^{2+\frac{4}{n}} dx
	<0.
	\end{align*}
Then $\varphi \in C([0, \mathcal{T}); H^1)$ is a blowup solution of  \eqref{eq:nls} in finite time $\mathcal{T} < \infty$,
which admits the $\log$-$\log$ blowup rate
	\begin{align}\label{e:log-log_Q}
	\lim_{t \rightarrow \cT}
	\frac{\| \nabla \varphi(t) \|_2}{\| \nabla Q \|_2}
	\sqrt{ \frac{\mathcal{T} - t}{\log \left| \log (\mathcal{T} - t)  \right|} }
	= \frac{1}{\sqrt{2 \pi}}\,,
	\end{align}
where $Q$ is the unique solution of \eqref{eq:ground_state_solution}. 
Moreover, according to \cite{Ra05s}, 
the set $S_{\al^*}=B_{\al^*}\cap \{\phi\in H^1:\mathcal{E}(\phi)<0\}$ is an open stable manifold
in the sense that given any element $\vphi_0$ in $S_{\al^*}$,
there is an open neighborhood $U=U_{\vphi_0}$ of $\vphi_0$ in $S_{\al^*}$ such that 
the blowup solution flow $\phi\mapsto \phi(t)$ admits the $\log$-$\log$ blowup rate (\ref{e:log-log_Q})
for all $\phi\in U$.  

In addition, the above statements hold for  $n=11, 12$ for radial solutions.
\end{theorem}

\section{$\mathcal{R}_\ga$-transform for $-\De -{\ga^2}|x|^2 + i A \cdot \nabla $}\label{s:R:transform} 
Let $H_{A,\ga}:=-\De +\sgn(\ga){\ga^2}|x|^2+ i A \cdot \nabla$. %with $\ga<0$.
In \cite{BHZ19a},  there was introduced the $\cR_\ga$-transform for the case ${\ga}>0$. 
In order to examine the behavior of the solution to (\ref{eq:nls_va}) where $\ga<0$,
 we extend the $\mathcal{R}_{\ga}$-transform to the case $\ga$ being negative,
 which  allows to convert the solutions between (\ref{eq:nls}) and \eqref{eq:nls_va}. 
This can be viewed as an extension of the lens type transforms to the rotational setting, also consult 
\cite{Car2011t} and \cite{Cassano2016}
for the cases of time-dependent potentials and magnetic potentials, respectively.  

\begin{theorem}\label{thm:p_transform} 
Denote $\cR:=\cR_{\ga}$ for simplicity. %\edz{$\ga<0$; sign of the following needs to be changed to $-\ga$} 
Let $p=1+\frac4n$ and $\ga<0$.
Let $\vphi_0=u_0$ and $\varphi\in C([0,\mathcal{T}),H^1)$ be a (forward) maximal solution to equation \eqref{eq:nls}.
Define
	\begin{align}\label{eq:p_transform}
	\mathcal{R} \varphi (t,x)
	:=\frac{1}{\cosh^\frac{n}{2}(2\gamma t)}
	e^{i\frac{\gamma}{2}|x|^2\tanh(2\gamma t)}
	\varphi\left(\frac{\tanh(2\gamma t)}{2\gamma},\frac{e^{tM}x}{\cosh(2\gamma t)}\right).\qquad
	\end{align}
%for $t$ in $\R$. %
Let $\cT=\cT_{max}$ and $T^*=T_{max}$ 
and denote $[0,\cT)$ and $[0,T^*)$ the forward maximal intervals for $\vphi$
and $u$ respectively. 
 If $\cT<\frac{1}{2|\gamma|}$, then
$u:=\mathcal{R} \varphi$ is a solution of equation \eqref{eq:nls_va}
in $C\left( \left[ 0,\frac{\tanh^{-1}(2\gamma \mathcal{T})}{2\gamma} \right),\mathscr{H}^1 \right)$.
%In particular,
If $\mathcal{T} \geq \frac{1}{2|\gamma|}$,
then $u\in C([0,\infty),\mathscr{H}^1)$ is a global solution of equation \eqref{eq:nls_va}.

The transform $\mathcal{R}$ is invertible and its inverse is given by
	\begin{align}\label{eq:inverse_r_transform}
	\mathcal{R}^{-1} u(t,x)
	=\frac{1}{\left( 1-4\gamma^2 t^2 \right)^\frac{n}{4}}
	e^{-i \gamma^2 |x|^2 \frac{t}{1-4\gamma^2 t^2}}
	u\left( \frac{\tanh^{-1}(2\gamma t)}{2\gamma},
	\frac{e^{-\frac{\tanh^{-1}(2\gamma t)}{2\gamma} M} x}{\sqrt{1-4\gamma^2 t^2}} \right).
	\end{align}
The inverse transform $\mathcal{R}^{-1}$ maps $u\in C([0,T^*),\mathscr{H}^1)$,
a solution of equation \eqref{eq:nls_va},
to $\varphi\in C\left( \left[0,\frac{\tanh(2\gamma T^*)}{2\gamma} \right),H^1 \right)$,
a solution of equation \eqref{eq:nls} for any $t<\frac{1}{2|\ga|}$.   
%\footnote{regardless of $T^*$ finite or infinity, in which case \nabla \phi might go to infinite or not depending on the growth of $\nabla u$ as $t\to T^*$}
\end{theorem}

\footnote{note:  
(i) $ \mathcal{T}<\frac{\tanh^{-1}(2\ga \cT )}{2\ga}<+\iy$ if $2\ga \cT<1$;\\
(ii) $\frac{\tanh (2\ga t)}{2\ga}<\frac{1}{2\ga}\le \cT $   if $2\ga \cT\ge 1$
}
% \footnote{
% Separate note:$\tanh x<x<\tan x$ and $\tan\inv x<x<\tanh\inv x $ for $x>0$.
% }  

\begin{remark}\label{re:Tmax}
 Under the $\mathcal{R}$ transform, $u(0)=\vphi(0)$. 
Observe that if $(0,\mathcal{T})$ is the lifespan of the forward in time solution $\vphi$
and $\mathcal{T}<\frac1{2|\ga|}$,
then
$\left[ 0,\frac{\tanh^{-1}(2\gamma \mathcal{T})}{2\gamma} \right)$
is the lifespan of $u$. 
\end{remark}

\begin{remark}  Note that %via  elementary calculations, 
if $y=\tanh\inv (2|\ga|\tau)$ with $2|\ga|\tau<1$,
then $\cosh y=\frac1{\sqrt{1-4\ga^2\tau^2}}$.
\end{remark}
%\end{document}

To prove Theorem \ref{thm:p_transform} we first need a few identities,
which  can be verified by straightforward calculations. 

\begin{lemma}\label{Lem:ddU:R}
\begin{enumerate}[label={\rm (\alph*)}]

\item For all $n\times n$ matrix $M_0$ and  $1\leq j,k\leq n$,
there holds
	\begin{align}\label{eq:rotation_t}
	\frac{d}{dt}\left((e^{tM_0})_{j,k}\right) =(M_0e^{tM_0})_{j,k}\,.
	\end{align}

\item If $u=\mathcal{R} \varphi$, using $M^T=-M$ we have 
	\begin{align}\label{eq:u_x}
	\begin{split}
	\nabla u
	&=\frac{i\gamma x\sinh(2\gamma t)}{\cosh^{\frac{n}{2}+1}(2\gamma t)}
	e^{i\frac{\gamma}{2}|x|^2\tanh(2\gamma t)}
	\varphi\left(\frac{\tanh(2\gamma t)}{2\gamma},\frac{e^{tM}x}{\cosh(2\gamma t)}\right) \\
	&\qquad+\frac{1}{\cosh^{\frac{n}{2}+1}(2\gamma t)}
	e^{i\frac{\gamma}{2}|x|^2\tanh(2\gamma t)}
	e^{-tM}\nabla \varphi\left(\frac{\tanh(2\gamma t)}{2\gamma},\frac{e^{tM}x}{\cosh(2\gamma t)}\right),
	\end{split}
	\end{align}
and
	\begin{align}\label{eq:u_xx}
	\begin{split}
	\Delta u
	&=\frac{in\gamma\sinh(2\gamma t)}{\cosh^{\frac{n}{2}+1}(2\gamma t)}
	e^{i\frac{\gamma}{2}|x|^2\tanh(2\gamma t)}
	\varphi\left(\frac{\tanh(2\gamma t)}{2\gamma},\frac{e^{tM}x}{\cosh(2\gamma t)}\right) \\
	&\qquad-\frac{\gamma^2|x|^2\sinh^2(2\gamma t)}{\cosh^{\frac{n}{2}+2}(2\gamma t)}
	e^{i\frac{\gamma}{2}|x|^2\tanh(2\gamma t)}
	\varphi\left(\frac{\tanh(2\gamma t)}{2\gamma},\frac{e^{tM}x}{\cosh(2\gamma t)}\right) \\
	&\qquad+\frac{i2\gamma\sinh(2\gamma t)}{\cosh^{\frac{n}{2}+2}(2\gamma t)}
	e^{i\frac{\gamma}{2}|x|^2\tanh(2\gamma t)}
	\nabla \varphi\left(\frac{\tanh(2\gamma t)}{2\gamma},\frac{e^{tM}x}{\cosh(2\gamma t)}\right)\cdot(e^{tM}x) \\%\qquad{checked}
	&\qquad+\frac{1}{\cosh^{\frac{n}{2}+2}(2\gamma t)}
	e^{i\frac{\gamma}{2}|x|^2\tanh(2\gamma t)}
	\Delta \varphi\left(\frac{\tanh(2\gamma t)}{2\gamma},\frac{e^{tM}x}{\cosh(2\gamma t)}\right)\,.
	\end{split}
	\end{align} %\edz{checked}
\end{enumerate}
\end{lemma}

In Lemma \ref{Lem:ddU:R} we note that for any matrix $A_0\in M_{n\times n}$,  
$\nabla_x (\phi(A_0 x)) = A_0^T(\nabla\phi)(A_0x)$ and, if $C_0, A_0\in M_{n\times n}$, 
then   
\[
\dive(C_0(\nabla_x \vphi)(A_0x) ) =\tr \left(C_0 (D^2 \vphi)(A_0x) A_0 \right);
\] 
and  in addition, $C \tr(N) C\inv=\tr(N)$ for any invertible matrix $C$,
where  
$D^2\phi=( \partial_{ij}\phi)_{n\times n}$ denotes the Hessian of $\phi$.

\subsection{Pseudo-conform transform} %for RNLS (\ref{eq:nls_va})}
\label{ss:pc-RNLS}
\bigskip
\begin{proof}[Proof of Theorem \ref{thm:p_transform}]
In view of (\ref{eq:p_transform}), we show  that $u=\cR\vphi$
 solves (\ref{eq:nls_va}) via  an elementary calculation.
 Using \eqref{eq:rotation_t} in Lemma \ref{Lem:ddU:R} for the differentiation of the angular momentum term,
the left-hand side of (\ref{eq:nls_va})
	\begin{align*}
	\text{LHS}
	&=iu_t \\
	&=-\frac{in\gamma\sinh(2\gamma t)}{\cosh^{\frac{n}{2}+1}(2\gamma t)}
	e^{i\frac{\gamma}{2}|x|^2\tanh(2\gamma t)}
	\varphi\left(\frac{\tanh(2\gamma t)}{2\gamma},\frac{e^{tM}x}{\cosh(2\gamma t)}\right) \\
	&\qquad-\frac{\gamma^2|x|^2}{\cosh^{\frac{n}{2}+2}(2\gamma t)}
	e^{i\frac{\gamma}{2}|x|^2\tanh(2\gamma t)}
	\varphi\left(\frac{\tanh(2\gamma t)}{2\gamma},\frac{e^{tM}x}{\cosh(2\gamma t)}\right) \\
	&\qquad+\frac{i}{\cosh^{\frac{n}{2}+2}(2\gamma t)}
	e^{i\frac{\gamma}{2}|x|^2\tanh(2\gamma t)}
	\varphi_t\left(\frac{\tanh(2\gamma t)}{2\gamma},\frac{e^{tM}x}{\cosh(2\gamma t)}\right) \\
	&\qquad+\frac{i}{\cosh^{\frac{n}{2}+1}(2\gamma t)}
	e^{i\frac{\gamma}{2}|x|^2\tanh(2\gamma t)}
	\nabla \varphi\left(\frac{\tanh(2\gamma t)}{2\gamma},\frac{e^{tM}x}{\cosh(2\gamma t)}\right)\cdot(Me^{tM}x) \\
	&\qquad-\frac{i2\gamma\sinh(2\gamma t)}{\cosh^{\frac{n}{2}+2}(2\gamma t)}
	e^{i\frac{\gamma}{2}|x|^2\tanh(2\gamma t)}
	\nabla \varphi\left(\frac{\tanh(2\gamma t)}{2\gamma},\frac{e^{tM}x}{\cosh(2\gamma t)}\right)\cdot(e^{tM}x) \\
	&=-\frac{in\gamma\sinh(2\gamma t)}{\cosh^{\frac{n}{2}+1}(2\gamma t)}
	e^{i\frac{\gamma}{2}|x|^2\tanh(2\gamma t)}
	\varphi\left(\frac{\tanh(2\gamma t)}{2\gamma},\frac{e^{tM}x}{\cosh(2\gamma t)}\right) \\
	&\qquad-\frac{\gamma^2|x|^2}{\cosh^{\frac{n}{2}+2}(2\gamma t)}
	e^{i\frac{\gamma}{2}|x|^2\tanh(2\gamma t)}
	\varphi\left(\frac{\tanh(2\gamma t)}{2\gamma},\frac{e^{tM}x}{\cosh(2\gamma t)}\right) \\
	&\qquad-\frac{1}{\cosh^{\frac{n}{2}+2}(2\gamma t)}
	e^{i\frac{\gamma}{2}|x|^2\tanh(2\gamma t)}
	\Delta \varphi\left(\frac{\tanh(2\gamma t)}{2\gamma},\frac{e^{tM}x}{\cosh(2\gamma t)}\right) \\
	&\qquad-\frac{1}{\cosh^{\frac{n}{2}+2}(2\gamma t)}
	e^{i\frac{\gamma}{2}|x|^2\tanh(2\gamma t)}
	\left|\varphi\left(\frac{\tanh(2\gamma t)}{2\gamma},\frac{e^{tM}x}{\cosh(2\gamma t)}\right)\right|^\frac{4}{n}
	\varphi\left(\frac{\tanh(2\gamma t)}{2\gamma},\frac{e^{tM}x}{\cosh(2\gamma t)}\right) \\
	&\qquad+\frac{i}{\cosh^{\frac{n}{2}+1}(2\gamma t)}
	e^{i\frac{\gamma}{2}|x|^2\tanh(2\gamma t)}
	\nabla \varphi\left(\frac{\tanh(2\gamma t)}{2\gamma},\frac{e^{tM}x}{\cosh(2\gamma t)}\right)\cdot(Me^{tM}x) \\
	&\qquad-\frac{i2\gamma\sinh(2\gamma t)}{\cosh^{\frac{n}{2}+2}(2\gamma t)}
	e^{i\frac{\gamma}{2}|x|^2\tanh(2\gamma t)}
	\nabla \varphi\left(\frac{\tanh(2\gamma t)}{2\gamma},\frac{e^{tM}x}{\cosh(2\gamma t)}\right)\cdot(e^{tM}x)\,.
	\end{align*}%\edz{checked}
Meanwhile, the right-hand side of (\ref{eq:nls_va}) is
	\begin{align*}
	\text{RHS}
	&=-\Delta u-\gamma^2|x|^2u-|u|^\frac{4}{n}u+iA\cdot\nabla u \\
	&=-\frac{in\gamma\sinh(2\gamma t)}{\cosh^{\frac{n}{2}+1}(2\gamma t)}
	e^{i\frac{\gamma}{2}|x|^2\tanh(2\gamma t)}
	\varphi\left(\frac{\tanh(2\gamma t)}{2\gamma},\frac{e^{tM}x}{\cosh(2\gamma t)}\right) \\
	&\qquad+\frac{\gamma^2|x|^2\sinh^2(2\gamma t)}{\cosh^{\frac{n}{2}+2}(2\gamma t)}
	e^{i\frac{\gamma}{2}|x|^2\tanh(2\gamma t)}
	\varphi\left(\frac{\tanh(2\gamma t)}{2\gamma},\frac{e^{tM}x}{\cosh(2\gamma t)}\right) \\
	&\qquad-\frac{i2\gamma\sinh(2\gamma t)}{\cosh^{\frac{n}{2}+2}(2\gamma t)}
	e^{i\frac{\gamma}{2}|x|^2\tanh(2\gamma t)}
	\nabla \varphi\left(\frac{\tanh(2\gamma t)}{2\gamma},\frac{e^{tM}x}{\cosh(2\gamma t)}\right)\cdot(e^{tM}x) \\
	&\qquad-\frac{1}{\cosh^{\frac{n}{2}+2}(2\gamma t)}
	e^{i\frac{\gamma}{2}|x|^2\tanh(2\gamma t)}
	\Delta \varphi\left(\frac{\tanh(2\gamma t)}{2\gamma},\frac{e^{tM}x}{\cosh(2\gamma t)}\right) \\
	&\qquad-\frac{\gamma^2|x|^2}{\cosh^\frac{n}{2}(2\gamma t)}
	e^{i\frac{\gamma}{2}|x|^2\tanh(2\gamma t)}
	\varphi\left(\frac{\tanh(2\gamma t)}{2\gamma},\frac{e^{tM}x}{\cosh(2\gamma t)}\right) \\
	&\qquad-\frac{1}{\cosh^{\frac{n}{2}+2}(2\gamma t)}
	e^{i\frac{\gamma}{2}|x|^2\tanh(2\gamma t)}
	\left|\varphi\left(\frac{\tanh(2\gamma t)}{2\gamma},\frac{e^{tM}x}{\cosh(2\gamma t)}\right)\right|^\frac{4}{n}
	\varphi\left(\frac{\tanh(2\gamma t)}{2\gamma},\frac{e^{tM}x}{\cosh(2\gamma t)}\right) \\
	&\qquad+\frac{i}{\cosh^{\frac{n}{2}+1}(2\gamma t)}
	e^{i\frac{\gamma}{2}|x|^2\tanh(2\gamma t)}
	\nabla \varphi\left(\frac{\tanh(2\gamma t)}{2\gamma},\frac{e^{tM}x}{\cosh(2\gamma t)}\right)\cdot(Me^{tM}x).
	\end{align*}%\edz{checked. }

Here, if denoting $\la v,w\ra$ to be the standard inner product on $\R^n$, we have,
$\la Mx, x\ra=\la x, M^T x\ra=-\la x,Mx\ra\To \la Mx, x\ra=0$.
And, for the last line of RHS, we write $\la v, e^{-tM}w\ra=\la e^{tM}v,  w\ra$, since the adjoint of $e^{-tM}$ is equal to
$(e^{-tM})^T=e^{tM}\in SO(n)$ owing to the skew-symmetry for $M\in \mf{so}(n)$. 
Hence, a comparison of the two sides gives (\ref{eq:nls_va}). 
% LHS$=$RHS.
%\end{proof}
\footnote{For the terms involving $\ga^2|x|^2$, we have used $+\sh^2 z-\ch^2 z=-1$}

It remains to consider the intervals of existence for $u$ and $\vphi$. 
Assume that $u_0\in\mathscr{H}^1$ and the solution $\varphi$ to equation \eqref{eq:nls} blows up at time 
$\mathcal{T}=\cT_{max}>0$
in the usual sense $\norm{\nabla\vphi}_2\to \iy$ as $t\to \mathcal{T}$.
Then, by \eqref{eq:u_x},  one can easily compute that 
if $|\gamma|<\frac{1}{2 \mathcal{T}}$,
the solution $u$ to equation \eqref{eq:nls_va} blows up at time $T^*=\frac{\tanh^{-1}(2\gamma \mathcal{T})}{2\gamma}$,
that is, $\norm{\nabla u}_2\to \iy$ as $t\to T^*$.
Moreover, if $|\gamma|\geq \frac{1}{2 \mathcal{T}}$,
the solution $u$ to equation \eqref{eq:nls_va} exists globally in $C(\mathbb{R},\mathscr{H}^1)$.
Indeed,   in view of (\ref{eq:u_x}), we observe 
 that  $\nabla u(t)$ has a singularity when $\frac{\tanh(2\gamma t)}{2\gamma}=\cT$, %temporal variable for $\vphi$ 
that is,  $t=\frac{\tanh\inv (2\ga \cT)}{2\ga}$. %\edz{\crr True} 
Now, the statements of Theorem \ref{thm:p_transform} concerning the lifespan intervals are valid in light of the following observations:
\begin{enumerate}
\item[(i)] %(for the former claim on blow-up)
 $|\gamma|< \frac{1}{2 \mathcal{T}} \To$ 
the singularity for $\nabla u$, namely, 
  $ T^*=\frac{\tanh\inv (2\ga \cT)}{2\ga}$ is finite  (since $2|\ga|\cT<1$). %if $\ga<\frac1{2\cT} and domain of \tanh^{-1} is (-1,1)!
\item[(ii)]  %for the latter on global time existence 
$|\gamma|\ge\frac{1}{2 \mathcal{T}}\To$  
$t=\frac{\tanh\inv (2\ga \cT)}{2\ga}$ can not be achieved for any $t$, or equivalently, 
the singularity $\cT$ of $\vphi$ is never achievable because
$\frac{\tanh(2\gamma t)}{2\gamma}<\cT$ (since $2|\ga|\cT\ge 1$).
 % and \vphi exists for any t  < \cT_{max}
\end{enumerate}
This concludes the proof of Theorem \ref{thm:p_transform}. 
\end{proof}

Some elementary calculations give the following lemma,
which implies that the pseudo-conformal transform $\cR$ 
preserves the topology as well as the stability of the flow $\vphi(t)=\vphi(t,\cdot)$ in $\Sigma$
 \begin{equation}
 \norm{\cR \vphi}_\Sigma\approx \norm{\vphi}_\Sigma. 
\end{equation} 

\begin{lemma}\label{l:u-vphi:Sigma-norm}  
Let $\varphi\in C([0,\mathcal{T}),H^1)$ be a solution to equation \eqref{eq:nls} 
and let $u=\cR \vphi$.
Under the hypotheses in Theorem \ref{thm:p_transform}, 
we have  If $\cT<\frac{1}{2|\gamma|}$,  then
\begin{enumerate}  
\item[(a)] For all $t\in [0,\frac{\tanh\inv(2\gamma \cT)}{2\gamma})$
\be \label{eL2:u-vphi}
\norm{\cR \vphi}_2=\norm{\vphi(\frac{\tanh(2\gamma t)}{2\gamma},\cdot) }_2
\ee
%\end{enumerate}
%\end{lemma} \end{document}
\item[(b)] 
\be %\label{du:dphi} 
\norm{\nabla (\cR \vphi)}_2
\le\th (2\ga t) \norm{x\vphi (\frac{\tanh(2\gamma t)}{2\gamma},\cdot)}_2
+\frac1{\ch 2\ga t} \norm{\nabla \vphi(\frac{\tanh(2\gamma t)}{2\gamma}, \cdot) }_2\notag
\ee
\be \label{eq:xU}
\norm{x (\cR \vphi)}_2 = \cosh(2\ga t) \norm{x\vphi(\frac{\tanh(2\gamma t)}{2\gamma},\cdot)}_2\,.
\ee
\item[(c)]  For all $t\in[0,\cT )$, 
\be \notag%\label{ed:phi-u} 
\norm{\nabla \vphi}_2 
\le\frac{2\ga^2 t}{ \sqrt{1-4\ga^2 t^2} }  \norm{x (\cR \vphi ) (\frac{\tanh^{-1}(2\gamma t)}{2\ga} , \cdot) }_2
+ \norm{\nabla (\cR\vphi )(\frac{\tanh^{-1}(2\gamma t)}{2\ga},\cdot)}_2
\ee
\be \notag
\norm{x \vphi(t,\cdot)}_2 = \sqrt{1-4\ga^2 t^2} \norm{x (\cR\vphi) (\frac{\tanh^{-1}(2\gamma t)}{2\ga},\cdot)}_2\,. 
\ee  
\end{enumerate}  
 \end{lemma} 
 
 % Remark can consider the problem for complex potential  V_z= z^2 |x|^2  that is partially serving a role of damping  

\begin{remark}\label{re:gamma_critical_value}
%  compare this result with that for the RNLS with  $V=V_\gamma$ in equation \eqref{eq:nls_va}
Recall that in the attractive case where $V_\ga=\gamma^2 |x|^2$, $\ga>0$,
the solution $u$ to equation \eqref{eq:nls_va} 
always blows up at time $T^*=\frac{\tan^{-1}(2\gamma \mathcal{T})}{2\gamma}$
which is strictly less than $\cT$, see \cite[Proposition 4.3]{BHZ19a}.
However, in the repulsive case where $V_\ga=-\ga^2|x|^2$, $\ga<0$,  
the blowup time for $u$ is achieved at $T^*=\frac{\tanh^{-1}(2\gamma \mathcal{T})}{2\gamma}$,
which is strictly larger than $\cT$, the blowup time for $\vphi$. 
The comparison suggests that an attractive harmonic potential accelerates the blowup time
while the presence of a repulsive harmonic potential reduces the blowup time. 
More remarkably, in the repulsive case, 
when the magnitude $|\ga|$ is sufficiently large 
(dependent on the profile of the initial data $u_0\in \Sigma$),
the system (\ref{eq:nls_va}) will admit a global in time solution!
\end{remark} 
\iffalse
\footnote{bec if $V=V_{+\ga}$, $u:=\mathcal{R} \varphi$ is a solution to  \eqref{eq:nls_va} 
in $C\left( \left[ 0,\frac{\tan^{-1}(2\ga \mathcal{T})}{2\ga} \right),\mathscr{H}^1 \right)$.
As $t\to \tan\inv (2\ga\cT)/(2\ga)<\cT=\cT_{max}$, 
in (\ref{eq:phi-to-u}) %eq:p_transform}) 
we note that the temporal variable for $\vphi$ approaches to 
$\tan (2\ga\cT)/(2\ga)=\cT_{max}$, which is the singularity time for $\nabla \vphi$.
}
\fi

\subsection{Solutions of RNLS that blowup at infinity time} \label{critical-massQ}

Theorem \ref{thm:p_transform} indicates that $\gamma=-\frac{1}{2 \mathcal{T}}$ is a critical value. 

\begin{corollary}[Global solutions having exponential growth in $\sH^1$-norm]\label{c:blup-infty} 
Let $\gamma \le -\frac{1}{2 \mathcal{T}}<0$. 
Then there exists a global  solution $u$ of (\ref{eq:nls_va}) that blows up forward at infinity time 
such that $\norm{\nabla u}_2=O(e^{2|\ga| t})$ as $t\to +\iy$. 
%By time-reversal symmetry, there exists a global solution of \eqref{eq:nls_va} that blowups backward in 
%$t=-\iy$ with blowup rate $O(e^{2\ga|t| })$ as well.  
%at gss level $\norm{u_0}_2=\norm{Q}_2
\end{corollary} 
\begin{proof} We can construct such solution by virtue of Theorem \ref{thm:p_transform}. 
   Consider  the solution to (\ref{eq:nls}) with initial data $\varphi(0,x)=u_0$ 
   such that $\norm{u_0}_2=\norm{Q}_2$ and $\vphi(t)$ is a finite time blowup solution on $[0,\cT)$.
   We have %   {Mer93} showed that if $u_0\in H^1$, $\norm{u_0}_2=\norm{Q}_2$  and the solution $\varphi$ to equation \eqref{eq:nls} blows up at time $\mathcal{T}>0$, then the only minimal mass blowup solution is of the form (up to scaling and phase invariance
	\begin{align}\label{Phi:mini-mass}
	\varphi(t,x)
	=\frac{1}{(\mathcal{T} - t)^\frac{n}{2}}
	e^{-\frac{i|x|^2}{4(\mathcal{T} - t)}}
	e^{\frac{i}{\mathcal{T} - t}}
	Q\left( \frac{x}{\mathcal{T} - t} -x_0 \right)
	\end{align}
for some $x_0\in\mathbb{R}^n$, cf. \cite{Mer93} or \cite{BHZ19a,Car02c}. 
We divide the discussion into two cases. 

\noindent Case 1.  $\gamma=-\frac{1}{2 \mathcal{T}}$. 
% according to Theorem \ref{thm:p_transform}, 
Using (\ref{eq:p_transform})  we obtain %if $\gamma=-\frac{1}{2 \mathcal{T}}$,
 the solution to equation \eqref{eq:nls_va} 
	\begin{align}\label{U:singularity-infinity}
	u(t,x)
	=\left( 2|\gamma| e^{2{|\gamma|} t} \right)^{\frac{n}{2}}
	e^{ i\frac{\cor{\gamma}}{2}|x|^2}
	e^{i |\gamma| \left( e^{4 |\gamma| t}+1 \right)}
	Q\left( 2|\gamma| e^{2|\gamma| t} e^{tM}x -x_0 \right).
	\end{align} %\edz{\cor{checked}}
This is a global solution with exponential growth in $\Sigma=\sH^1$-norm.  In fact, 
	\begin{align*} 
	\nabla u(t,x)
	&=\left( 2 |\gamma| e^{2 |\gamma| t} \right)^{\frac{n}{2}}
	e^{i\frac{\gamma}{2}|x|^2}
	(i \gamma x)
	e^{i |\gamma| \left( e^{4 |\gamma| t}+1 \right)}
	Q\left( 2|\gamma| e^{2|\gamma| t} e^{tM}x -x_0 \right) \\
	&\; +\left( 2|\gamma| e^{2|\gamma| t} \right)^{\frac{n}{2}}
	e^{ i\frac{\gamma}{2}|x|^2}
	e^{i |\gamma| \left( e^{4 |\gamma| t}+1 \right)}
	2|\gamma| e^{2|\gamma| t} (e^{tM})^T
	\nabla Q\left( 2|\gamma| e^{2|\gamma| t} e^{tM}x -x_0 \right) \\
	&:=\text{I}+\text{II}.
	\end{align*}
It is easy to see that 
 $\norm{\text{I}}_2$ is bounded 
 and $\norm{\text{II}}_2= 2|\gamma| e^{2|\gamma| t}\norm{\nabla Q}_2$\,. 
%\end{proof} 

\bigskip
\noindent Case 2.  $\gamma<-\frac{1}{2 \mathcal{T}}$, i.e., $|\ga|>\frac1{2\mathcal{T}}$. 
In this case, %comparison  if |\gamma| >\frac{1}{2 \mathcal{T}}$, 
 the solution $u$ is global with {exponential growth of $\norm{\nabla u}_2$} as well.  

To see this, applying the $\mathcal{R}_{\ga}$-transform to the minimal mass blowup solution \eqref{Phi:mini-mass},
we obtain
	\begin{align*}
	u(t,x)
	=&\left( \frac{ 2 \gamma }{ \cosh(2\gamma t) \left( 2\gamma \mathcal{T} - \tanh(2\gamma t) \right) } \right)^\frac{n}{2}
	e^{i\frac{\gamma}{2} |x|^2 \left( \tanh(2\gamma t) - \frac{1}{\cosh^2(2\gamma t) (2\gamma \mathcal{T} -\tanh(2\gamma t))} \right)}\\
	&\cdot e^{\frac{i 2\gamma }{2\gamma \mathcal{T} - \tanh(2\gamma t) }} 
	 \cdot  Q\left( \frac{ 2 \gamma e^{tM}x }{\cosh(2\gamma t) \left( 2 \gamma \mathcal{T} - \tanh(2\gamma t) \right) } -x_0 \right), 
                        \end{align*}
% =&\left( \frac{ 2 \gamma }{ \cosh(2\gamma t) \left( 2\gamma \mathcal{T} - \tanh(2\gamma t) \right) } \right)^\frac{n}{2}
%	e^{i\frac{\ga}{2} |x|^2 \frac{ \frac12\sh(4\ga t)  \big(2\ga \cT -\th 2\ga t\big) -1}{2\ch^2(2\ga t) \big(2\ga\cT-\th 2\ga t \big)}}\\
%	&\cdot e^{\frac{i 2\gamma }{2\gamma \mathcal{T} - \tanh(2\gamma t) }} 
%	 \cdot
%	Q\left( \frac{ 2 \gamma e^{tM}x }{\cosh(2\gamma t) \left( 2 \ga \cT - \th(2\ga t) \right) } -x_0 \right)\,,
%	\end{align*}
where note that $ 2|\ga| \cT -\th(2|\ga| t) | >1\cor{-}\th 2 |\ga| t>0$ for all $t$.  
%if $0<t<\frac{\tanh^{-1}(2 |\gamma| \mathcal{T})}{2 |\gamma|} $. 
So
	\begin{align*}
	&\nabla u(t,x)\\
	=&\left( \frac{ 2 \gamma }{ \cosh(2\gamma t) \left( 2\gamma \mathcal{T} - \tanh(2\gamma t) \right) } \right)^\frac{n}{2}
	i \gamma x \left( \tanh(2\gamma t) - \frac{1}{\cosh^2(2\gamma t) (2\gamma \mathcal{T} -\tanh(2\gamma t))} \right) \\
	\quad \cdot& e^{i\frac{\gamma}{2} |x|^2 \left( \tanh(2\gamma t) - \frac{1}{\cosh^2(2\gamma t) (2\gamma \mathcal{T} -\tanh(2\gamma t))} \right)}
	e^{\frac{i 2\gamma }{2\gamma \mathcal{T} - \tanh(2\gamma t) }}
	Q\left( \frac{ 2 \gamma e^{tM}x }{\cosh(2\gamma t) \left( 2 \gamma \mathcal{T} - \tanh(2\gamma t) \right) } -x_0 \right) \\
	+&   \left( \frac{ 2 \gamma }{ \cosh(2\gamma t) \left( 2\gamma \mathcal{T} - \tanh(2\gamma t) \right) } \right)^{\frac{n}{2}+1}
	e^{i\frac{\gamma}{2} |x|^2 \left( \tanh(2\gamma t) - \frac{1}{\cosh^2(2\gamma t) (2\gamma \mathcal{T} -\tanh(2\gamma t))} \right)}
	e^{\frac{i 2\gamma }{2\gamma \mathcal{T} - \tanh(2\gamma t) }} \\
	\quad \cdot& %\frac{ 2 \gamma }{\cosh(2\gamma t) \left( 2 \gamma \mathcal{T} - \tanh(2\gamma t) \right) }
	\left( e^{tM} \right)^T
	\nabla Q\left( \frac{ 2 \gamma e^{tM}x }{\cosh(2\gamma t) \left( 2 \gamma \mathcal{T} - \tanh(2\gamma t) \right) } -x_0 \right) \\
	:=& \text{I} + \text{II}.
	\end{align*}
Thus we have 
	\begin{align*}
	\norm{\text{I} }_2
	=&\frac{1}{2} \left| \sinh(2\gamma t) \left( 2\gamma \mathcal{T} -\tanh(2\gamma t) \right) - \frac{1}{ \cosh(2\gamma t)} \right|
	\norm{ {\red x}\cdot Q(\cdot -x_0)}_2\\
	\approx &\frac14  e^{2|\ga| t} \left( 2|\ga| \cT-1 \right) \norm{ { x}\cdot Q(\cdot -x_0)}_2\qquad \text{as $t\to \iy$}
	\end{align*} %\edz{checked 08-26, 25'}
and
	\begin{align*}
	\| \text{II} \|_2
	=&\frac{ 2 \gamma }{ \cosh(2\gamma t) \left( 2\gamma \mathcal{T} - \tanh(2\gamma t) \right) } \|\nabla Q\|_2\\
	=&\left( \frac{ 4 |\gamma| e^{-2|\ga| t} }{ 2|\ga|\mathcal{T}-1 } +o(1)\right) \|\nabla Q\|_2 \qquad \text{as $t\to \iy$},
	\end{align*}
which suggest that $\norm{\text{I} }_2=O(e^{2|\ga| t})$ grows exponentially as $t$ tends to infinity and 
$\norm{ \text{II} }_2$ is bounded.
\end{proof} %\edz{checked}

\iffalse
\footnote{  %\begin{remark} 
The corollary suggests that  [Car02] might be incorrect claiming that $u$ is global with exponential decay.
} %\end{remark}
\fi

\begin{remark}[Comparison with the attractive case $V_\ga$, $\ga>0$] 
Recall that in \cite[Proposition 4.3]{BHZ19a} in the case 
$V_\ga(x)=\ga^2 |x|^2,\ga>0$, the corresponding $\cR_\ga$-transform is given as follows. 
 %\begin{proposition}\label{p:u-phi-R_pseudo-conformal}
Let $\varphi(t,x) \in C([0,\cT), H^1)$ be a maximal solution of \eqref{eq:nls} where $\cT=\cT_{max} > 0$.
Define the $\mathcal{R}_\ga$-transform $\vphi\mapsto \mathcal{R_\ga}(\vphi)$ to be
	\begin{align}\label{eq:phi-to-u}
	\mathcal{R}_\ga\vphi(t, x)
	:= \frac{1}{\cos^\frac{n}{2}(2 \gamma t)}
	e^{-i \frac{\gamma}{2} |x|^2 \tan(2 \gamma t)}
	\varphi\left( \frac{\tan(2 \gamma t)}{2 \gamma}, \frac{e^{tM} x}{\cos(2 \gamma t)} \right).
	\end{align}
Then $u=\mathcal{R}_\ga\vphi$ is a solution of \eqref{eq:nls_va} 
in $C([0, \frac{\arctan(2 \gamma \cT)}{2 \gamma}), \mathscr{H}^1)$.

Conversely, let $u(t,x) \in C([0,T^*), \mathscr{H}^1)$
be a solution to \eqref{eq:nls_va} where $T^* \in (0, \frac{\pi}{4 \gamma}]$.
Then $\vphi=\mathcal{R}_\ga\inv u$,  given by 
	\begin{align}\label{eq:u-to-phi}
	\varphi(t, x)
	:= \frac{1}{(1 + (2 \gamma t)^2)^\frac{n}{4}}
	e^{i \frac{\gamma^2 |x|^2 t}{1 + (2 \gamma t)^2}}
	u\left( \frac{\arctan(2 \gamma t)}{2 \gamma}, \frac{e^{-\frac{\arctan(2\gamma t)}{2\gamma} M} x}{\sqrt{ 1 + (2 \gamma t)^2}} \right),
	\end{align}%\edz{\eqref{eq:u-to-phi} checked aug.5, 3:23am}
is a solution of \eqref{eq:nls} in $C([0, \frac{\tan(2 \gamma T^*)}{2 \gamma}), H^1)$.
%here $\mathcal{R}\inv $ is the inverse of $\mathcal{R}
%\end{proposition}
%What if $T\ge \pi/4\ga$? Are we going to have solitons?

Let $\ga>0$. In view of the above statements, namely \cite[Prop.4.3]{BHZ19a}, 
%from the conversion between $u$ and $\vphi$, noting  \norm{\nabla u(t)}_2\approx \norm{\nabla \vphi}_2$,
we must have if $u$ has finite time blowup solution then
$T^*=T_{max}$ strictly less than $\frac{\pi}{4\ga}$.  
If $u(t)$ has a lifespan beyond $(0,\frac{\pi}{4\ga})$,
then $u(t)$ must be a global solution on $[0,\iy)$. 

At the critical mass level  $\norm{u_0}_2=\norm{Q}_2$, 
if $T^*\le \pi/4\ga$,  \cite{BHZ19a} gives the characterization of 
all blowup solutions of RNLS (\ref{eq:nls_va}) with $\ga>0$. 
Especially, if $\vphi(t,x)=e^{it}Q(x)$,  then under the pseudo-conformal mapping 
$u=\cR_\ga \vphi$ is a finite time blowup solution such that $\norm{\nabla u}_2=O( (T^*-t)\inv )$
as $t\to T^*=\frac{\pi}{4\ga}$. 

In the case $\ga<0$, %in view of
$\cR_\ga$ in (\ref{eq:p_transform})  maps $\vphi=e^{it}Q$ as 
\be
u(t,x) %\mathcal{R} \varphi (t,x) 
	=\frac{1}{\cosh^\frac{n}{2}(2\gamma t)}
	e^{i\frac{\tanh(2\gamma t)}{2\gamma}  (1+\ga^2|x|^2)} 
	Q\left(\frac{e^{tM}x}{\cosh(2\gamma t)}\right). 
\ee
We find that  as $t\to +\iy$,
\begin{align*}
\nabla_x u(t,x)\approx & (2e^{-2|\gamma| t})^{\frac{n}{2}} 
(-i\ga x) e^{-\frac{i}{ 2\gamma} (1+\ga^2|x|^2)} 
	Q\left(2e^{-2|\gamma| t} x\right) \\
+& (2e^{-2|\gamma| t})^{1+\frac{n}{2}} 
 e^{\frac{i}{ 2\gamma} (1+\ga^2|x|^2)} 
	(\nabla Q)\left(2e^{-2|\gamma| t} x\right) 
\end{align*}
%since $\ga<0
so that $\norm{\nabla u}_2=\norm{xu}_2= %\frac{|\ga|}{2} \norm{ x Q}_2
O(e^{2|\gamma| t} )$ as $t\to +\iy$.
The morale is that although the repulsive potential can extend the lifespan of a solution,
it may still increase the $\Sigma$-norm of the solution. 
\end{remark}

%\footnote{ For $V_{+\ga}$,  add here what is the profile or behavior for $u(t,x)$ in  the cases: 
%(i) $\ga=\frac1{2\cT}$ and \\
%(ii) $\ga>\frac{1}{2\cT}$ ?}
%\end{document} 

\begin{remark} \label{re:lifespan:gamma}
Summarizing, concerning the blowup solutions $\vphi$ and $u$
we have observed the following on the blowup times  
in the presence of $V_\ga$.  Suppose the (forward) lifespan, or (forward) maximal interval of existence 
 of $\vphi$ is $[0,\cT)$,  $\cT=\cT_{max}$ finite. %We have 
\begin{enumerate}
\item[(a)] If  $\ga>0$,  $u(t)=u(t,\cdot)$ has a shorter lifespan $[0,T^*)$ with $T^*<\mathcal{T}$.
%$\vphi$ to a blow-up $u$ with shorter lifespan
\item[(b)] If $\ga<0$, $|\ga|=-\ga< \frac1{2\cT}$, then $u(t)$ has a longer lifespan
$[0,T^*)$ with $T^*>\mathcal{T}$. 
%R transform changes a blowup $\vphi$ to u
However, if $|\ga|\ge \frac1{2\cT}$, then %when the initial data at or above $\norm{Q_0}_2$,
the solution $u(t)$ exists globally  but its $\dot{H}^1$ norm may grow to infinity as $t\to \iy$.   
%Is this scattering result for $iu_t=-\De u+ V_{-\ga} u$?  cf.\cite{Car05}  
%\item what if $V(x)=-\ga |x|^{2+\eps} ? 
\end{enumerate}
\end{remark}
%\end{document} 

\section{The $\log$-$\log$ blowup rate for rotational NLS} \label{s:loglog}

In this section, we give the proof of Theorem \ref{thm:log-log-u}, namely, the $\log$-$\log$ law for RNLS (\ref{eq:nls_va}), 
where $V_\ga(x)=-\ga^2|x|^2$, $\ga<0$.
%\subsection{Virial identity} 
A priori, we need the  standard virial identity for equation \eqref{eq:nls} in the weighted Sobolev space 
$\Sigma=\mathscr{H}^1$,
which can be proved by a direct calculation, see e.g. \cite{Wein83,BHHZ23t}.
 
\begin{lemma}\label{lem:virial-identity}
Let $\varphi$ be a solution of the equation \eqref{eq:nls} in $C([0, \mathcal{T}), \mathscr{H}^1)$. 
Let $J(t):= \left\| x\varphi(t,\cdot) \right\|_2^2$. Then
	\begin{align*}
	J'(t)
	= 4 \Im \int x \overline{\varphi} \cdot \nabla \varphi dx, \qquad
	J''(t)
	= 8 \mathcal{E}(u_0).
	\end{align*}
\end{lemma}

Now we are ready to prove  Theorem \ref{thm:log-log-u}.

\begin{proof}[Proof of Theorem \ref{thm:log-log-u}]
Let $u \in C([0, T^*), \mathscr{H}^1)$ be the blowup solution of  \eqref{eq:nls_va},
where $[0, T^*)$ is the (forward) maximal interval of existence for $u$. We shall write $T=T^*$ for brevity. 
In view of \eqref{eq:inverse_r_transform},
there is a $\varphi(t,x) \in C\left( \left[ 0, \frac{\tanh(2 \gamma T)}{2 \gamma} \right), \mathscr{H}^1 \right)$
that solves \eqref{eq:nls} with $\vphi_0=u_0$,
where $\left[ 0, \frac{\tanh(2 \gamma T)}{2 \gamma} \right)$ is the (forward) maximal interval of existence for $\vphi$.
According to Theorem \ref{thm:log-log-phi},
the condition $\mathcal{E}(u_0)<0$ suggests that $\varphi$ is a blowup solution.
From \eqref{eq:p_transform},
we have
	\begin{align*}
	u(t,x)
	=\frac{1}{\cosh^\frac{n}{2}(2\gamma t)}
	e^{i\frac{\gamma}{2}|x|^2\tanh(2\gamma t)}
	\varphi\left(\frac{\tanh(2\gamma t)}{2\gamma},\frac{e^{tM}x}{\cosh(2\gamma t)}\right).
	\end{align*}
Then for all $t \in [0, T)$,
we compute
	\begin{align}\label{eq:Du-I1-I2}
	\begin{split}
	\nabla_x u(t,x)
	& = i \gamma x \frac{\sinh(2 \gamma t)}{\cosh^{\frac{n}{2}+1}(2 \gamma t)}
	e^{i \frac{\gamma}{2} |x|^2 \tanh(2 \gamma t)}
	\varphi\left( \frac{\tanh(2 \gamma t)}{2 \gamma}, \frac{e^{tM} x}{\cosh(2 \gamma t)} \right) \\
	& \qquad + \frac{1}{\cosh^{\frac{n}{2}+1}(2 \gamma t)}
	e^{i \frac{\gamma}{2} |x|^2 \tanh(2 \gamma t)}
	(e^{tM})^T \nabla \varphi\left( \frac{\tanh(2 \gamma t)}{2 \gamma}, \frac{e^{tM} x}{\cosh(2 \gamma t)} \right) \\
	& := I_1 + I_2.  
	\end{split}
	\end{align}
For $ I_1 $,
a change of variable gives
	\begin{align*}
	\|I_1\|_{2}
	= \gamma \sinh(2 \gamma t)
	\left\| x \varphi\left( \frac{\tanh(2 \gamma t)}{2 \gamma}, \cdot \right) \right\|_{2}.
	\end{align*}
Since $\sinh(2\gamma t)\leq \sinh(2\gamma T)$,
we just need to see a bound for the remaining part of $\|I_1\|_2$.
Let $J(t) = \| x \varphi(t, \cdot) \|_{2}^2$.
Then \begin{align*}
	J(t)
	= J(0) + J'(0) t + \int_0^t J''(\tau) \,(t - \tau) d\tau.
	\end{align*}
Note that
	\begin{align*}
	|J(0)|
	= \Vert x u_0 \Vert_2^2
	\leq \Vert u_0 \Vert_{\mathscr{H}^1}^2,
	\end{align*}
and by Lemma \ref{lem:virial-identity},
we have
	\begin{align*}
	|J'(0)|
	= \left| 4 \Im \int x \overline{u_0} \cdot \nabla u_0 dx \right|
	\leq 4 \| x u_0 \|_2 \| \nabla u_0 \|_2
	\leq 4 \| u_0 \|_{\mathscr{H}^1}^2 \quad
	\text{and} \quad
	J''(t)
	= 8 \mathcal{E}(u_0).
	\end{align*}
Then we obtain
	\begin{align*}
	\left\Vert x \varphi\left( \frac{\tanh(2 \gamma t)}{2 \gamma}, \cdot \right) \right\Vert_{2}^2
	&= J \left( \frac{\tanh(2 \gamma t)}{2 \gamma} \right) \\
	& \leq |J(0)|
	+ |J'(0)| \frac{\tanh(2 \gamma t)}{2 \gamma}
	+ 4 \mathcal{E}(u_0) \left( \frac{\tanh(2 \gamma t)}{2 \gamma} \right)^2 \\
	& \leq \Vert u_0 \Vert_{\mathscr{H}^1}^2
	+ 4 \|u_0\|_{\mathscr{H}^1}^2 \frac{\tanh(2 \gamma T)}{2 \gamma}
	+ 4 \mathcal{E}(u_0) \left( \frac{\tanh(2 \gamma T)}{2 \gamma} \right)^2.
	\end{align*} 
Thus, for all $t\in [0,T)$
\begin{align}
 \left\Vert x \varphi\left( \frac{\tanh(2 \gamma t)}{2 \gamma}, \cdot \right) \right\Vert_{2}^2
\le C=C(\ga,\norm{u_0}_{\mathcal{H}^1},T) \label{bound:x-vphi}
\end{align} 
and so,
\be
\|I_1\|_2\leq C, \label{e:I2-bound}
\ee
where $C$ is a constant dependent on $\gamma$,
$\|u_0 \|_{\mathscr{H}^1}$,
$\mathcal{E}(u_0)$,
and $T$.

For $I_2 $,
a change of variable gives that
	\begin{align*}
	\Vert I_2 \Vert_{2}
	= \frac{1}{\cosh (2 \gamma t)}
	\left\Vert \nabla \varphi\left( \frac{\tanh(2 \gamma t)}{2 \gamma}, \cdot \right) \right\Vert_{2}.
	\end{align*}
Recall that $\left[ 0, \frac{\tanh(2 \gamma T)}{2 \gamma} \right)$ is the maximal interval of existence
for $\varphi$,
so by Theorem \ref{thm:log-log-phi},
	\begin{align*}
	\lim_{t \rightarrow T}
	\frac{\left\Vert \nabla \varphi \left( \frac{\tanh(2 \gamma t)}{2 \gamma}, \cdot \right) \right\Vert_{2}}
	{\Vert\nabla Q \Vert_{2}}
	\sqrt{ \frac{\frac{\tanh(2 \gamma T)}{2 \gamma} - \frac{\tanh(2 \gamma t)}{2 \gamma}}
	{\log \left| \log \left( \frac{\tanh(2 \gamma T)}{2 \gamma} - \frac{\tanh(2 \gamma t)}{2 \gamma} \right) \right|} }
	= \frac{1}{\sqrt{2 \pi}}.
	\end{align*}
Note that
	\begin{align*}
	\frac{\tanh(2 \gamma T)}{2 \gamma} - \frac{\tanh(2 \gamma t)}{2 \gamma}
	&=\frac{1}{2\gamma}
	\left( \frac{\sinh(2\gamma T)}{\cosh(2\gamma T)}
	-\frac{\sinh(2\gamma t)}{\cosh(2\gamma t)} \right) \\
	&=\frac{\sinh(2\gamma (T - t) )}{2\gamma \cosh(2\gamma T)\cosh(2\gamma t)} \\
	&=\frac{\sinh(2\gamma (T - t) )}{2\gamma(T-t)}
	\cdot \frac{1}{\cosh(2\gamma T)\cosh(2\gamma t)}
	\cdot (T-t).
	\end{align*}
Since $\frac{\sinh \theta}{\theta}\rightarrow 1$ as $\theta\rightarrow 0$,
the above blowup rate can be simplified as
	\begin{align*}
	\lim_{t \rightarrow T}
	\frac{\left\Vert \nabla \varphi \left( \frac{\tanh(2 \gamma t)}{2 \gamma}, \cdot \right) \right\Vert_{2}}
	{\Vert \nabla Q \Vert_{2}}
	\sqrt{ \frac{T - t}{\log \left| \log (T - t) \right|} }
	= \frac{\cosh(2 \gamma T)}{\sqrt{2 \pi}},
	\end{align*}
and this yields
	\begin{align}\label{eq:I2-estimate}
	\lim_{t \rightarrow T}
	\frac{\left\Vert I_2 \right\Vert_{2}}{\Vert \nabla Q\Vert_2}
	\sqrt{ \frac{T - t}{\log \left| \log (T - t) \right|} }
	= \frac{1}{\sqrt{2 \pi}}.
	\end{align}
Therefore,
combining the estimates of $\norm{I_1}_2$ and $\norm{I_2}_2$ in \eqref{e:I2-bound}  
and (\ref{eq:I2-estimate}), 
we obtain from (\ref{eq:Du-I1-I2})
	\begin{align*}
	\lim_{t \rightarrow T}
	\frac{\left\Vert \nabla u \right\Vert_2}{\Vert \nabla Q \Vert_2}
	\sqrt{ \frac{T - t}{\log \left| \log (T - t) \right| } }
	= \frac{1}{\sqrt{2 \pi}}\,.
	\end{align*}
\end{proof}

\section{Limiting behavior: concentration of mass}\label{s:limiting}   
In this section, we show Theorem \ref{thm:limiting_profile}, 
an $L^2$- mass concentration property for RNLS (\ref{eq:nls_va})  
when initial data $u_0$ belongs in the submanifold $\mathcal{S}_{\al^*}\subset \mathcal{H}^1$ as given in Theorem \ref{thm:log-log-u}.  
   The concentration region in (\ref{Ec:w(t):Q}) is  $R_c=\{ |x-x(t)|< w(t) \}$, 
where $w$ can be chosen such that for any $\delta\in (0,\frac12)$
\begin{equation}\label{eq:w(t)}
w(t)= \frac{\sqrt{T-t} }{ (\log |\log (T-t) | )^{\de}}=o( \sqrt{T-t} )\,. 
  \end{equation}

We begin with observing a divergence property for blowup solutions of the RNLS.

\subsection{Divergence in $L^2$ at blowup time} %for RNLS} 
%\edz{add a concentration compactness behavioral result  as in \cite{HLeeZ24damp} ? } 
One application of Theorem \ref{thm:p_transform} is to prove that
the blowup solution to equation \eqref{eq:nls_va} does not converge in $L^2$ as $t\rightarrow T^*=T_{max}$.
Merle and Tsutsumi \cite{MerTsu} first proved such a result
for the solution $\varphi \in C([0,\mathcal{T}), H^1)$ to equation \eqref{eq:nls} with $p=1+\frac{4}{n}$. %\edz{is it a proposition or theorem?}
 More specifically,
they proved that if there is a sequence $\{t_n\}$ such that $t_n \rightarrow \mathcal{T}$
and that $\{ \varphi(t_n) \}$ converges in $L^2$ as $n\rightarrow \infty$,
then $\norm{ \nabla \varphi(t) }_2 \in L^\infty([0,\mathcal{T}))$.
Thus,
if $\varphi$ is the blowup solution at finite time $\mathcal{T}<\infty$,
then $\varphi(t_n)$ does not have an $L^2$ limit for any $t_n \rightarrow \mathcal{T}$.

Since, according to (\ref{eL2:u-vphi}), $\mathcal{R}$-transform ``preserves'' the $L^2$ norm in the sense that
	\begin{align}\label{eq:l2_preserve}
	\norm{ \mathcal{R} \varphi(t) }_2
	=\norm{ \varphi \left( \frac{\tanh(2 \gamma t)}{2 \gamma} \right) }_2 \quad
	\text{and} \quad
	\norm{ \mathcal{R}^{-1} u(t) }_2
	=\norm{ u \left( \frac{\tanh^{-1}(2 \gamma t)}{2 \gamma} \right) }_2,
	\end{align}
we can extend Merle and Tsutsumi's result to equation \eqref{eq:nls_va} as follows.

\begin{prop}\label{t:nonexistence}
Assume that $u \in C([0, T), \mathscr{H}^1)$ is the blowup solution to equation \eqref{eq:nls_va} at finite time 
$T=T^*<\infty$.
Then there does not exist a sequence $\{t_n\}$ such that $t_n \rightarrow T$
and $\{ u(t_n) \}$ converges in $L^2$ as $t_n \rightarrow T$.
\end{prop}

\begin{proof}
Since $u$ is the finite time blowup solution in $C([0, T), \mathscr{H}^1)$,
from \eqref{eq:inverse_r_transform}
we see that $\varphi:= \mathcal{R}^{-1} u$ is the blowup solution to equation \eqref{eq:nls}
in $C\left( \left[0,\frac{\tanh(2\gamma T)}{2\gamma} \right),H^1 \right)$.
If there exists a sequence $\{t_n\}$ such that such that $t_n \rightarrow T$
and $\{ u(t_n) \}$ converges in $L^2$ as $t_n \rightarrow T$,
then by \eqref{eq:l2_preserve},
we also have $\{ \varphi(t_n) \}$ converges in $L^2$ as $t_n \rightarrow \frac{\tanh(2\gamma T)}{2\gamma}$.
However,
this contradicts the non-existence of $L^2$-limit for $\{\varphi(t_n)\}$ in virtue of Proposition 1 in \cite{MerTsu}. 
\end{proof} 

\begin{remark} The divergence statement in Proposition \ref{t:nonexistence} is valid for all blowup profiles of (\ref{eq:nls_va}) 
with any $\norm{u_0}_2$ greater or equal to $\norm{Q}_2$. 
%Further, this theorem   shows  that for any singular solution of (\ref{eq:nls_va}),
%the solution $u(t)=u(t,\cdot)$ does not admit any subsequence that converges in $L^2$ or $\dot{H}^1$ as $t\to T_{max}$. 
%the latter is a default result of the definition of "singular solution"
\end{remark}

\subsection{$L^2$ concentration nearing blowup time $T^*$}\label{ss:L2-concen} 

Recall that the repulsive rNLS is
	\begin{align*}
	i u_t = - \Delta u - \gamma^2 |x|^2 u - |u|^{\frac{4}{n}}u + i A \cdot\nabla u, \qquad
	u(x, 0) = u_0(x).
	\end{align*}
The mass is defined as
	\begin{align*}
	\mathcal{M}(u) := \int_{\mathbb{R}^n} |u|^2 dx
	\end{align*}
and it is conserved.

Define
	\begin{align*}
	\mathcal{E}_0(u)
	:=& \int_{\mathbb{R}^n} \left( |\nabla u(x, t)|^2
	- \frac{n}{2 + n} |u|^{\frac{4}{n} + 2} \right) dx\\
	=& \| \nabla u(t) \|_{L^2}^2
	- \frac{n}{2 + n} \| u(t) \|_{L^{\frac{4}{n} + 2}}^{\frac{4}{n} + 2}.
	\end{align*}
 We have the following lemma by a straightforward computation. 
\begin{lemma}\label{lem:dE0}
\begin{align*}
	\frac{d}{dt} \big( \mathcal{E}_0(u) \big)
	= - 4 \gamma^2 \Im \int_{\mathbb{R}^n} x u \cdot \nabla \overline{u} dx.
	\end{align*}
\end{lemma}

%\begin{proof}  By direct computation ......

Here we will prove a mass concentration result for blow-up solutions,
and then give a proof of the global well-posedness of solutions $u$ when $\| u_0 \|_{L^2} < \| Q \|_{L^2}$.
Our result relies on a refined compactness lemma for the classical $L^2$-critical NLS in \cite{HmiKe}.

\begin{theorem}[\cite{HmiKe}]\label{thm:compactness_lemma}
Let $\{ v_k \}$ be a bounded sequence in $H^1(\mathbb{R}^n)$ such that
	\begin{align}\label{eq:upper_bound}
	\limsup_{k \rightarrow \infty} \| \nabla v_k \|_{L^2} \leq M
	\end{align}
and
	\begin{align}\label{eq:lower_bound}
	\limsup_{k \rightarrow \infty} \| v_k \|_{L^{2 + \frac{4}{n}}} \geq m.
	\end{align}
Then there exists some $V$ in $H^1$ and a sequence $\{ x_k \} \subseteq \mathbb{R}^n$ such that
(up to a subsequence) $v_k (\cdot + x_k) \rightharpoonup V$ weakly in $H^1$ and
	\begin{align*} 
	\| V \|_{L^2} \geq \left( \frac{n}{n + 2} \right)^\frac{n}{4} \frac{m^{\frac{n}{2} + 1}}{M^{\frac{n}{2}}} \| Q \|_{L^2}.
	\end{align*}
\end{theorem}
% \footnote{I checked 09-17 that the theorem is true verbatim, but $V$ should be replaced with other notation due to preceding notation for a potential}

\begin{remark}
This result is sharp in the sense that the lower bound of $\| V \|_{L^2}$ is attained by taking $v_n = Q$.
\end{remark}

We will also use the following lemma from \cite{OhTo09}.
% We include a brief proof for completeness.

\begin{lemma}[\cite{OhTo09}]\label{lem:f_limit}
Let $T \in (0, \infty)$,
and assume that $f: [0, T) \rightarrow \mathbb{R}^+$ is a continuous function.
If $\displaystyle \lim_{t \rightarrow T} f(t) = \infty$,
then there exists a sequence $\{ t_k \}$ in $\subseteq [0, T)$ such that
	\begin{align*}
	t_k \rightarrow T \qquad
	\textup{and} \qquad
	\frac{\displaystyle \int_0^{t_k} f(\tau) \ d\tau}{f(t_k)} \rightarrow 0 \qquad
	\textup{as} \qquad
	k \rightarrow \infty.
	\end{align*}
\end{lemma}

Now we prove the mass concentration result.   

\begin{theorem}\label{thm:limiting_profile}
Suppose that the solution $u$ to the Cauchy problem of the $L^2$-critical RNLS (\ref{eq:nls_va})
blows up at finite time $T^* < \infty$.
Then for any function $w(t)$ satisfying $w(t) \| \nabla u(t) \|_{L^2} \rightarrow \infty$ as $t \rightarrow T^*$,
there exists a function $x(t) \in \mathbb{R}^n$ such that
(up to a subsequence)
	\begin{align}\label{Ec:w(t):Q}
	\liminf_{t \rightarrow T^*} \| u(t) \|_{L^2(\{ |x - x(t)| < w(t) \})}
	\geq \| Q \|_{L^2}.
	\end{align}
\end{theorem}

\begin{proof}
Note that by Lemma \ref{lem:dE0},
we have
	\begin{align}\label{eq:dE0_estimate}
	\left| \frac{d}{dt} \big( \mathcal{E}_0(u) \big) \right|
	= \left| - 4 \gamma^2 \Im \int_{\mathbb{R}^n} x u \cdot \nabla \overline{u} dx \right|
	\lesssim \| xu \|_{L^2} \| \nabla \overline{u} \|_{L^2}
	\lesssim \| \nabla u \|_{L^2}.
	\end{align}
Here the boundedness of 
$ \| xu \|_{2}$ follows from 
 \eqref{eq:xU} and \eqref{bound:x-vphi}.  
Since $\left\| \nabla u(t) \right\|_{L^2} \rightarrow \infty$ as $t \rightarrow T^*$,
by Lemma \ref{lem:f_limit},
there exists a sequence $t_k \rightarrow T^*$ such that
	\begin{align*}
	\frac{\displaystyle \int_0^{t_k} \| \nabla u(\tau) \|_{L^2} d\tau}
	{\| \nabla u(t_k) \|_{L^2}} \rightarrow 0,
	\end{align*}
so by \eqref{eq:dE0_estimate} we have
	\begin{align}\label{eq:integral_control}
	\frac{\displaystyle \left| \int_0^{t_k} \frac{d}{d\tau} \big( \mathcal{E}_0(u(\tau)) \big) d\tau \right|}
	{\left\| \nabla u(t_k) \right\|_{L^2}^2}
	\lesssim \frac{\displaystyle \int_0^{t_k} \| \nabla u(\tau) \|_{L^2} d\tau}
	{\left\| \nabla u(t_k) \right\|_{L^2}}
	\cdot \frac{1}
	{\left\| \nabla u(t_k) \right\|_{L^2}}
	\rightarrow 0.
	\end{align}
Let
	\begin{align*}
	\rho(t) := \frac{\| \nabla Q \|_{L^2}}{\| \nabla u(t) \|_{L^2}} \qquad
	\text{ and } \qquad
	v(t, x) := \rho^{\frac{n}{2}} u(t, \rho x).
	\end{align*}
Then it is easy to check that
	\begin{align*}
	\| v(t) \|_{L^2}
	= \| u(t) \|_{L^2}
	= \| u_0 \|_{L^2} \qquad
	\text{ and } \qquad
	\| \nabla v(t) \|_{L^2}
	= \rho \| \nabla u(t) \|_{L^2}
	= \| \nabla Q \|_{L^2}.
	\end{align*}
Now set $\rho_k := \rho(t_k)$ and $v_k(x) := v(t_k, x)$.
Then we have checked that $\{ v_k \}$ is bounded in $H^1$,
and inequality \eqref{eq:upper_bound} is satisfied with $M = \| \nabla Q \|_{L^2}$.
To check inequality \eqref{eq:lower_bound},
first we integrate the equality in Lemma \ref{lem:dE0} on $[0, t)$ and get
	\begin{align*}
	\mathcal{E}_0(u)
	= \mathcal{E}_0(u_0)
	+ \int_0^t \frac{d}{d\tau} \big( \mathcal{E}_0(u(\tau)) \big) d\tau.
	\end{align*}
By \eqref{eq:integral_control},
we have
	\begin{align*}
	\mathcal{E}_0(v_k)
	&= \| \nabla v_k \|_{L^2}^2
	- \frac{n}{2 + n} \| v_k \|_{L^{\frac{4}{n} + 2}}^{\frac{4}{n} + 2} \\
	&= \rho_k^2 \| \nabla u(t_k) \|_{L^2}^2
	- \frac{n}{2 + n} \left( \rho_k^2 \| u(t_k) \|_{L^{\frac{4}{n} + 2}}^{\frac{4}{n} + 2} \right) \\
	&= \rho_k^2 \, \mathcal{E}_0( u(t_k) ) \\
	&= \rho_k^2 \, \mathcal{E}_0(u_0)
	+ \rho_k^2 \, \int_0^{t_k} \frac{d}{d\tau} \big( \mathcal{E}_0(u(\tau)) \big) d\tau \\
	&= \frac{\| \nabla Q \|_{L^2}^2 \mathcal{E}_0(u_0)}{\| \nabla u(t_k) \|_{L^2}^2}
	+ \| \nabla Q \|_{L^2}^2  \frac{\displaystyle \int_0^{t_k} \frac{d}{d\tau} \big( \mathcal{E}_0(u(\tau)) \big) d\tau}{\| \nabla u(t_k) \|_{L^2}^2}
	\rightarrow 0 \qquad
	\textup{as } k \rightarrow \infty,
	\end{align*}
or equivalently,
	\begin{align*}
	\| v_k \|_{L^{\frac{4}{n} + 2}}^{\frac{4}{n} + 2}
	\rightarrow \frac{2 + n}{n} \| \nabla v_k \|_{L^2}^2
	= \frac{2 + n}{n} \| \nabla Q \|_{L^2}^2,
	\end{align*}
so inequality \eqref{eq:lower_bound} is also satisfied
with $m = \left( \frac{2 + n}{n} \| \nabla Q \|_{L^2}^2 \right)^\frac{n}{4 + 2n}$.
Thus,
by Theorem \ref{thm:compactness_lemma},
there exists a sequence $\{ x_k \} \subseteq \mathbb{R}^n$ such that
	\begin{align}\label{eq:weak_convergence}
	\rho_k^\frac{n}{2} u(t_k, \rho_k \cdot + x_k) \rightharpoonup V
	\end{align}
weakly in $H^1$,
and
	\begin{align*}
	\| V \|_{L^2}
	\geq \left( \frac{n}{n + 2} \right)^\frac{n}{4} \frac{m^{\frac{n}{2} + 1}}{M^{\frac{n}{2}}} \| Q \|_{L^2}
	= \left( \frac{n}{n + 2} \right)^\frac{n}{4}
	\frac{\left( \frac{2 + n}{n} \| \nabla Q \|_{L^2}^2 \right)^{\frac{n}{4}}}
	{\| \nabla Q \|_{L^2}^{\frac{n}{2}}} \| Q \|_{L^2}
	= \| Q \|_{L^2}.
	\end{align*}
By \eqref{eq:weak_convergence},
for every $R > 0$,
there is
	\begin{align*}
	\liminf_{k \rightarrow \infty} \int_{|x| \leq R} \rho_k^n \left| u(t_k, \rho_k x + x_k) \right|^2 dx
	\geq \int_{|x| \leq R} \left| V \right|^2 dx,
	\end{align*}
or equivalently,
	\begin{align*}
	\liminf_{k \rightarrow \infty} \int_{|x - x_k| \leq \rho_k R} \left| u(t_k, x) \right|^2 dx
	\geq \int_{|x| \leq R} \left| V \right|^2 dx.
	\end{align*}
If $w(t)$ is a function satisfying $w(t) \| \nabla u(t) \|_{L^2} \rightarrow \infty$ as $t \rightarrow T^*$,
i.e.,
$w(t_k)/\rho_k \rightarrow \infty$ as $k \rightarrow \infty$,
then for sufficiently large $k$ there is $w(t_k) \geq \rho_k R$.
Hence,
	\begin{align*}
	\int_{|x| \leq R} \left| V \right|^2 dx
	\leq \liminf_{k \rightarrow \infty} \sup_{y \in \mathbb{R}^n}
	\int_{|x - y| \leq \rho_k R} \left| u(t_k, x) \right|^2 dx
	\leq \liminf_{k \rightarrow \infty} \sup_{y \in \mathbb{R}^n}
	\int_{|x - y| \leq w(t_k)} \left| u(t_k, x) \right|^2 dx.
	\end{align*}
The right-hand side of the above inequality is independent of $R$,
so
	\begin{align*}
	\int_{\mathbb{R}^n} \left| V \right|^2 dx
	\leq \liminf_{k \rightarrow \infty} \sup_{y \in \mathbb{R}^n}
	\int_{|x - y| \leq w(t_k)} \left| u(t_k, x) \right|^2 dx.
	\end{align*}
For each fixed $t \in [0, T^*)$,
the above supremum in $y$ is attained at some $x(t) \in \mathbb{R}^n$ because the function
	\begin{align*}
	y \mapsto \int_{|x - y| \leq w(t)} \left| u(t, x) \right|^2 dx
	\end{align*}
is continuous and vanishes at infinity.
With this $x(t)$,
we have
	\begin{align*}
	\liminf_{k \rightarrow \infty} \int_{|x - x(t_k)| \leq w(t_k)} \left| u(t_k, x) \right|^2 dx
	\geq \int_{\mathbb{R}^n} \left| V \right|^2 dx
	\geq \| Q \|_{L^2}^2.
	\end{align*}
This completes the proof.
\end{proof}

With Theorem \ref{thm:limiting_profile},
we can give an alternative proof of 
 the  global existence  for RNLS (\ref{eq:nls_va}).
  
\begin{corollary}
If $\| u_0 \|_{L^2} < \| Q \|_{L^2}$,
then the solution $u$ of (\ref{eq:nls_va}) exists globally in time.
\end{corollary}

\begin{proof}
If $u$ blows up at finite time $T^* < \infty$,
then by Theorem \ref{thm:limiting_profile} and the conservation of mass,
we have
	\begin{align*}
	\| u_0 \|_{L^2}
	< \| Q \|_{L^2}
	\leq 
	\liminf_{t \rightarrow T^*} \| u(t) \|_{L^2 \big( |x - x(t)| < w(t) \big)}
	\leq \| u_0 \|_{L^2},
	\end{align*}
a contradiction.
\end{proof}

%{\color{blue}
\section{Numerical simulations in $\R^{1+2}$} \label{s:numerics} 
 Theorem \ref{thm:log-log-u}  shows that the ``log-log" blow-up dynamics exists for RNLS \eqref{eq:nls_va}
 when $u_0$ is in a region close to the ground state $Q$, i.e.,  
\[
u_0\in \mathcal{S}_{\al^*}= \{\phi\in \Sigma:\mathcal{E}(\phi)<0\}\cap \lbrace \|Q\|_2 {<} \|\phi \|_2 < \|Q\|_2+\alpha^* \rbrace.
\]
In this section, we perform  numerical simulations for the  p.d.e. (\ref{NLS:numerical}) on $\R^2$
in order to provide experimental verification on the blowup speed as well as the limiting behavior
for the solution  $u=u(t,x,y)$ nearing the singularity time.  % {\ga}=\pm 1 
 Let $p=3$,  $V_\ga(x,y)=\sgn(\ga)\ga^2 (x^2+y^2)$ and  $A=0$. 
 we take on the following equation %and n=2 %\edz{Is it possible letting $\td{V}$ be more general potential?}
\begin{align}\label{NLS:numerical} 
iu_t=-\frac{1}{2}\Delta u +V_{\gamma}  u-|u|^{p-1}u+iA\cdot \nabla u\,.
\end{align} 
When ${\gamma}=1$, it is the attractive case;  when ${\gamma}=-1$, it is the repulsive case.
The condition $A=0$ means it is rotation free. 
Our algorithm can cover
the quadratic potential case including the anisotropic  
 $V(x,y)=\pm\ga_1^2x^2\pm\ga_2^2y^2$. %for \ga_1$ and $\ga_2$ will not differ too much.
Due to technical reason, the algorithm does not extend to the rational case $A\ne 0$,
although the same result is anticipated to be valid  as well. 
%\footnote{Is it possible letting $\td{V}$ be more general potential? Yes, we can also check $V=(a_1x^2+a_2y^2)$, but only for $a_1$ and $a_2$ will not differ too much. Otherwise, we cannot find a stable mesh generating function and this numerical algorithm will fail.

The numerics shows that for  generic data above the ground state, that is, $\|u_0\|_2 > \|Q\|_2+ \eps$ 
(e.g., Gaussian type one bump data), even starting with the non-radial data, the blow-up solution will converge to the rescaled ground state $Q_\lam(x):= 
 \lam\inv Q(\lam\inv x)$ 
with radial symmetry. This is true for both  potential cases $\ga=\pm 1$. 
The blow-up rate is $\| \nabla u\|_{L_x^2} \sim (T-t)^{-\frac{1}{2}}$, where $T$ is the blow-up time. However, we cannot observe the ``log-log" correction from our numerical simulation so far. 
%\footnote{ our  simulation shows  beside the ground state type initial 
%$u_0=cQ$ ($c>1$), for the generic data $\|u_0\|_2 > \|Q\|_2 + \epsilon^2 
%(e.g., Gaussian type one bump data), even starting with the non-radial data, the blowup solution will 
%converge to the rescaled ground state $Q_c$ ($Q_c=cQ(\sqrt{c}x)$) with radial symmetry. 
%This is true for both of the attractive and the repulsive harmonic potential 
% blup  \| \nabla u\|_{L_x^2} \sim (T-t)^{-\frac{1}{2}}  T  the blow-up   cannot observe the ``log-log" correction from our direct numerical simulation so far 

The iterative grid redistribution method from \cite{RW1999} for the spatial discretization 
could be applied here. In brief, we consider the mapping $x=x(\xi,\eta)$, and $y=y(\xi,\eta)$, where $\xi$ and $\eta$ are uniform mesh (usually called the computational domain), and $x$ and $y$ are the coordinates in our physical space $\mathbb{R}^2$ (physical domain). The mapping $x=x(\xi,\eta)$ and $y=y(\xi,\eta)$ are determined from the mesh generating PDE with the weighted function 
$$w=\sqrt{1+\| \nabla u \|_2^2/\|u\|_{\infty}^2+ \| \Delta u \|_2 / \|u\|_{\infty}^2}$$ (we refer interested readers to (ref) for more details). As a result, the solution $u(\xi,\eta,t)$ can always keep a good shape ($\|\nabla_{\xi,\eta} u(\xi,\eta)\|_2 \leq C\|u(\xi,\eta)\|_{\infty}$ for some given constant $C$, e.g., $C\leq  5$) at any time $t=t_m$. Thus, we can simulate the numerical results when the solution is very close to the singularity, e.g., $\|u\|_{\infty} \sim 10^{12}$.

We use the 4th order central difference method on the computational variables $(\xi,\eta)$
for the spatial discretization, and the standard 4th order Runge-Kutta (RK4) for the time integration. We take the adaptive time step from $t_m$ to $t_{m+1}$ to be $\Delta t_m=\frac{1}{\| U_m \|_{\infty}^{p-1}}$, where $U_m\approx u(t_m,x(\xi,\eta),y(\xi,\eta) )$
 is our numerical solution of \eqref{NLS:numerical} at $t=t_m$. We denote $L_m=\frac{1}{\|\nabla U_m\|_2}$ to track the blow-up rate. We assume the time we end our simulation $t_{\mbox{end}}$ to be the blow-up time $T$. Then, note that $T=\sum_{j=0}^{M-1} \Delta t_j$, and consequently, the quantity $T-t$ at each time step $t_j$ can be obtained by $T-t_j=\sum_{k=j}^{M-1} \Delta t_k$.  

We take the initial data $u_0=5\exp(-(2x)^2-y^2)$ for both the repulsive and attractive cases. We first show the numerical results for the repulsive case in 
Figure \ref{NLS repulsive blow up}. %shows the solution for the repulsive case.
The top left subplot shows our numerical solution at $t \approx 0.10343$. We take this time as the approximation of the blow-up time $T$. The top right plot shows the zoom of the solution. The bottom left plot shows the solution $|u(\xi,\eta)|$ on the computational domain $(\xi,\eta)$. It shows the reasonable profile ($\|\nabla_{\xi,\eta} u(\xi,\eta)\|_2 \leq C\|u(\xi,\eta)\|_{\infty}$), which indicates the effectiveness of our scheme. The bottom right subplot shows how the physical coordinate $(x,y)$ distributes at $t\approx 0.10343$. One can see it is concentrating near the origin where the blow-up phenomenon occurs, which is as expected.

\begin{figure}
%\begin{center}
\includegraphics[width=0.48\textwidth]{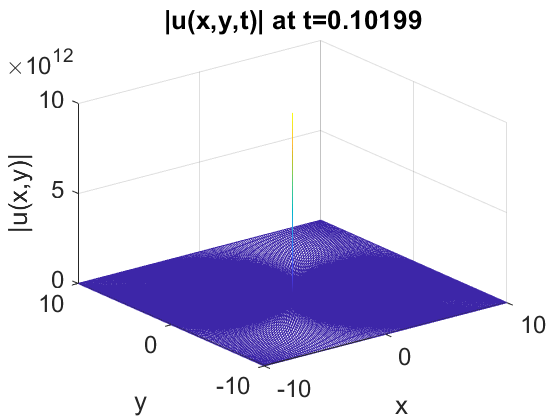}
\includegraphics[width=0.48\textwidth]{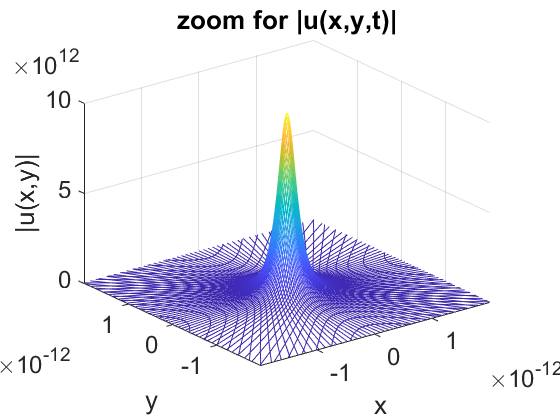}
\includegraphics[width=0.48\textwidth]{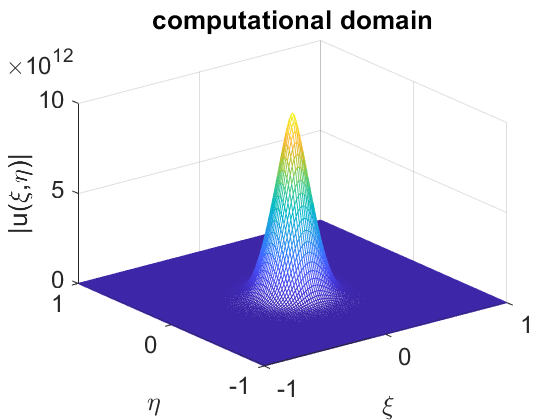}
\includegraphics[width=0.48\textwidth]{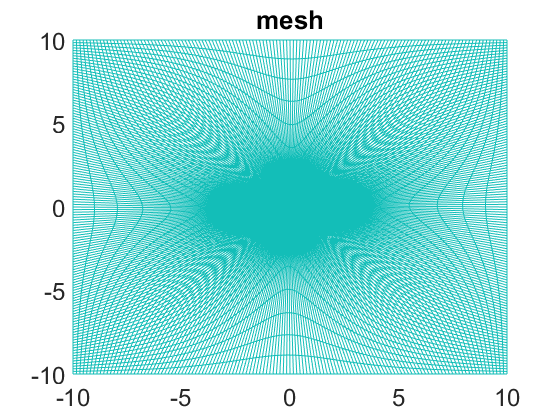}
\caption{\label{NLS repulsive blow up} Solution for $u_0=5\exp(-(2x)^2-y^2)$ at $t\approx 0.10343$ for the repulsive harmonic potential case $\td{\ga}=-1$. Top left: $|u(x,y)|$; top right: zoom for $|u(x,y)|$; bottom left: solution $|u(\xi,\eta)|$ on the computational domain. bottom right: the mesh profile.}
%\end{center}
\end{figure} 
% \footnote{$u_0=5\exp(-4x^2-y^2)$ is correct} 

Figure \ref{NLS repulsive profile} shows more details about the blow-up solution structures. The left subplot shows that the solution eventually converges to the rescaled ground state profile {$Q_\lam(r)$} 
with radial symmetry. The right plot tracks the quantity $T-t$ v.s $L(t)$ in log scale. One can see that $L(t) \sim \sqrt{T-t}$, indicating the blow-up rate is $\frac{1}{2}$. Moreover, we can see that the slop is $\approx 0.50201$, which is similar to the case in 
\cite{YangRouZh18, MRRY2021, MRY2021} %references for the NLS with radial symmetry and 1d case with stochastic noise) 
obtained from the dynamic rescaling method and the adaptive mesh refinement method. This indicates the possible existence of the ``log-log" correction term in our rotational NLS case as well. Indeed, consider the rescaling $u_{\lambda}=\lambda u(\lam^2t, \lambda x, \lambda y)$, putting it into \eqref{NLS:numerical} yields
\begin{align*}
i \pa_tu = -\frac{1}{2}\Delta u + \frac{1}{\lambda^2} \td{\gamma} \td{V}u +|u|^{p-1}u - \frac{i}{\lambda} A \cdot \nabla u\,. 
\end{align*} %\edz{check the functions in this eqn. Corrected}
Heuristically, when $\lambda \rightarrow \infty$, the harmonic term and the rotation term tend to $0$, and thus, play no roles in the blow-up solution structure.

\begin{figure}
%\begin{center}
\includegraphics[width=0.48\textwidth]{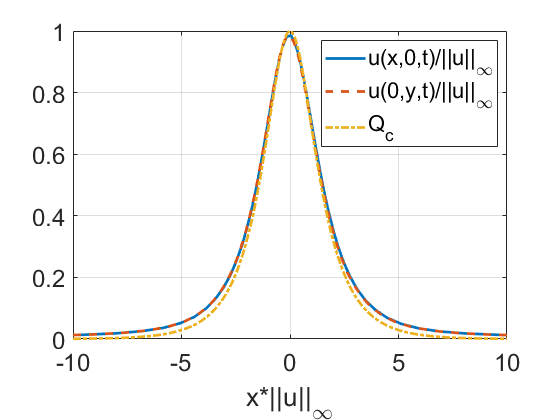}
\includegraphics[width=0.48\textwidth]{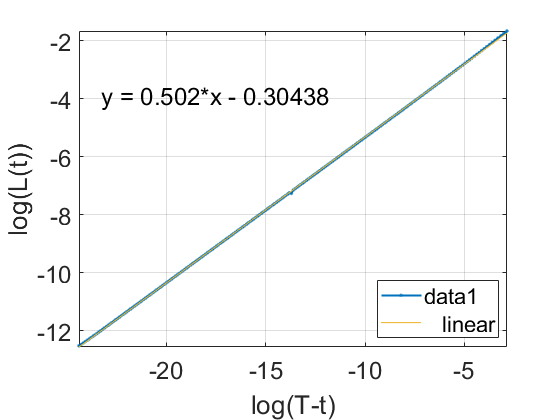}
\caption{\label{NLS repulsive profile} Solution profile for $u_0=5\exp(-(2x)^2-y^2)$ at $t\approx 0.10343$ for the repulsive harmonic potential case. Left: $|u(t,x,0)|$ and $|u(t,0,y)|$ and rescaled $Q_c$. Right: $L(t)=\|\nabla u \|_2$ on the log scale.}
%\end{center}
\end{figure}

%{\color{red} We track the quantity $a=LL_t$. Let $\frac{dtau}{dt}=\frac{1}{L^2}$ we need $a(\tau)$ decay as $1/\ln(\tau)$ to show the numerical evidence of the ``log-log" correction. However, I cannot obtain this in Figure \ref{NLS repulsive a} so far and I need to investigate more on the numerics. 
%
%\begin{figure}
%%\begin{center}
%\includegraphics[width=0.48\textwidth]{RNLS_re_aa.png}
%\includegraphics[width=0.48\textwidth]{RNLS_re_aa.png}
%\caption{\label{NLS repulsive a} Left: $a(\tau)$. Right (to be corrected): $a(\tau)$ v.s. $1/\ln(\tau)$.}
%%\end{center}
%\end{figure}

The attractive case is absolutely similar to the repulsive case, see Figure \ref{NLS attractive blow up} for the blow-up profiles, and Figure \ref{NLS attractive profile} for more details. We also find that the blow-up time $T$ for the attractive case is shorter than the repulsive case (see $T \approx 0.10199$ for the attractive case v.s $T \approx 0.10343$ for the repulsive case). Indeed, this tells us that the attractive harmonic potential speed up the concentration of the solutions, while the repulsive potential slows it down. On the other hand, once the solution starts to blow up, the dynamics will be almost the same.
This also agrees with the numerical results by putting the stochastic noise during the blow-up process for the 1D $L^2$-critical cases in \cite{MRRY2021, MRY2021}. 

\begin{figure}
%\begin{center}
\includegraphics[width=0.48\textwidth]{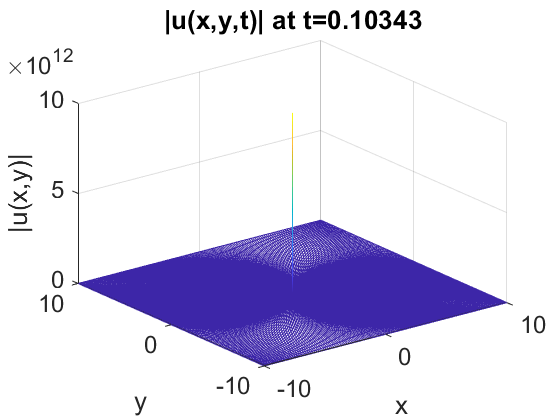}
\includegraphics[width=0.48\textwidth]{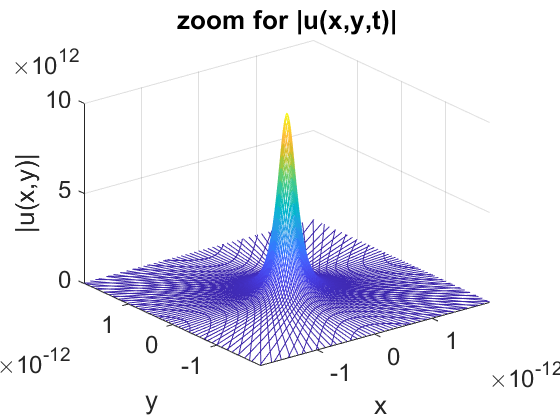}
\includegraphics[width=0.48\textwidth]{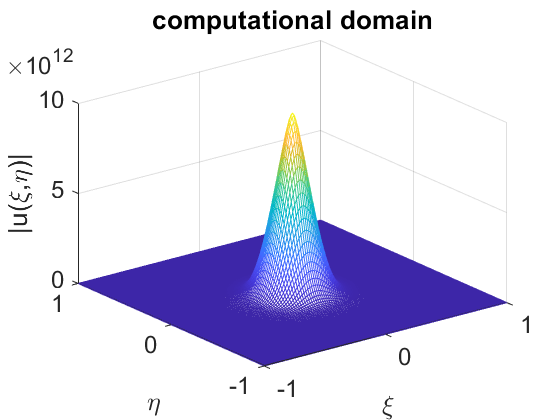}
\includegraphics[width=0.48\textwidth]{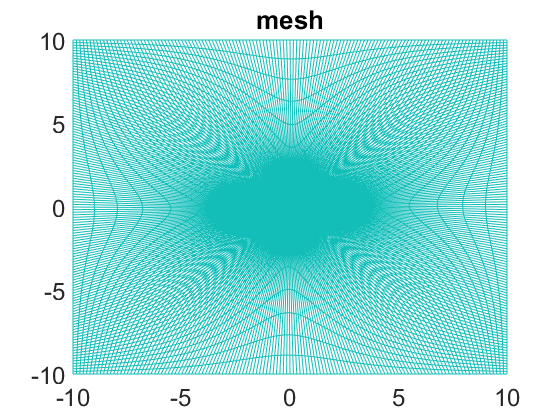}
\caption{\label{NLS attractive blow up} Solution for $u_0=5\exp(-(2x)^2-y^2)$ at $t\approx 0.10199$ for the attractive harmonic potential case $\td{\ga}=1$. Top left: $|u(x,y)|$; top right: zoom for $|u(x,y)|$; bottom left: solution $|u(\xi,\eta)|$ on the computational domain. bottom right: the mesh profile.}
%\end{center}
\end{figure}

\begin{figure}
\includegraphics[width=0.48\textwidth]{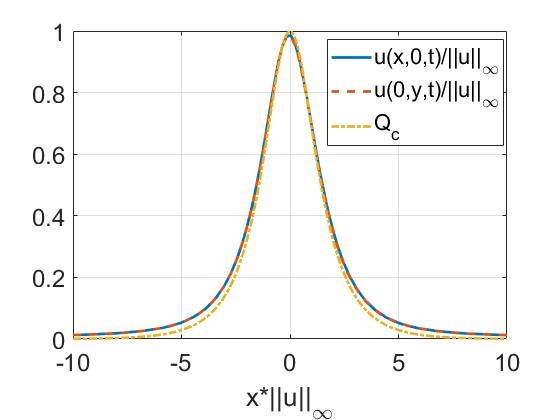}
\includegraphics[width=0.48\textwidth]{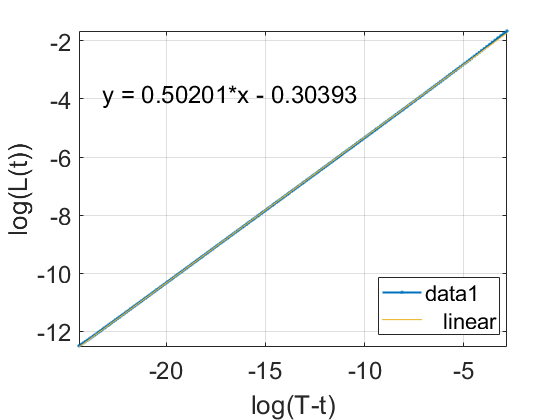}
\caption{\label{NLS attractive profile} Solution profile for $u_0=5\exp(-(2x)^2-y^2)$ at $t\approx 0.10199$ for the attractive harmonic potential case. Left: $|u(t,x,0)|$ and $|u(t,0,y)|$ and rescaled $Q_c$. Right: $L(t)=\|\nabla u \|_2$ on the log scale.}
\end{figure}
%\begin{figure}
%%\begin{center}
%\includegraphics[width=0.48\textwidth]{RNLS_at_aa.png}
%\includegraphics[width=0.48\textwidth]{RNLS_at_aa.png}
%\caption{\label{NLS attractive a} Left: $a(\tau)$. Right (to be corrected): $a(\tau)$ v.s. $1/\ln(\tau)$.}
%%\end{center}
%\end{figure}
\clearpage
%\pagebreak


\begin{thebibliography}{99}
%\bibitem{Af06} A. Aftalion, Vortices in Bose-Einstein condensates.  {\em Progress in Nonlinear Differential Equations and their Applications} {\bf 67}, Birkh\"auser,  2006.

\bibitem{Aftalion2009}  Amandine Aftalion, Xavier Blanc, and Nicolas Lerner,  Fast rotating condensates in an asymmetric harmonic trap.  Phys. Rev. A 79, 011603(R),  2009. %\href{https://journals.aps.org/pra/abstract/10.1103/PhysRevA.79.011603}{journal} 

\iffalse\bibitem{Adh2020}
S.K. Adhikari,  Symmetry-breaking vortex-lattice of a binary superfluid in a rotating bucket,
Physics Letters A, Volume 384, Issue 4,  2020, 126105,
%ISSN 0375-9601,
\href{https://doi.org/10.1016/j.physleta.2019.126105}{doi}.

 %study spontaneous-symmetry-broken phase-separated vortex lattice in a weakly interacting uniform rapidly rotating binary Bose superfluid contained in a quasi-two-dimensional circular or square bucket. For the inter-species repulsion above a critical value, the two superfluid components separate and form a demixed phase with practically no overlap in the vortex lattices of the two components, which will permit an efficient experimental observation of such vortices and study their properties. In case of a circular bucket with equal intra-species energies of the two components, the two components separate into two non-overlapping semicircular domains for all frequencies of rotation $\Om$ generating distinct demixed vortex lattices. In case of a binary Bose superfluid in both circular and square buckets, (a) the number of vortices increases linearly with $\Om$ in agreement with a suggestion by Feynman, and (b) the rotational energy in the rotating frame decreases quadratically with $\Om$ in agreement with a suggestion by Fetter.

Highlights
\begin{enumerate}
\item Phase separated stable vortex lattice found in a uniform binary condensate. 
\item Considers two types of superfluids in rotating square and circular buckets.
\item Efficient numerical scheme used to solve the mean-field model.
\item Non-overlapping vortices in the two components are attractive for new experiments.
\item Establishes parameter domain for phase separation to plan experiments.
\end{enumerate}
\fi


%\bibitem{AMS12} P. Antonelli, D. Marahrens,  C. Sparber, On the Cauchy problem for nonlinear Schr\"odinger equations with rotation.\textit{Discrete Contin. Dyn. Syst.} \textbf{32} (2012), no. 3, 703--715. 

%\bibitem{ArNenSp18} J. Arbunich, I Nenciu, C Sparber, \href{http://homepages.math.uic.edu/~sparber/Christofs_Webseite/Publications_files/RotationStability.pdf}{Stability and instability properties of rotating Bose-Einstein condensates}.   arXiv preprint arXiv:1809.09236, 2018 - arxiv.org 

% consider the mean-field dynamics of Bose-Einstein condensates in rotating harmonic traps and establish several stability and instability properties for the corresponding solution. We particularly emphasize the difference between the situation in which the trap is..

\bibitem{ALZh24n} 
O. Asipchuk, O., C. Leonard, C., S. Zheng,   Existence and non-existence of ground state solutions for magnetic NLS. In: %Wanduku, D., Zheng, S., Zhou, H., Chen, Z., Sills, A., Agyingi, E. (eds) 
Applied Mathematical Analysis and Computations II.  %SGMC 2021. 
Springer Proceedings in Mathematics \& Statistics, vol. {472} (2024), 319-361.
Springer, Cham. %\href{https://doi.org/10.1007/978-3-031-69710-4_14}{doi}


\bibitem{AHS78a} 
 J. Avron, I. Herbst, B. Simon, Schr\"odinger operators with magnetic fields I, Duke Math. 
J. 45 (1978), 847-883.

%\bibitem{Bani04} Valeria Banica, \href{https://arxiv.org/pdf/math/0401129.pdf}{Remarks on the blow-up for the Schr\"odinger equation with critical mass on a plane domain.} Ann. Sc. Norm. Super. Pisa (5), Vol. III (2004), 139--170.

% concentrate on the analysis of the critical mass blowing-up solutions for the cubic focusing Schr\"odinger equation with Dirichlet boundary conditions, posed on a plane domain. We bound the blow-up rate from below, for bounded and unbounded domains. If the blow-up occurs on the boundary, the blow-up rate is proved to grow faster than $(T-t)\inv$, the expected one. Moreover, we show that blow-up cannot occur on the boundary, under certain geometric conditions on the domain	


%\bibitem{BaCD10} V. Banica, R. Carles, T. Duyckaerts, Minimal blow-up solutions to the mass-critical inhomogeneous NLS equation. Communications in Partial Differential Equations 36 (3), 2010, 487--531.


%\bibitem{BanVisc15} Valeria Banica, Nicola Visciglia, Scattering for NLS with a delta potential. Apr 13 2015. math.AP
%math.MP arXiv:1504.02640v1  abs. We prove $H^1$ scattering for defocusing NLS with a delta potential and mass-supercritical nonlinearity, hence extending in an inhomogeneous setting the classical 1-D scattering results first proved by Nakanishi in the translation invariant case.

%\bibitem{Bao} W. Bao, \href{http://www.math.nus.edu.sg/~bao/PS/mantova.pdf}{GPE article.}

%\bibitem{BaoCai13} W. Z. Bao and Y. Y. Cai, Mathematical theory and numerical methods for Bose-Einstein condensation, Kinetic and Related Models 6 (2013), 1--135.

%\bibitem{BaoBenCai12} Weizhu Bao, Naoufel Ben Abdallah, and Yongyong Cai, \href{https://doi.org/10.1137/110850451}{Gross-Pitaevskii-Poisson Equations for Dipolar Bose-Einstein Condensate with Anisotropic Confinement}.  SIAM Journal on Mathematical Analysis, 2012, Vol. 44, No. 3 : pp. 1713-1741. \crr

%\bibitem{BaoBCai12} W. Bao, N. Ben Abdallah and Y. Cai, Gross-Pitaevskii-Poisson equations for dipolar Bose-Einstein condensate with anisotropic confinement. SIAM J. Math. Anal. 44 (2012),  1713--1741.

%\bibitem{BenMeSW05} N. Ben Abdallah, F. M\'ehats, C. Schmeiser, and R. M. Weish\"aupl, The nonlinear Schr\"odinger  equation with a strongly anisotropic harmonic potential, SIAM J. Math. Anal. 37 (2005),  pp. 189--199.

%\bibitem{BaoCai15} W. Bao, Y. Cai, %\href{https://doi.org/10.1137/140979241} 
%{Ground states and dynamics of spin-orbit-coupled Bose--Einstein condensates}.  {\em SIAM J.  Appl. Math.} {\bf 75} (2015), no. 2, 492--517.

\bibitem{BaoWaMar05} W. Bao, H. Wang,  P.  Markowich, Ground, symmetric and central vortex states in rotating 
Bose-Einstein condensates. {\em Comm. Math. Sci.} {\bf 3} (2005), 57--88. 


%\bibitem{BaoDuZh06} W. Bao, Q. Du,  Y. Zhang, Dynamics of rotating Bose-Einstein condensates and its efficient and accurate numerical computation. SIAM J. Appl. Math., 66(3):758--786, 2006. doi: 10.1137/050629392.

%\bibitem{Bao} W. Bao, \href{http://www.newton.ac.uk/programmes/HOP/Abstract5/bao.html}{Bao workshop}

%\bibitem{BaoJinM03} 5. W. Bao, S. Jin, and P. A. Markowich, Numerical study of time-splitting spectral discretizationsof nonlinear Schr\"odinger equations in the semiclassical regimes, SIAM J.  Sci. Comput. 25 (2003), no. 1, 27--64.

%\bibitem{BaoWaMar05} Bao, Weizhu, Hanquan Wang, and Peter A. Markowich, Ground, symmetric and central vortex states in rotating Bose-Einstein condensates. Communications in Mathematical Sciences 3.1 (2005): 57--88.

\bibitem{BaFiGa2010s}  G. Baruch, G. Fibich, N. Gavish,
Singular standing-ring solutions of nonlinear partial differential equations, 
Physica D: Nonlinear Phenomena, 
Volume 239, Issues 20-22,  2010, 1968-1983. %ISSN 0167-2789,
%\href{https://doi.org/10.1016/j.physd.2010.07.009}{doi}.
%(\href{https://www.sciencedirect.com/science/article/pii/S0167278910002198}{sciencedirect}) 

% present a general framework for constructing singular solutions of nonlinear evolution equations that become singular on a d-dimensional sphere, where $d>1$. The asymptotic profile and blowup rate of these solutions are the same as those of solutions of the corresponding one-dimensional equation that become singular at a point. We provide a detailed numerical investigation of these new singular solutions for the following equations: The nonlinear Schrodinger equation  with $\sigma>2$, the biharmonic nonlinear Schrodinger equation $p=1+2\sigma$ with $\sigma>4$, the nonlinear heat equation $\psi_t  -\De \psi -|\psi|^{2\sigma}\psi=0$with $\sigma>0$, and the nonlinear biharmonic heat equation $\psi_t + \De^2 \psi -|\psi|^{2\sigma}\psi=0$ with $\sigma>0
%   Biharmonic nls; Nonlin heat equ; Biharmonic nonlin heat eq; Blowup; Ring
%\href{https://arxiv.org/pdf/0907.2016.pdf}

\bibitem{BaFibMan10}
G. Baruch,  G. Fibich, and E. Mandelbaum,  Ring-type singular solutions of the biharmonic nonlinear Schr\"odinger equation.  
 Nonlinearity 23, no. 11 (2010): 2867.

\bibitem{BHZ19a}
N. Basharat, Y Hu and S. Zheng, Blowup rate for mass critical rotational nonlinear Schr\"odinger equations.
Contemporary Mathematics, Volume 725, 2019, 1-12.

\bibitem{BHHZ23t}
N. Basharat, H. Hajaiej, Y. Hu,  S. Zheng, Threshold for blowup and stability for nonlinear Schr\"odinger equation with rotation.  Annales Henri Poincar\'e 24 (4), (2023), 1377-1416.  
%\href{https://arxiv.org/abs/2002.04722}{arxiv}.  

%\bibitem{BW98}  J. Bourgain, W. Wang, Construction of blowup solutions for the nonlinear Schr\"odinger equation with critical nonlinearity. {\em Ann. Scer. Norm. Sup.  Pisa Cl. Sci}.  {\bf 25} (1997),  197--215.   %Issue: 1-2, page 197-215


\bibitem{BozGhMasmYang24}  F. Bozgan, T.-E. Ghoul, N. Masmoudi, K. Yang, 
Blow-up dynamics for the $L^2$ critical case of the 2D Zakharov-Kuznetsov equation. 2024 
\href{https://arxiv.org/abs/2406.06568}{arxiv} 


% \bibitem{CaSqu08} M. Caliari and M. Squassina, Location and phase segregation of ground and excited states for 2D Gross-Pitaevskii systems,Dyn. Partial Differ. Equ. (2008).  \href{https://profs.scienze.univr.it/caliari/schroedinger.htm}{web}


\iffalse \bibitem{CariSqua10} Marco Caliari and  M. Squassina, 
 Numerical computation of soliton dynamics for NLS equations in a driving potential. 
 Electronic Journal of Differential Equations, Vol. 2010(2010), No. 89, pp. 1-12.
% http://ejde.math.txstate.edu or http://ejde.math.unt.edu
%ftp ejde.math.txstate.edu
  \href{https://arxiv.org/abs/0908.3648v2}{arxiv}

%arXiv: Numerical Anal   provide some numerical computations for the soliton dynamics of the nonlinear Schr\"odinger equation with an external potential. After computing the ground state solution $r$ of a related elliptic equation we show that, in the semi-classical regime, the center of mass of the solution with initial datum modelled on $r$ is driven by the solution of a Newtonian type law. Finally, we provide some examples and analyze the numerical errors in the two dimensional case when $V$ is an harmonic potential

%\href{https://www.google.com/search?sca_esv=8be39d41f3b8291e&rlz=1C5GCCM_en&q=zucchero+splitting+method+math&spell=1&sa=X&ved=2ahUKEwj9ov-c9P6NAxXxAHkGHSSsIL4QBSgAegQIEBAB&biw=1600&bih=838&dpr=1.8#fpstate=ive&vld=cid:146fe182,vid:b93IjMHrrjg,st:0}{more about splitting method}
 
 \bibitem{CaOsPi17} 
 M. Caliari, A. Ostermann, C. Piazzola, 
A splitting approach for the magnetic Schr\"odinger equation, 
Journal of Computational and Applied Mathematics, 
Volume 316,  2017, Pages 74-85,
\href{https://doi.org/10.1016/j.cam.2016.08.041}{doi}. 
(\href{https://www.sciencedirect.com/science/article/pii/S0377042716304101}{science})

 The Schrdinger equation in the presence of an external electromagnetic field is an important problem in computational quantum mechanics. It also provides a nice example of a differential equation whose flow can be split with benefit into three parts. After presenting a splitting approach for three operators with two of them being unbounded, we exemplarily prove first-order convergence of Lie splitting in this framework. The result is then applied to the magnetic Schrdinger equation, which is split into its potential, kinetic and advective parts. The latter requires special treatment in order not to lose the conservation properties of the scheme. We discuss several options. Numerical examples in one, two and three space dimensions show that the method of characteristics coupled with a nonequispaced fast Fourier transform (NFFT) provides a fast and reliable technique for achieving mass conservation at the discrete level.
 Magnetic  Exponential splitting methods; Convergence; Fourier techniques; Nonequispaced fast Fourier transform
\fi
 
% \bibitem{CariZu21} Marco Caliari and Simone Zuccher,  A Fast Time Splitting Finite Difference Approach toGross-Pitaevskii Equations.  Commun. Comput. Phys.\href{doi:10.4208/cicp.OA-2020-0131}{doi}  Vol. 29, No. 5, pp. 1336-1364, 2021.   \href{https://profs.scienze.univr.it/zuccher/downloads/CZ_CICP2021.pdf}{pdf}
% Department of Computer Science, University of Verona, Italy.
%Rec 9 July 2020; Accepted (in revised version) 26 October 2020
 
 \bibitem{Car02c} R. Carles,  Critical nonlinear Schr\"odinger equations with and without harmonic potential,
 \textit{Math. Models Methods Appl. Sci.} \textbf{12} (2002), 1513--1523.


\bibitem{Car03re} R. Carles,  
Nonlinear Schr\"odinger equations with repulsive harmonic potential and applications,
\textit{SIAM J. Math. Anal.} \textbf{35} (2003), no. 4, 823--843. 
%\href{https://epubs.siam.org/doi/10.1137/S0036141002416936}{doi}

%\bibitem{Car}  R. Carles, Geometric optics and instability for semi-classical Schr?odinger equations, preprint, arXiv:math.AP/0505468, 2005.

\bibitem{Car2011t} R. Carles, %\href{https://hal.archives-ouvertes.fr/hal-00426530/en/}
{Nonlinear Schr\"odinger equation with time dependent potential.}  Commun. Math. Sci. 9 (2011), no. 4, 937--964. 

% prove a global well-posedness result for defocusing nonlinear Schrodinger equations with time dependent potential. We then focus on time dependent harmonic potentials. This aspect is motivated by Physics (Bose-Einstein condensation), and appears also as a preparation for the analysis of the propagation of wave packets in a nonlinear context. The main aspect considered here is the growth of high Sobolev norms of the solution.

%NONLINEAR SCHR\"oDINGER EQUATION WITH TIME DEPENDENT POTENTIAL. arXiv 2009.
%\textcolor{blue}{generalized  Lens transform} fundamental solution, Fujiuara-Yajima-Feynman path integral, time-dependent repulsive $V$. 

%\bibitem{CarSi15} R. Carles, \href{https://hal.archives-ouvertes.fr/hal-00823573/en/}{Large time behavior in nonlinear Schr\"odinger equation with time dependent potential.}   With Jorge Drumond Silva. Commun. Math. Sci. 13 (2015), no. 2, 443--460.
	
	
	
%\bibitem{Cartan1899} \href{https://en.wikipedia.org/wiki/Differential_form}{\'Elie Cartan} (1899), Sur certaines expressions diff\'erentielles et le probl\`eme de Pfaff, Annales scientifiques de l'\'Ecole Normale Sup\'erieure: 239--332.

%\bibitem{CD99} Y. Castin and R. Dum, Bose-Einstein condensates with vortices in rotating traps,Eur. Phys. J. D {\bf 7}, 399 (1999).

\bibitem{Cassano2016} B. Cassano, L. Fanelli, Gaussian decay of harmonic oscillators and related models,
Journal of Mathematical Analysis and Applications,
Vol. 456, No.1, 2017,  214-228.    %Cassano and Fanelli. 2016  lens transform, e^{tM}
%https://doi.org/10.1016/j.jmaa.2017.06.067.
%(https://www.sciencedirect.com/science/article/pii/S0022247X17306297) 
  
%prove that the decay of the eigenfunctions of harmonic oscillators, uniform electric or magnetic fields is not stable under 0-order complex perturbations, even if bounded, of these Hamiltonians, in the sense that we can produce solutions to the evolutionary Schrdinger flows associated to the Hamiltonians, with a stronger Gaussian decay at two distinct times. We then characterize, in a quantitative way, the sharpest possible Gaussian decay of solutions as a function of the oscillation frequency or the strength of the field, depending on the Hamiltonian which is considered. This is connected to the Hardy's Uncertainty Principle for free Schrödinger evolutions.
%Kwd: Schrödinger equation; Uniform electric potentials; Uniform magnetic potentials; Harmonic oscillator; Unique continuation; Uncertainty principle

%\footnote{\href{https://physics.stackexchange.com/questions/540019/heisenberg-picture-harmonic-oscillator-operators}{harmonic oscillator heisenberg picture}\crr 
%$V=E\cdot x$, and extension to E time-dependent;   
%cf. \cite{carles2021logarithmic} and \cite{Car05global,car11timeV}; 
%Cassano: spectral theory for H, lens transf.   for magnetic-quadratic potential 
 %JMAA  Transformations from Linear Schrodinger to $H_{V(t,x)}$ 

%{CaFan17} Biagio Cassano, Luca Fanelli, \href{https://arxiv.org/abs/1603.06738}{Gaussian decay of Harmonic Oscillators and related models}}
%( Mar 2016 (v1), last revised 28 May 2017  v2

% prove that the decay of the eigenfunctions of harmonic oscillators, uniform electric or magnetic fields is not stable under 0-order complex perturbations, even if bounded, of these Hamiltonians, in the sense that we can produce solutions to the evolutionary Schr\"odinger flows associated to the Hamiltonians, with a stronger Gaussian decay at two distinct times. We then characterize, in a quantitative way, the sharpest possible Gaussian decay of solutions as a function of the oscillation frequency or the strength of the field, depending on the Hamiltonian which is considered. This is connected to the Hardy's Uncertainty Principle for free Schr\"odinger evolutions.\edz{\crr has commutator relation $[H_b, L_z]$esp. when $b=0$, has kernel and eigenfunctions\crr}  MSC classes:	35J10, 35B99

%\bibitem{Caz93} T. Cazenave, An introduction to nonlinear Schr\"odinger equations, volume 26 of Text. Met. Mat. Univ. Fed. Rio de Jan., 1993.
 %An Introduction to Nonlinear SchrÛ€inger Equations, 
% (Textos de M\'etodos. Matem\'aticos, Universidade Federal do Rio de Janeiro, 26, third edition)

%characterization of Q via variational method using some functionals.\crr 




%\bibitem{Caz03} T. Cazenave, Semilinear Schr\"odinger equations. Courant Lecture Notes in Mathematics, vol. 10, New York University Courant Institute of Mathematical Sciences, New York, 2003.  

\bibitem{CazE88}  T. Cazenave, M. Esteban,
%\href{https://books.google.com/books?hl=en&lr=&id=vn8C0mwx_ZUC&oi=fnd&pg=PA155&dq=stability+magnetic+schrodinger+equation&ots=ByvdOJaGWz&sig=poiI103lTsWvsy9htIENepASo1k#v=onepage&q=stability%20magnetic%20schrodinger%20equation&f=false}
 {On the stability of stationary states for nonlinear Schr\"odinger equations with an external magnetic field}.  {\em Mat. Apl. Comput}. {\bf 7} (1988), 155--168.

%\bibitem{CazLion82} T. Cazenave, P.-L. Lions, Orbital stability of standing waves for some Schr\"odinger equations.  Comm. Math. Phys. {\bf 85} (1982), 549--561.

%\bibitem{CazW90}  T. Cazenave, F. Weissler, The Cauchy problem for the critical nonlinear Schr\"odinger equation in $H^s$.\textit{Nonlinear Anal.} \textbf{14} (1990), no. 10, 807--836.


%\bibitem{ChMarS16} A.~Cheskidov, D.~Marahrens, C.~Sparber, %\href{https://arxiv.org/abs/1506.04706v3} {Global attractor for a Ginzburg-Landau type model of rotating Bose-Einstein condensates}.   %(Submitted on 15 Jun 2015 (v1), last revised 5 Aug {\em Dyn.  Partial Differ. Eqn}. {\bf 14} (2017),  5--32.
%arXiv. 2016. %(this version, v3

% study the long time behavior of solutions to a nonlinear partial differential equation arising in the description of trapped rotating Bose-Einstein  condensates. The equation can be seen as a hybrid between the well-known nonlinear Schr\"odinger/Gross-Pitaevskii equation and the Ginzburg-Landau equation. We prove existence and uniqueness of global in-time solutions in the physical energy space and establish the existence of a global attractor within the associated dynamics. We also obtain basic structural properties of the attractor and an estimate on its Hausdorff and fractal dimensions	25  some more typos fixed; additional references added

%\bibitem{CaoTang2006} D. Cao, Z. Tang, Existence and uniqueness of multi-bump bound states of nonlinearSchr\"odinger equations with electromagnetic fields, J. Differential Equations 222 (2006), 381-424.

%\bibitem{ChaSzu05} J. Chabrowski, A. Szulkin, On the Schr\"odinger equation involving a critical Sobolevexponent and magnetic field, Topol. Methods Nonlinear Anal., 25 (2005), 3-21.

% S. Cingolani, Positive solutions to perturbed elliptic problems in RN involving criticalSobolev exponent, Nonlinear Anal. 48 (2002), 1165-1178.

\iffalse
\bibitem{Cing03} S. Cingolani, Semiclassical stationary states of Nonlinear Schr\"odinger equations with
an external magnetic field, J. Differential Equations 188 (2003), 52?79

\bibitem{CingJeanTan17} Silvia CINGOLANI, Louis JEANJEAN, Kazunaga TANAKA, Multiple complex valued solutions for nonlinear magnetic Schr\"odinger equations, Journal Fixed Point 
Theory and Applications, vol. 19 (2017), 37-66 (dedicated to Prof. Paul Rabinowitz).

\bibitem{CingPi04}
S. Cingolani, A. Pistoia, Nonexistence of single blow-up solutions for a nonlinear 
Schr\"odinger equation involving critical Sobolev exponent, Z. Angew. Math. Phys. 55
(2004), 201-215.

\bibitem{CingSecchi02} S. Cingolani, S. Secchi, Semiclassical limit for nonlinear Schr\"odinger equations with
electromagnetic fields, J. Math. Anal. Appl. 275 (2002), 108-130.

\bibitem{CingSe2002} Silvia Cingolani, Simone Secchi, 
Semiclassical limit for nonlinear Schr\"odinger equations with electromagnetic fields,
Journal of Mathematical Analysis and Applications,
Volume 275, Issue 1, 2002, Pages 108-130,
\href{https://doi.org/10.1016/S0022-247X(02)00278-0}{doi}.
%(https://www.sciencedirect.com/science/article/pii/S0022247X02002780) 

 study the existence of standing waves for a class of nonlinear Schrdinger equations in Rn, with both an electric and a magnetic field. Under suitable non-degeneracy assumptions on the critical points of an auxiliary function related to the electric field, we prove the existence and the multiplicity of complex-valued solutions in the semiclassical limit. We show that, in the semiclassical limit, the presence of a magnetic field produces a phase in the complex wave, but it does not influence the location of peaks of the modulus of these waves.

%Keywords: Nonlinear Schrödinger equations; Semiclassical limit; Electromagnetic fields; Complex-valued solutions

\bibitem{CingSe05} S. Cingolani, S. Secchi, 
\cob{Semiclassical states for NLS equations with magnetic potentials
having polynomial growths}, J. Math. Phys. 46 (2005), 1-19. \crr 

%Semiclassical states for NLS equations with magnetic potentials having polynomial growthsCingolani, Silvia ; Secchi, Simone

 prove existence of standing wave solutions for a nonlinear Schrödinger equation on R3 under the influence of an external magnetic field B. In particular we deal with the physically meaningful case of a constant magnetic field B=(0,0,b) having source in the potential A(x)=(b?2)(?x2,x1,0) corresponding to the Lorentz gauge.

%Journal of Mathematical Physics, Volume 46, Issue 5, id.053503
 %Pub Date: May 2005 DOI: 10.1063/1.1874333 

% Bibcode:  Kwd: 03.65.Ge; 02.30.-f; 03.65.Sq; 02.10.De; Solutions of wave equations: bound states; Function theory analysis; Semiclassical theories and applications; Algebraic structures and number theory

\bibitem{CingSe} S. Cingolani, S. Secchi, Multipeak solutions for NLS equations with magnetic fields in
semiclassical regime, to appear.

\bibitem{CingSe2018} 
Silvia CINGOLANI, SIMONE SECCHI, Intertwining solutions for magnetic relativistic Hartree type equations, Nonlinearity, vol. 31 (2018), 2294-2318.
\fi

\bibitem{De91}
A. De Bouard,   Nonlinear Schr\"odinger equations with magnetic fields. 
\textit{Differential Integral Equations.} \textbf{4} (1991), no. 1, 73--88.

\bibitem{Darw12a} M. Darwich, Blowup for the damped $L^2$-critical nonlinear Schr\"odinger equation. 
%\href{https://arxiv.org/abs/1101.2763}
{Advances in Differential Equations}. Vol. 17, Numbers 3-4 (2012), 337--367.

% consider the Cauchy problem for the L2-critical damped nonlinear Schr\"odinger equation. We prove existence and stability of finite time blowup dynamics with the log-log blow-up speed for $\norm{\nabla u(t)}_2$.		arXiv:1101.2763 [math.AP]
% 	(or arXiv:1101.2763v3 [math.AP] for this version)

%\bibitem{Darw12b} Mohamad Darwich, \href{https://arxiv.org/abs/1206.6082v4}{On the $L^2$-critical nonlinear Schr\"odinger Equation with a nonlinear damping}    ( Jun 2012 (v1), last revised 15 Jan 2013 consider the Cauchy problem for the L2-critical nonlinear Schr\"{o}dinger equation with a nonlinear damping. According to the power of the damping term, we prove the global existence or the existence of finite time blowup dynamics with the log-log blow-up speed for $\norm{\nabla u(t)}_2  

%\end{thebibliography}\end{document}

\bibitem{DeGangSiWa08}
S. Dejak, Z. Gang, I. M. Sigal, and S. Wang, Blow-up in nonlinear heat equations. 
{\em Advances in Appl. Math.}, {\bf 40}:433-481, 2008. 
%\href{https://www.sciencedirect.com/science/article/pii/S0196885807000784/pdf?md5=2dc3ee20d91c33adb1d9823d84ad1281&pid=1-s2.0-S0196885807000784-main.pdf&_valck=1}{paper}

% study the blow-up of solutions of nonlinear heat equations in dimension 1. We show thatfor an open set of even initial data which are characterized roughly by having maxima at theorigin, the solutions blow up in finite time and at a single point. We find the universal blow-up profile and remainder estimates. Our results extend previous results in several directions and our techniques differ from the techniques previously used for this problem. In particular   they do not rely on maximum principle.

\iffalse\bibitem{Dinh2022e}
Van Duong Dinh, Existence and stability of standing waves for nonlinear Schr\"odinger equations with a critical rotational speed.
2022 \href{arXiv:2201.02682}{arxiv} 	
 \href{https://doi.org/10.48550/arXiv.2201.02682}{doi}
\fi
% study the existence and stability of standing waves associated to the Cauchy problem for NLS with a critical rotational speed and an axially symmetric harmonic potential. This equation arises as an effective model describing the attractive Bose-Einstein condensation in a magnetic trap rotating with an angular velocity. By viewing the equation as NLS with a constant magnetic field and with (or without) a partial harmonic confinement, we establish the existence and orbital stability of prescribed mass standing waves for the equation with mass-subcritical, mass-critical, and mass-supercritical nonlinearities. Our result extends a recent work of [Bellazzini-Boussaid-Jeanjean-Visciglia, Comm. Math. Phys. 353 (2017), no. 1, 229-251], where the existence and stability of standing waves for the supercritical NLS with a partial confinement were established. 

\bibitem{Dinh2024_3Drev} V. D. Dinh,
The 3D nonlinear Schrödinger equation with a constant magnetic field revisited
%Dinh, V.D. 
J Dyn Diff Equat 36, 3643-3686 (2024). 
%\href{https://doi.org/10.1007/s10884-022-10235-1}{doi}
% \href{https://doi.org/10.48550/arXiv.2201.02690}{arxiv} 
 
% \crr If $c>\norm{Q}_2$, or $c=\norm{Q}_2$,  does there exist a g.s.s. for the RNLS when $\Om=\gamma$?  
%cf. also Guo, Yujin et al;  Lewin, et al. in 2D  \crr
 

\bibitem{FanMen20c} C.J. Fan, D. Mendelson,  
Construction of $L^2$ $\log$-$\log$ blowup solutions for the mass critical nonlinear Schr\"odinger equation. 
%\href{arXiv:2010.07821}{arxiv}  \href{https://arxiv.org/abs/2010.07821}{article}

% study the log-log blowup dynamics for the mass critical nonlinear Schrodinger equation on $\R^2$ under rough but  structured random perturbations at L2 regularity. In particular, by employing probabilistic methods, we provide a construction of a family of $L^2(\R^2)$ regularity solutions which do not lie in any $H^s(\R^2)$ for any $s>0$, and which blowup according to the log-log dynamics
	
%\bibitem{CFanK24biV} L. Cossetti, L. Fanelli, D. Krejcirík,  Uniform resolvent estimates and absence of eigenvalues of biharmonic operators with complex potentials. Journal of Functional Analysis 287 (12), 110646, 2024. \href{https://scholar.google.es/scholar?oi=bibs&hl=en&cites=17367308175077216322}{scholar} 

%\bibitem{DanFan08} D'ancona, P. and Fanelli, L. (2008). Strichartz and Smoothing Estimates for Dispersive Equations with Magnetic Potentials. Communications in Partial Differential Equations, 33(6), 1082-1112.  \href{https://doi.org/10.1080/03605300701743749}{doi}

%\bibitem{EsLion89} M. Esteban, P.L. Lions, Stationary solutions of nonlinear Schr\"odinger equations withan external magnetic field, in PDE and Calculus of Variations, in honor of E. De Giorgi,  Birkh\"auser, 1990

\iffalse\bibitem{FanSuWaZZheng24stark+}
Luca Fanelli, Xiaoyan Su, Ying Wang, Junyong Zhang, Jiqiang Zheng, 
Intertwining operators beyond the Stark Effect, 
2024/12/5. \href{https://arxiv.org/pdf/2412.04406}{arxiv}
% arXiv preprint arXiv:2412.04406

The main mathematical manifestation of the Stark effect in quantum mechanics is the shift and the formation of clusters of eigenvalues when a spherical Hamiltonian is perturbed by lower order terms. Understanding this mechanism turned out to be fundamental in the description of the large-time asymptotics of the associated Schr\"odinger groups and can be responsible for the lack of dispersion in Fanelli, Felli, Fontelos and Primo [Comm. Math. Phys., 324(2013), 1033-1067; 337(2015), 1515-1533]. Recently, Miao, Su, and Zheng introduced in [Tran. Amer. Math. Soc., 376(2023), 1739--1797] a family of spectrally projected intertwining operators, reminiscent of the Kato's wave operators, in the case of constant perturbations on the sphere (inverse-square potential), and also proved their boundedness in . Our aim is to establish a general framework in which some suitable intertwining operators can be defined also for non constant spherical perturbations in space dimensions 2 and higher. In addition, we investigate the mapping properties between -spaces of these operators. In 2D, we prove a complete result, for the Schr\"odinger Hamiltonian with a (fixed) magnetic potential an electric potential, both scaling critical, allowing us to prove dispersive estimates, uniform resolvent estimates, and -bounds of Bochner--Riesz means. In higher dimensions, apart from recovering the example of inverse-square potential, we can conjecture a complete result in presence of some symmetries (zonal potentials), and open some interesting spectral problems concerning the asymptotics of eigenfunctions.
\fi

\bibitem{Fetter2009r} A.L. Fetter,  Rotating trapped Bose-Einstein condensates. Rev. Mod. Phys., 81 (2009), p. 647.
%Geballe Laboratory for Advanced Materials, Departments of Physics and Applied Physics, Stanford University, Stanford, CA 94305
%\href{http://www.phys.ufl.edu/~pjh/teaching/phz7429/fetter.pdf}{pdf}

% After reviewing the ideal Bose-Einstein gas in a box and in a harmonic trap, the effect of interactions on the formation of a Bose-Einstein condensate are discussed, along with the dynamics of small-amplitude perturbations (the Bogoliubov equations). When the condensate rotates with angular velocity $\Om$, one or several vortices nucleate, leading to many observable consequences. With more rapid rotation, the vortices form a dense triangular array, and the collective behavior of these vortices has additional experimental implications. For $\Om$ near the radial trap frequency $\om$, the lowest-Landau-level approximation becomes applicable, providing a simple picture of such rapidly rotating condensates. Eventually, as $\Om\to \om$, the rotating dilute gas is expected to undergo a quantum phase transition from a superfluid to various highly correlated (nonsuperfluid) states analogous to those familiar from the fractional quantum Hall effect for electrons in a strong perpendicular magnetic field. %\copyright 2009 The American Physical Society. 


\bibitem{FibMeRa06} G. Fibich, F. Merle, P. Rapha\"el, Proof of a spectral property related to singularity formulation for the $L^2$ critical nonlinear Schr\"odinger equation. Phys D, 220: 1--13 (2006).



%\bibitem{Fra85} G. Fraiman, Asymptotic stability of manifold of self-similar solutions in self-focusing, Sov. Phys. JETP 61 (2) (1985) 228--233.



%\bibitem{FiMe01} G. Fibich and F. Merle. Self-focusing on bounded domains. Phys. D, 155(2001):132?58

\bibitem{Gar12} A. Garcia, Virial id for magnetic NLS and wave equations. 2012. 


%\bibitem{Fro17} J\"urg Fr\"ohlich, \href{https://www.researchgate.net/publication/317279103_Statistical_Mechanics_of_the_Universe}{Statistical mechanics of the Universe}. ETHZ, Conference: Colloquium in astrophysics and cosmology, Affiliation: ETH Zurich, DOI: 10.13140/RG.2.2.21027.86569 (institute for theoretical physics)  Mar.28, 2017.



%\bibitem{Helffer94} B. Helffer, On Spectral Theory for Schr\"odinger Operators with Magnetic Potentials, Advanced Studies in Pure Mathematics vol. 23, 113-141 (1994).

%\bibitem{Hel97} B. Helffer, Semiclassical analysis for Schr\"odinger operator with magnetic wells, in Quasiclassicalmethods (J. Rauch, B. Simon Eds.). The IMA Volumes in Mathematics and its applications vol. 95, Springer-Verlag New-York 1997.

\bibitem{HelMo2001} % bec ,  optics background for magnetic A 
 B. Helffer, A. Morame, Magnetic bottles in connection with superconductivity, J. Functional
Anal. 93 A (2001), 604-680.


\bibitem {HmiKe}
T. Hmidi and S. Keraani,
Blowup theory for the critical nonlinear Schr\"odinger  equations revisited,
\textit{Int. Math. Res. Not.} 46, (2005), 2815--2828.

 
 \bibitem{HLeeZ24damp} Y. Hu, Y. Lee, S. Zheng,  Damped nonlinear Schr\"odinger equation with Stark effect. 
 %In: Manukure, S., Ma, WX. (eds) 
 Nonlinear and Modern Mathematical Physics. NMMP 2022. 
 Springer Proceedings in Mathematics \& Statistics, vol. {459} (2024), pp. 189-205.
 Springer, Cham. %https://doi.org/10.1007/978-3-031-59539-4_7
 
%\bibitem{Kurata2000}  K. Kurata, Existence and semi-classical limit of the least energy solution to a nonlinearSchr\"odinger equation with electromagnetic fields, Nonlinear Anal. 41 (2000), 763?778.

%\bibitem{LuPan2000} K. Lu, X.-B. Pan, Surface nucleation of superconductivity in 3-dimensions, J. DifferentialEquations 168 (2000), 386-452. 

%\bibitem{SeSqua05} S. Secchi, M. Squassina, On the location of spikes for the Schr\"odinger equations withelectromagnetic field, Commun. Contemp. Math. 7 (2005), 251-268.

\iffalse\bibitem{SeSqua2013frQ} Simone Secchi, Marco Squassina, 
Soliton dynamics for fractional Schrodinger equations. 
% 8 May 2013 (v1), last revised 23 May 2013 (this version, v2 

 investigate the soliton dynamics for the fractional nonlinear Schrodinger equation by a suitable modulational inequality. In the semiclassical limit, the solution concentrates along a trajectory determined by a Newtonian equation depending of the fractional diffusion parameter.  
%:	arXiv:1305.1804 [math.AP]	(or arXiv:1305.1804v2 [math.AP] for this version)
 
\href{https://doi.org/10.48550/arXiv.1305.1804}{doi}
\fi

%\bibitem{Shen96}  Z. Shen, Eigenvalue asymptotics and exponential decay of the eigenfunctions forSchr\"odinger operators with magnetic fields, Trans. Amer. Math. Soc. 348 (1996), 4465-4488.



% J. J. Garcia-Ripoll and V. M. Perez-Garcia, Phys. Rev. A 64, 013602 (2001).

%J. Colliander and P. Raphael. Rough blowup solutions to the L2 critical NLS.  Math. Ann., 345(2009):307?66.

%\bibitem{GaZ13a} L. Galati and S. Zheng, Nonlinear Schr\"odinger equations for Bose-Einstein condensates.   {\em Proceedings of  American Institute of Physics on NMMP}, 2013 %(nonlinear and modern math phys

%\bibitem{Gross61} E.  Gross, Structure of a quantized vortex in boson systems. {\em Nuovo Cimento} {\bf 20} (1961), no.3, 454-477.

\iffalse\bibitem{HaoHsiaoLi1}
C. Hao, L. Hsiao,  H. Li,
Global well posedness for the Gross-Pitaevskii equation with an angular momentum rotational term in three dimensions.
\textit{J. Math. Phys.} \textbf{48} (2007), no. 10, 102105.

\bibitem{HaoHsiaoLi2}
C. Hao, L. Hsiao,  H. Li,
Global well posedness for the Gross-Pitaevskii equation with an angular momentum rotational term.
\textit{Math. Methods Appl. Sci.} \textbf{31} (2008), 655--664.
\fi

\bibitem{HLZh21u} 
Hu, Y., Leonard, C., Zheng, S.,  Universal upper bound on the blowup rate of nonlinear Schr\"odinger equation 
with rotation. In: 
Excursions in Harmonic Analysis, Vol. 6 (2021), 59-76. 
Applied and Numerical Harmonic Analysis. 
Birkh\"auser, Cham. %\href{https://doi.org/10.1007/978-3-030-69637-5_4}{doi}

\bibitem{KT98}  M. Keel and T. Tao,  Endpoint Strichartz estimates, \textit{Amer. J. Math.} {120} (1998), 955--980.

%\bibitem{KopLan95} Kopell, N.; Landman, M., Spatial structure of the focusing singularity of the nonlinear Schr\"odinger equation: a geometrical analysis. SIAM J. Appl. Math. {\bf 55} (1995), no. 5, 1297--1323.\crr%  %exact solve \De Q_b+Q_b^{p-1}Q_b-ib\Lam Q_b =Q_b 

%\bibitem{Kw89} M. Kwong, Uniqueness of positive solutions of $\Delta u -u + u^p = 0$ in  $\mathbb{R}^n$. {\em Arch. Ration. Mech. Anal}. {\bf 105}  (1989), no.3, 243--266


\bibitem{LanPaSuSu88} M. J. Landman, G. C. Papanicolaou, C. Sulem, and P.-L. Sulem, Rate of blowup for solutions of the nonlinear Schr\"odinger equation at critical dimension, Phys. Rev. A 38 (1988), 3837--3843.  
%\href{https://journals.aps.org/pra/abstract/10.1103/PhysRevA.38.3837}{journal}

%A perturbation analysis with respect to the space dimension is used to construct singular solutions of the two-dimensional Schrdinger equation with cubic nonlinearity. These solutions blow up at a rate {ln ln[(??*-t)?1]/(??*-t)}1/2, in contrast to the behavior in three dimensions where there is no logarithmic correction. The form of such solutions is supported by the results of high-resolution numerical simulations.

%\bibitem{LaPaSuSu88} M. J. Landman, G. C. Papanicolaou, C. Sulem, and P.-L. Sulem, \href{https://journals.aps.org/pra/abstract/10.1103/PhysRevA.38.3837}{Rate of blowup for solutions of the nonlinear Schr\"odinger equation at critical dimension.} Phys. Rev. A (3) 38 (1988), no. 8, 3837--3843.

% a perturbation analysis with respect to the space dimension is used to construct singular solutions of the two-dimensional Schrodinger equation with cubic nonlinearity. These solutions blow up at a rate  $\{\ln \ln[(t^*-t)\inv]/(t^*-t)\}^{1/2}$, in contrast to the behavior in three dimensions where there is no logarithmic correction. The form of such solutions is supported by the results of high-resolution numerical simulations. 
%Received 7 March 1988\href{https://doi.org/10.1103/PhysRevA.38.3837}{DOI}    1988 American Physical Society


%\bibitem{Lan16} Lan, Y.  \href{https://doi.org/10.1007/s00220-016-2589-8}{Stable Self-Similar Blow-Up Dynamics for Slightly  $L^2$-Supercritical Generalized KDV Equations}\crrCommun. Math. Phys. (2016) 345: 223--269

\bibitem{Lan2021b} Yang Lan,  Blow-up dynamics for $L^2$-critical fractional Schr\"odinger equations. 
 International Mathematics Research Notices 2021. % Int Math Res Notices. 2021. 
% \href{https://doi.org/10.1093/imrn/rnab086}{doi}
%\href{https://ui.adsabs.harvard.edu/link_gateway/2019arXiv190809561L/arxiv:1908.09561}{harvard}\href{DO 10.1093/imrn/rnab086}{doi}  %SN 1073-7928


%  will consider the $L^2$-critical fractional Schrodinger equation $iu_t-|D|^{\beta}u+|u|^{2\beta}u=0$ with initial data $u_0\in H^{\beta/2}(\mathbb{R})$ and $\beta$ close to $2$. We will show that the solution blows up in finite time if the initial data has negative energy and slightly supercritical mass. We will also give a specific description for the blow-up dynamics. This is an extension of the work of F. Merle and P. Raphael for $L^2$-critical Schrodinger equations but the nonlocal structure of this equation and the lack of some symmetries make the analysis more complicated, hence some new strategies are required. 

%\bibitem{LePaSuSu} B. LeMesurier, G. Papanicolaou, C. Sulem, P. Sulem, Local structure of the self-focusing singularity of the nonlinear Schr\"odinger equation.  Physica D 32 (1988) 210--226. 
 
%\bibitem{LeoZ18a} C. Leonard, S. Zheng,  Sharp condition on blowup for rotational mass super-critical   NLS with an anisotroptic potential. Preprint.   

\bibitem{LeoZheng22n} C. Leonard, S. Zheng,   
Note on rotating BEC under a confining potential.    
PDE in Applied Math. Volume 6, 2022, 100461. 
% https://doi.org/10.1016/j.padiff.2022.100461. 
%  well-known,  eq \eqref{e:nls-VOm}   model  bec  ultra-cold dilute atomic gases   

%\bibitem{LeNaRo14} Lewin, Mathieu, Phan Th\`anh Nam, and Nicolas Rougerie, Derivation of Hartree's theory for generic mean-field Bose systems. Advances in Mathematics 254 (2014): 570--621.


\bibitem{LeNaRou18} M. Lewin, P. T. Nam, N. Rougerie, {Blow-up profile of rotating 2D focusing Bose gases}.
%( Feb 2018 (v1), last revised 
(2018).  \href{https://arxiv.org/abs/1802.01854}{arXiv}

% consider the Gross-Pitaevskii equation describing an attractive Bose gas trapped to a quasi 2D layer by means of a purely harmonic potential, and which rotates at a fixed speed of rotation $\Om$. First we study the behavior of the ground state when the coupling constant approaches $a_*$, 
% the critical strength of the cubic nonlinearity for the focusing nonlinear Schr{\"o}dinger equation. We prove that blow-up always happens at the center of the trap, with the blow-up profile given by the Gagliardo-Nirenberg solution. In particular, the blow-up scenario is independent of $\Om$, to leading order. This generalizes results obtained by Guo and Seiringer (Lett. Math. Phys., 2014, vol. 104, p. 141--156)  in the non-rotating case. In a second part we consider the many-particle Hamiltonian for $N$ bosons,  interacting with a potential rescaled in the mean-field manner  --$a_N N^{2\beta-1} w(N\beta x)$,with $w$ a positive function such that $\int_{\mathbb{R}^2} w(x) dx = 1$. Assuming that
% $\beta < 1/2$ and that $a_N \to a\_*$ sufficiently slowly, we prove that the many-body systemis fully condensedon the Gross-Pitaevskii ground state in the limit N \to \iy

%\bibitem{LewinSeir09} M. Lewin and R. Seiringer, Strongly correlated phases in rapidly rotating Bose gases, J. Stat. Phys.,137 (2009), pp. 1040-1062.

\bibitem{LieSei06}
E. Lieb and R. Seiringer, Derivation of the Gross-Pitaevskii equation for rotating Bose gases, Comm. Math. Phys.
264 (2006), 505-537.

%\bibitem{MAHHWC99} M.  Matthews, B.  Anderson, P. Haljan, D. Hall, C. Wiemann, E. A. Cornell, Vortices in a Bose-Einstein condensates. {\em Phys. Rev. Lett}. {\bf 83} (1999), 2498--2501.

%\bibitem{Pita61} L.  Pitaevskii, Vortex lines in an imperfect Bose gas. Zh. Eksper. Teor. Fiz. {\bf 40} (1961), 646--651.

\bibitem{Mer93} F. Merle,
Determination of blow-up solutions with minimal mass for nonlinear Schr\"odinger equations with critical power.
\textit{Duke Math. J.} \textbf{69} (1993),
no. 2, 427--454.

%\bibitem{MeRa02} F. Merle and P. Raphael, Blow up dynamic and upper bound on the blow up rate for critical nonlinear Schr\"odinger equation. In ``Journ\'ees Equations aux DÌ'eriv\'ees Partielles'' (Forges-les-Eaux, 2002), pages Exp. No. XII, 5. Univ.  Nantes, Nantes, 2002.

\bibitem{MeRa03} F. Merle and P. Raphael,  Sharp upper bound on the blow-up rate for the critical nonlinear Schr\"odinger equation. Geom. Funct. Anal., 13(2003):591--642.

%consider the critical nonlinear Schrodinger equation with initial condition $u_0$.  For $u_0\in H^1$,  local existence in time of solutions on an interval$ [0, T)$  is known, and there exist finite time blow-up solutions, that is  $u_0$ such that [inline formula]. This is the smallest power in the nonlinearity for which blow-up occurs, and is critical in this sense. The question we address is to control the blow-up rate from above for small (in a certain sense) blow-up solutions with negative energy. In a previous paper [MeR], we established some blow-up properties of (NLS) in the energy space which implied a control 
 
% and removed the rate of the known explicit blow-up solutions which is ...     ....... we prove thesharp upper bound  expected from numerics as  % \[. ......    \]  by exhibiting the exact geometrical structure of dispersion for the prob

\bibitem{MerRa05a} F. Merle, P. Rapha\"el,
Blow up dynamic and upper bound on the blow up rate for critical nonlinear Schr\"odinger equation.
\textit{Ann. of Math.} \textbf{161} (2005), no. 1, 157--222. 

\bibitem{MerRa05b}  F. Merle, P. Rapha\"el,  
Profiles and quantization of the blow up mass for critical nonlinear Schr\"odinger equation.
\textit{Comm. Math. Phys.} \textbf{253} (2005), no. 3, 675--704. 

%\bibitem{MerRa2} F. Merle and P. Rapha\"el,  On a sharp lower bound on the blow-up rate for the $L^2$-critical nonlinear Schr\"odinger equation,\textit{J. Amer. Math. Soc.} \textbf{19} (2006), no. 1, 37--90.  

%\bibitem{MerRa} Frank Merle, Pierre Raphael, Igor Rodnianski,  arXiv:1102.4308     Blow up dynamics for smooth data equivariant solutions to the energy critical Schrodinger map problem. arXiv:1106.0912 

%\bibitem{MerRaSz} Frank Merle, Pierre Raphael, Jeremie Szeftel,  %arXiv:1010.5168 
%The instability of Bourgain-Wang solutions for the L^2 critical NLS

\bibitem{MerTsu} F. Merle, Y. Tsutsumi,
$L^2$ concentration of blow-up solutions for the nonlinear Schr{\"o}dinger equation with critical power nonlinearity,
\textit{J. Differ. Equ.} \textbf{84} (1990),
205--214.


%\bibitem{MRRY2021} A. Millet, A. Rodriguez, S. Roudenko, K. Yang, Behavior of solutions to the 1D focusing stochastic nonlinear Schr\"odinger equation with spatially correlated noise. \textit{SPDE: Anal. \& Comp.} \textbf{159} (2000), no. 1, 1031--1080.  

%\bibitem{MRY2021} A. Millet, A. Rodriguez, S. Roudenko, K. Yang, Behaviour of solutions to the 1D focusing stochastic $L^2$-critical and supercritical nonlinear Schr\"odinger equation with space-time white noise. \textit{IMA Applied. Math.} \textbf{86} (2021), no. 6, 1349--1396. 


%\bibitem{Mi} L. Michel, Remarks on non-linear Schr\"odinger equation with magnetic fields,  \textit{Comm. Partial Differential Equations} \textbf{33} (2008), no. 7, 1198--1215

%\bibitem{Mo2001K^41} G. Modugno et al.,  Bose-Einstein Condensation of potassium atoms by sympathetic cooling.Science {\bf 294} (2001),1320-1322. \href{DOI:10.1126/science.1066687}{doi} 

%\bibitem{BiSow04} Iwo Bialynicki-Birula, Tomasz Sowinski (Center for Theoretical Physics PAS),\href{https://arxiv.org/abs/quant-ph/0310195}{Solutions of the Logarithmic Schrodinger Equation in a Rotating Harmonic Trap.}( Oct 2003 (v1), last revised  Sep 2004 (this version, v3))
% study the influence of the nonlinearity in the Schrodinger equation on the motion of quantum particles in a harmonic trap. In order to obtain exact analytic solutions, we have chosen the logarithmic nonlinearity. The unexpected result of our study is the existence in the presence of nonlinearity of two or even three coexisting Gaussian solutions.
	

%\bibitem{MorPenTo98} I.M. Moroz, R. Penrose, P. Tod, Spherically-symmetric solutions of the Schr\"odinger-Newton equations,Classical Quantum Gravity 15 (1998), 2733-2742.


\bibitem{MRRY2021} A. Millet, A. Rodriguez, S. Roudenko, K. Yang, Behavior of solutions to the 1D focusing stochastic nonlinear Schr\"odinger equation with spatially correlated noise.
\textit{SPDE: Anal. \& Comp.} \textbf{159} (2000), no. 1, 1031--1080. 


\bibitem{MRY2021} A. Millet, A. Rodriguez, S. Roudenko, K. Yang, Behaviour of solutions to the 1D focusing stochastic $L^2$-critical and supercritical nonlinear Schr\"odinger equation with space-time white noise.
\textit{IMA Applied. Math.} \textbf{86} (2021), no. 6, 1349--1396. 


\bibitem{OhTo09}
M. Ohta,  G. Todorova,
Remarks on global existence and blowup for damped nonlinear Schr\"odinger equations,
{Discrete Contin. Dyn. Syst.} {23}, (2009), no. 4, 1313-1325.


%\bibitem{Penrose1998} R. Penrose, Quantum computation, entanglement and state reduction,  R. Soc. Lond. Phys. Trans. Ser. A, Math. Phys. Eng. Sci. 356 (1998), 1927-1939.


%\bibitem{PaSuSuWang1991}  G. C. Papanicolaou, C. Sulem, P.-L. Sulem, and X. P. Wang, Singular solutions of the Zakharov equations for Langmuir turbulence, Phys. Fluids B 3 (1991), 969-980

\bibitem{Pere01} G. Perelman, On the blow up phenomenon for the critical nonlinear 
Schr\"odinger equation in 1D. {\em Ann. Henri Poincar\'e} {\bf 2} (2001), 605--673.  

\bibitem{Pere2022} G. Perelman,  Formation of singularities in nonlinear dispersive PDEs.  
 Collected Volume pp. 3854-3879. 
International Congress of Mathematicians
2022 July 6-14. 
%\href{https://ems.press/books/standalone/277}{EMS Press}.

%https://ems.press This contribution addresses the problem of singularity formation in nonlinear dispersive equations. Despite significant progress made in the last 20 years, for most even simplest
%canonical models our understanding of the question is far from being complete. The aim of this note is to give a selection of results and open questions illustrating the present state of the problem in the context of some basic model equations, mostly of Schrödinger type, such as the semilinear Schrdinger and Schrödinger map equations, putting an emphasison the role of solitons in the mechanisms of singularity formation. 
%\href{https://ems.press/content/book-chapter-files/33273}{chapter}


%\bibitem{PlanchonRa07e}  Fabrice Planchon and Pierre Rapha\"el,  Existence and Stability of the  log-log Blow-up Dynamics for the L 2-Critical Nonlinear Schr\"odinger Equation in a Domain.   Annales Henri Poincare 8(6):1177-1219. (2007).   \href{DOI: 10.1007/s00023-007-0332-x}{doi} 
%\href{https://link.springer.com/content/pdf/10.1007/s00023-007-0332-x.pdf}{pdf}

%Let \(iu_{t} = -\Delta u-|u|^{\frac{4}{N}}u\) be the L2-critical nonlinear Schrodinger equation, in a domain \(\Omega \subset \mathbb{R}^{N}\) with initial data in \(H^{1}_{0}(\Omega)\) (Dirichlet boundary condition) and \(N \leq 4\). We prove existence and stability of finite time blow-up dynamics with the log-log blow-up speed \({|\Delta u(t)|_{{L}^{2}}}\sim\sqrt{\frac{\rm{log|log}({\it T-t})|}{T-t}}.\) Moreover, for a suitable class of finite time blow-up solutions, we derive global rigidity properties which turn out to be modeled after the \(\mathbb{R}^{N}\) ones




\bibitem{Ra05s} P. Raphael, Stability of the log-log bound for blow up solutions to the critical nonlinear Schr\"odinger equation, Math. Ann. 331 (3) (2005), 577--609.


\bibitem{RW1999} W. Ren,  X. Wang, An iterative grid redistribution method for singular problems in multiple dimensions. 
\textit{Journal of Computational Physics} \textbf{159} (2000), no. 2, 246--273.  


%\bibitem{RW1999} W. Ren,  X. Wang, An Iterative Grid Redistribution Method for Singular Problems in Multiple Dimensions.\textit{Journal of Computational Physics} \textbf{159} (2000), no. 2, 246--273.  


\bibitem{ReZaStri01} A. Recati, F. Zambelli,  S. Stringari,  Overcritical rotation of a trapped Bose-Einstein condensate. 
{\em Phys. Rev. Lett}. {\bf 86} (2001), 377--380.  %Published 15 January 2001 

%\iffalse The rotational motion of an interacting Bose-Einstein condensate confined by a harmonic trap is investigated by solving the hydrodynamic equations of superfluids, with the irrotationality constraint for the velocity field. We point out the occurrence of an overcritical branch where the system can rotate with angular velocity larger than the oscillator frequencies. We show that in the case of isotropic trapping the system exhibits a bifurcation from an axisymmetric to a triaxial configuration, as a consequence of the interatomic forces. The dynamical stability of the rotational motion with respect to the dipole and quadrupole oscillations is explicitly discussed   \fi


\bibitem{Sa14potassium}
G. Salomon, L. Fouch\'e, S. Lepoutre, A. Aspect, and T. Bourdel,
All-optical cooling of  ${}^{39} K$    to Bose-Einstein condensation.
Phys. Rev. A 90, 033405, 2014. 
%\href{DOI: https://doi.org/10.1103/PhysRevA.90.033405}{doi} 
%Laboratoire Charles Fabry, Institut d'Optique, CNRS, Univ Paris-Sud - 2, Avenue Augustin Fresnel, F-91127 Palaiseau Cedex, France

%\bibitem{SowBi04} Tomasz Sowinski, Iwo Bialynicki-Birula, %(Center for Theoretical Physics PAS), 
% {Harmonic oscillator in a rotating trap: Complete solution in 3D}.  (2004).
%\href{https://arxiv.org/abs/quant-ph/0409070}{arxiv}

%Complete description of the classical and quantum dynamics of a particle in an anisotropic, rotating, harmonic trap is given. The problem is studied in three dimensions and no restrictions on the geometry are imposed. In the generic case, for an arbitrary orientation of the rotation axis, there are two regions of instability with different characteristics. The analysis of the quantum-mechanical problem is made simple due to a direct connection between the classical mode vectors and the quantum-mechanical wave functions. This connection is obtained via the matrix Riccati equation that governs the time evolution of squeezed states of the harmonic oscillator. It is also shown that the inclusion of gravity leads to a resonant behavior -- the particles are expelled from the trap

\bibitem{SimZw2011} 
G. Simpson, I. Zwiers; Vortex collapse for the $L^2$-critical nonlinear Schr\"odinger equation. 
J. Math. Phys. 52 (2011), no.8,  083503. 
\href{https://doi.org/10.1063/1.3608054}{doi}  

\bibitem{Squa2009} 
Marco Squassina, 
Soliton dynamics for the nonlinear Schr\"odinger equation with magnetic field. 

%The semiclassical limit of a nonlinear focusing Schrdinger equation in presence of nonconstant electric and magnetic potentials V,A is studied by taking as initial datum the ground state solution of an associated autonomous elliptic equation. The concentration curve of the solutions is a parameterization of the solutions of a Newton ODE involving the electric force as well as the magnetic force via the Lorenz law of electrodynamics	30 pages, 2 figures
%	83C50, 81Q05, 35Q40, 35Q51, 35Q55, 37K40, 37K45 \href{https://doi.org/10.48550/arXiv.0811.2584}{arxiv}


%\bibitem{Whi1904} E. T. Whittaker, A Treatise on the Analytical Dynamics of Particles and Rigid Bodies, Cambridge University Press, Cambridge, 1904 (Reprinted by Dover in 1944), p.207.



%\bibitem{Tao09} Tao, Terence,  \href{http://emis.ams.org/journals/NYJM/j/2009/15-14.pdf}{A pseudoconformal compactification of the nonlinear Schr\"odinger equation and applications.}\crr   New York J. Math 15 (2009): 265--282




%\bibitem{Tsa2011} Marios C. Tsatsos, Attractive Bose-Einstein condensates in three dimensions under rotation: Revisiting the problem of stability of the ground state in harmonic traps.  Phys. Rev. A 83, 063615, June 2011. \href{https://doi.org/10.1103/PhysRevA.83.063615}{doi}

% A. Soffer and M. I. Weinstein, Resonances, radiation damping and instability in Hamil- tonian nonlinear wave equations, Invent. Math. 136 (1999), 9–74.

%\bibitem{SuSulem99}  C. Sulem and P.-L. Sulem, The Nonlinear Schro ̈dinger Equation, Self-focusing and Wave Collapse, Applied Math. Sciences 139, Springer-Verlag, New York, 1999.

\bibitem{Wein83} M. Weinstein, Nonlinear Schr\"odinger equations and sharp interpolation estimates. {\em Comm. Math. Phys}. {\bf 87} (1983) 567--576. 

\bibitem{Wein86ss} M. Weinstein, 
On the structure and formation of singularities in solutions of nonlinear dispersive evolution equations. 
Comm. Partial Diff. Eqns., Volume 11 (1986) 545-565. %\href{https://www.columbia.edu/~miw2103/bib.html}{web}

% Dispersive decay estimates for Dirac equations with a domain wall(with J. Kraisler and A. Sagiv), SIAM J. Math. Analysis 56, 7194-7227, 2024

% Charge transfer plasmonics in bespoke graphene/alpha Cl_3 cavities(with Vitalone, Rocco; SÃ?rensen Jessen, Bjarke; Jing, Ran; Rizzo, Daniel; Xu, Suheng; Hsieh, Valerie; Cothrine, Matthew;
%Mandrus, David; Wehmeier, Lukas; Carr, G. Lawrence; Bisogni, Valentina; Dean, Cory; Hone, James;Liu, Mengkun; Fogler, Michael; Basov, Dimitri), ACS Nano, Vol 18/Issue 43 2024

% Quantum tunneling and its absence in deep wells and strong magnetic fields(with C.L. Fefferman and J. Shapiro), Proc. Nat. Acad. Sci., 12 (8) e2420062122 (2025



\bibitem{Ya91} K. Yajima, Schr\"odinger evolution equations with magnetic fields.
\textit{J. Analyse Math.} \textbf{56} (1991), 29--76. 


\bibitem{YangRouZh18}  K. Yang, S. Roudenko, Y. Zhao,
Blow-up dynamics and spectral property in the $L^2$-critical nonlinear Schr\"odinger equation in high dimensions.
\textit{Nonlinearity} \textbf{31} (2018), no. 9, 4354--4392.


%\bibitem{YangRouZh19b} K. Yang, S Roudenko, Y. Zhao, Blow-up dynamics in the mass super-critical NLS equations. Physica D: Nonlinear Phenomena, 2019.  \href{https://www.sciencedirect.com/science/article/pii/S0167278918305268}{Elsevier}

%\bibitem{YangRouZhao20s} K. Yang,  S. Roudenko and Yanxiang Zhao, Stable blow-up dynamics in the $L^2$-critical and $L^2$-supercritical generalized Hartree equation. \href{abs/2002.05830}{arxiv} (2020). 


%\bibitem{ZakhSha72} V. E. Zakharov and A. B. Shabat, Exact theory of two-dimensional self-focusing andone-dimensional self-modulation of waves in nonlinear media, Sov. Phys. JETP 34  (1972), 62--69.
 
\bibitem{Zhao2013} Zihui Zhao,  
Blow-up for the critical generalized Korteweg-de Viries equation.  {Dissertation} (2013),
 \'Ecole Normale Sup\'erieure de Paris. 
% \href{https://www.math.ens.psl.eu/shared-files/10338/?ZHAO_Zihui_Memoire_Blow-up_for_critical_gKdV.pdf} 
%(presented for the Diplome de l'ENS; advised by Yvan Martel, \'Ecole Polytechnique)  

%part I: minimal blowup for gkdv % for the Diplome de l'ENS

\bibitem{Zheng2012a}  S. Zheng, Fractional regularity for nonlinear Schr\"odinger equations with magnetic fields.
\textit{Contemp. Math.} \textbf{581} (2012), 271--285. 

\end{thebibliography}
\end{document}